\documentclass[10pt]{amsart}

\usepackage{graphicx}
\usepackage{mathrsfs}
\usepackage{oubraces}
\usepackage{float}
\usepackage{array,booktabs}
\usepackage{enumerate}
\usepackage{amsfonts}
\usepackage{amsmath}
\usepackage{amssymb}
\usepackage{hyperref}

\usepackage{latexsym}

\newcommand\s{\sigma}
\newcommand\dv{\mathfrak{v}}

\newcommand\hv{\hat{\varpi}}
\newcommand\bn{\bar{n}}

\newcommand\h{\upsilon}

\newcommand\f{{\mathfrak{f}}}
\newcommand\ep{\varepsilon}
\newcommand\vp{\varphi}
\newcommand\Th{\Theta}
\newcommand\df{\mathfrak{d}}

\newcommand\A{{\mathbb{A}}}
\newcommand\Q{{\mathbb{Q}}}
\newcommand\C{{\mathbb{C}}}
\newcommand\F{{\mathbb{F}}}
\newcommand\Z{{\mathbb{Z}}}
\newcommand\E{{\mathbb{E}}}
\newcommand\G{{\mathbb{G}}}
\newcommand\Rb{{\mathbb{R}}}

\newcommand\B{{\mathscr{B}}}

\newcommand\Ts{{\mathscr{T}}}
\newcommand\Ss{{\mathscr{S}}}

\newcommand\Rc{{\mathcal{R}}}
\newcommand\Sc{{\mathcal{S}}}

\newcommand\p{{\mathfrak{p}}}
\newcommand\OO{\mathcal{O}}
\newcommand\Gr{{\mathrm{G}}}
\newcommand\K{{\mathbb{K}}}
\newcommand\Fb{\mathbf{F}}

\newcommand\1{\mathbf{1}}%

\newcommand\GL{{\mathrm{GL}}}
\newcommand\PGL{{\mathrm{PGL}}}

\newcommand\GSp{{\mathrm{GSp}}}
\newcommand\PGSp{{\mathrm{PGSp}}}
\newcommand\Sp{{\mathrm{Sp}}}

\newcommand\GO{{\mathrm{GO}}}

\newcommand\nid{\noindent}
\newcommand\app{\approx}
\newcommand\ol{\overline}
\newcommand\ot{\otimes}
\newcommand\bt{\boxtimes}
\newcommand\op{\oplus}
\newcommand\rt{\rtimes}
\newcommand\bs{\backslash}
\newcommand\wt{\widetilde}
\newcommand\ra{\rangle}
\newcommand\la{\langle}

\newcommand\Ir{{\mathrm{Irr}}}
\newcommand\Hom{{\mathrm{Hom}}}
\newcommand\Alg{{\mathrm{Alg}}}
\newcommand\Ch{{\mathrm{Ch}}}
 
\newcommand\vol{{\mathrm{vol}}}

\numberwithin{equation}{section}

\newtheorem{thm}{Theorem}

\newtheorem{prop}[thm]{Proposition}
\newtheorem{Cor}[thm]{Corollary}
\newtheorem{lem}[thm]{Lemma}
\thanks{
The author is supported in part by JSPS Grant-in-Aid for Scientific Research No. 24740017.}
\address{Department of Mathematics, Faculty of Science, Nara Woman University,
Kitauoyahigashi-machi, Nara 630-8506, Japan. }
\email{okazaki@cc.nara-wu.ac.jp}
\begin{document}

\title{Newforms of Saito-Kurokawa lifts}

\author{Takeo Okazaki}


\maketitle

\begin{abstract}
A new- and old-form theory for Bessel periods of (cuspidal automorphic) Saito-Kurokawa representations $\pi = \ot_v \pi_v$ of $\PGSp_4$ is given.
We introduce arithmetic subgroups so that a local Bessel vector fixed by the subgroup indexed by the conductor of $\pi_v$ is unique up to scalars.
This vector is called the local newform of $\pi_v$.
The global Langlands $L$-function of a holomorphic Saito-Kurokawa representation coincides with a canonically settled Piatetski-Shapiro zeta integral of the global newform.
\end{abstract}
\tableofcontents
\section{Introduction}\label{intro}
New- and old-form theory for elliptic modular cuspforms has various applications to number theory.
In particular, the modular curve or the Shimura curve defined by the arithmetic subgroup fixing the newform can be viewed as a geometric realization of the Galois representation associated to the curve or to the Hecke eigen cuspforms fixed by the arithmetic subgroup.
This is the $\GL_2$-version of the coincidence of conductors of an ideal class character over a number field and the corresponding class field.
A next step is a newform theory for Siegel modular forms of degree two.
Roberts and Schmidt \cite{R-S} completed the local newform theory for generic representations and some other ones of $\PGSp_4$ over nonarchimedean field; the (local) arithmetic subgroup fixing the local newvector is a {\it paramodular group}.
In particular, in the generic case, the Whittaker vector fixed by the paramodular group of the representation level is unique up to scalars, and its Novodvorsky zeta integral coincides with the Langlands $L$-function (see also \cite{O}).
This is a desired result, from the view of Casselman's local Whittaker newform theory \cite{C} for $\GL_2$.

But, different from the elliptic modular case, Whittaker periods of Siegel modular forms of degree $2$ vanish, and paramodular vectors do not appear in some local representations (e.g., nongeneric supercuspidal ones), although nonarchimedean local components of automorphic holomorphic representation may be generic, and some globally generic automorphic (nonholomorphic) representations contribute to $H^{1,2}(X,\C)$ of some Siegel modular threefolds $X$.
However, any automorphic form on $\GSp_4$ always has some Bessel periods, and Piatetski-Shapiro \cite{PS2} defined zeta integrals and $L$-functions for Bessel models of local representations.
Considering these aspects of automorphic forms, and viewing recent progress in the Gross-Prasad conjecture, contributions of $\GL_2$-newform theories (including those for non-Whittaker models) to the Gross-Zagier formula, we think a more convenient model for $\GSp_4$-newform theory is the Bessel one.

As a first step of this perspective, this article treats the so-called Saito-Kurokawa lifts (or representations) $\pi = \ot_v \pi_v$ of irreducible cuspidal automorphic representation $\tau = \ot_v \tau_v$ of $\GL_2$ over a totally real field $\Fb$.
All local components $\pi_v$ are nongeneric, and possibly supercuspidal.
Although it was known by \cite{R-S} that $\pi_v$ has a one-dimensional paramodular vector space if $\pi_v$ is given by the local $\theta$-lift from the split orthogonal group of rank four, we need complement other cases, and study Piateteski-Shapiro zeta integrals of Bessel periods fixed by concrete arithmetic subgroups.
To seek convenient Bessel vectors for the zeta integrals, we need other arithmetic subgroups, indeed. 

Now we describe our main results.
Let $F = \Fb_v$ for a nonarchimedean place $v$, and abbreviate $\pi_v, \tau_v$ to $\pi, \tau$.
Let $\mathfrak{o}$ denote the ring of integers of $F$.
Let $\varpi$ and $\p = \varpi \mathfrak{o}$ denote an uniformizer and the prime ideal of $\mathfrak{o}$ respectively.
Let $q = |\mathfrak{o}/\p|$.
Let $B$ and $R$ denote $\mathrm{M}_2(F)$ and $\mathrm{M}_2(\mathfrak{o})$ respectively.
Let $\s' \in B$ be regular symmetric.
The identity component of generalized orthogonal group relevant to $\s'$ is isomorphic to the multiplicative group of the quadratic field $E = F(\sqrt{-\det(\s')})$ if $-\det(\s')$ is nonsquare (this is called the nonsplit case), and to that of $E= F + F$ otherwise (this is called the split case).
Let 
\begin{align}
J = \begin{bmatrix}
& & & -1 \\
& & -1& \\
& 1& &  \\
1& & &  \\
\end{bmatrix} \label{eq:defJ} 
\end{align} 
be the defining matrix of $\GSp_4$.
Let $\psi$ be a nontrivial additive character of $F$ such that $\psi(\mathfrak{o}) = \{1\} \neq \psi(\p^{-1})$.
Bessel vectors relevant to $\s$ and $\psi$ are $\C$-valued functions with the property 
\begin{align*}
\beta(\begin{bmatrix}
1_2&x  \\
&1_2
\end{bmatrix}g) = \psi(Tr(\s x))\beta(g), 
\end{align*}
and let $\B_\s(\pi_v)$ denote the realization of $\pi_v$ in the space of such functions (c.f. sect. \ref{sec:prep} for the precise definition of $\B_\s(\pi_v)$), where 
\begin{align*}
\s = \s'\begin{bmatrix}
 & 1\\
 1&
\end{bmatrix}. 
\end{align*} 
For $\beta \in \B_\s(\pi_v)$, define a zeta integral
\begin{align*}
Z(s,\beta) = \int_{F^\times} \beta(\begin{bmatrix}
u 1_2&  \\
&1_2
\end{bmatrix})|u|^{s-3/2} d^\times u.
\end{align*} 
From now, for the sake of simplicity, we assume that $F$ is odd residual.
Further assume that $\det(\s)$ is an invertible element of $\mathfrak{o}$, and 
\begin{align*}
\s = \begin{bmatrix}
& -\det(\s)  \\
1&
\end{bmatrix}.
\end{align*}  
Then, the subalgebra $E_\s:= F + F \s \subset B$ is isomorphic to $F+ F$ or $F(\sqrt{\det(\s)})$.
Set 
\begin{align*}
\f = \begin{cases}
1 & \mbox{in the split case}, \\
2 & \mbox{in the nonsplit case.} 
\end{cases}
\end{align*} 
Let $m$ be a nonnegative integer.
Let $\varrho$ be an element of reduced norm $\varpi^\f$ and set an order  
\begin{align*}
R_m = \mathfrak{O} + \varrho^m R,
\end{align*} 
where $\mathfrak{O}$ indicates $\mathfrak{o} \op \mathfrak{o}$(resp. the ring of integers of $E$) in the split (resp. nonsplit) case.
Set a lattice 
\begin{align*}
L_m = \varpi^{\f m} R_m^\sharp \op R_m,
\end{align*}
where $R_m^\sharp$ indicates the dual lattice of $R_m$.
Our arithmetic subgroup is the stabilizer subgroup of $L_m$ in the group $\GSp_4(F)$, and denoted by $K_{\f m}$.
Here we recognize $\GSp_4(F)$ as the subgroup of $\GL_2(B)$.
Observe that $K_0$ is $\GSp_4(\mathfrak{o})$, and that, in the split case, $K_{m}$ is conjugate to the paramodular group of level $\p^m$ defined in \cite{R-S}.
Let $\B_{m}$ denote the subspace of $K_{m}$-invariant vectors in $\B_\s(\pi)$.
Then, the idempotent $e_{K_{\f (m+1)}}$ of the Hecke algebra of $K_{\f(m+1)}$ defines a mapping $\B_{\f m} \to \B_{\f (m+1)}$, and we can consider a sequence 
\begin{align}
\B_0 \to \B_{\f} \to \B_{2\f} \to \cdots  \label{eq:seqintro}
\end{align} 
Assume that $\B_\s(\pi) \neq \{0\}$.
Then this sequence is nontrivial.
In this case, we call the first nontrivial subspace $\B_{\f m}$ and the $\f^{-1}$ multiple of the index, denoted by $M_\pi$, the minimal space and minimal level of $\B_\s(\pi)$ respectively.
Write the $\ep$-factor of the Langlands parameter $\phi_\pi$ attached to $\pi$ as 
\begin{align*}
\ep(s,\phi_\pi,\psi) = E_\pi q^{N_\pi(-s+1/2)}.
\end{align*} 
Our first main result is as follows.
\begin{thm}\label{thm:intro1}
With notations and assumptions as above, it holds that 
\begin{align*}
\dim \B_{\f M_\pi} = 1, \ \ M_\pi = \frac{N_\pi}{\f}.
\end{align*} 
If $\beta \in \B_{\f M_\pi}$ is not identically zero, then $Z(s,\beta)$ is a nonzero constant multiple of  
\begin{align}
\frac{L(s,\tau)}{1- q^{-s+1/2}} \times \begin{cases}
(1 - q^{-s-1/2} )& \mbox{in the split case}, \\
(1 + q^{-s-1/2} )& \mbox{in the nonsplit case with $\tau$ unramified}, \\
1& \mbox{in the nonsplit case with $\tau$ ramified.} 
\end{cases}
\label{eq:zetaintro}
\end{align}
\end{thm}
The integrality of $M_\pi$ means that $N_\pi$ is always even if $\B_\s(\pi) \neq \{ 0\}$ and $E_\s$ is an unramified field extension of $F$, see sect. \ref{sec:NFN} for the detail.
The unique local Bessel vector $\beta \in \B_{\f M_\pi }$ such that $\beta(1) = 1$ is called the normalized newform of $\B_\s(\pi)$, and denoted by $\beta^{new}$.
This newform has eigenvalue $E_\pi$ for an Atkin-Lehner type involution.
Further, by (\ref{eq:zetaintro}) and setting a characteristic function $\vp_{M_\pi}$ of a lattice in $\mathfrak{O}^2$ corresponding to $K_{\f M_\pi}$, we have an identity 
\begin{align*}
Z(s,\beta^{new},\vp_{M_\pi}) = L(s,\phi_\pi).
\end{align*} 
Here $Z(s,\beta,\vp)$, for various $\beta \in \B_\s(\pi)$ and Schwartz functions $\vp$ of $E^2$, indicate the Piatetski-Shapiro zeta integrals, and define the Piatetski-Shapiro $L$-factor $L(s,\pi)$ and $\ep$-factor $\ep(s,\pi,\psi)$, see sect. \ref{sec:prep}.
Our second result is 
\begin{thm}\label{thm:intro2}
For any nonarchimedean local component $\pi_v$ of a Saito-Kurokwa lift, we have identities: 
\begin{align*}
L(s,\pi_v) = L(s,\phi_{\pi_v}), \ \ \ep(s,\pi_v,\psi) = \ep(s,\phi_{\pi_v},\psi). 
\end{align*} 
\end{thm}
This article is organized as follows.
In sect. \ref{sec:prep}, we recall some known results on Saito-Kurokawa lifts briefly, and review Bessel vectors from the view point of $P_2$-theory, where $P_2$ is the mirabolic subgroup of $\GL_2$.
In sect. \ref{sec:NFsp}, we complement the paramodular newform theory of \cite{R-S} for the Bessel vectors, and, applying this, show Theorem \ref{thm:intro1} in the split case.
The novelty of this article is the nonsplit case.
For this case,  in sect. \ref{sec:NPG} we introduce nonsplit paramodular groups, and study a Hecke theory for Bessel vectors.
But, since it seems to be difficult to compute a Hecke operator for Bessel vectors fixed by the above $K_{2m}$ in general, we introduce other arithmetic subgroups, and consider a refinement (\ref{eq:secpfm}) of the sequence (\ref{eq:seqintro}).
We can compute a Hecke operator for the first nontrivial subspace of (\ref{eq:secpfm}) and lift to the subspace fixed by $K_{2m}$.
In sect. \ref{sec:NFN}, applying the Hecke theory, we give a newform theory for the nonsplit case.
An oldform theory is given in sect. \ref{sec:old}, and the injectivity of $e_{K_{2m}}$ is showed. 
Applying these nonarchimedean results,  in sect. \ref{sec:real}, we give a functional equation for special Bessel models of holomorphic discrete series representations of $\PGSp_4(\Rb)$.
By this we can complete the proof of Theorem \ref{thm:intro2} in the nonsplit case.
In sect. \ref{sec:Siegel}, we describe our results in classical terms. 

{\bf Acknowledement:}
The author would like to thank to Ralf Schmidt for suggesting the topic treated in this article.

{\bf Notation:}
\begin{itemize}
\item Throughout this article, $F$ denotes a local field of characteristic zero, with $F = \Rb$ if $F$ is archimedean.
\item 
For a nonarchimedean $F$, let $\mathfrak{o}$ and $\p = \varpi \mathfrak{o}$ denote the ring of integers and  the prime ideal respectively, where $\varpi$ is a fixed uniformizer.
In this case, let $q = |\mathfrak{o}/\p|$, and  
\begin{align*}
X = q^{-s+1/2}, X' = q^{-s-1/2}
\end{align*} 
for a complex number $s \in \C$.
\item For elements $g,h$ of a group $\mathrm{G}$, let $g\langle h \rangle = h g h^{-1}$.
\item 
If $\mathrm{G}$ is an algebraic group defined over a nonarchimedean $F$, then $\mathrm{G}(F)$ indicates  the $F$-rational points in $\mathrm{G}$, and $\Ir(\mathrm{G}(F))$ the category of irreducible admissible representations of $\mathrm{G}(F)$ up to isomorphisms.
For $\pi \in \Ir(\mathrm{G}(F))$, $w_\pi$ indicates the central character of $\pi$.
\end{itemize}
\section{Preparations}\label{sec:prep}
\subsection{Bessel vectors}\label{subsec:Bv}
Let $J$ as in (\ref{eq:defJ}) be the defining matrix for the generalized symplectic group $\GSp_4 \subset \GL_4$, and let $G= \GSp_4(F), PG = \PGSp_4(F)$.
Let $\mu$ denote the similitude factor of $G$.
We denote by $H_2$ the set of $F$-rational $2 \times 2$ Hankel matrices.
Let
\begin{align*}
N &= \{n_y := \begin{bmatrix}
1_2& y  \\
& 1_2
\end{bmatrix} \mid y \in H_2 \}, \\
\bar{N} &= \{\bn_y := \begin{bmatrix}
1_2&   \\
y & 1_2
\end{bmatrix} \mid y \in H_2 \}, \\
\hat{F}^\times &= \{\hat{u}: = \begin{bmatrix}
u 1_2 & \\
 & 1_2
\end{bmatrix} \mid u \in F^\times \}, \\
A &= \{ a_h :=
\begin{bmatrix}
 h & \\
 & h^\dag \end{bmatrix} \mid h \in \GL_2(F) \}, \\
 P &= \{\begin{bmatrix}
* &* &* &*  \\
*& *& * &* \\
& & * &*  \\
& & * &*  \\
\end{bmatrix}\} =  \hat{F}^\times A N,
\end{align*}
where 
\begin{align*}
h^\dag = \det (h) {}^th^{-1}
\langle \begin{bmatrix}
& 1 \\
1 &
\end{bmatrix}\rangle \in \GL_2(F).
\end{align*} 
If $H$ is a subgroup of $G$, then $N_H, \bar{N}_H$ and $A_H$ indicate $N \cap H, \bar{N} \cap H$ and $A \cap H$, respectively.
Let $\s \in H_2$.
Define a linear functional $l_\s$ on $H_2$ by 
\begin{align*}
y \longmapsto tr(\s y).
\end{align*} 
Let $T = T_\s$ denote the (algebraically connected) identity component of the stabilizer subgroup in the Levi part $ \hat{F}^\times A\subset P$ relevant to $l_\s$.
For a regular $\s$, set a semi-simple algebra 
\begin{align*}
E_\s = E= \begin{cases}
F + F & \mbox{ if $\det(\s) \in (F^\times)^2$ (this is the {\it split case}),} \\
F(\sqrt{\det(\s)}) & \mbox{otherwise (the {\it nonsplit case}),}
\end{cases}
\end{align*}
whose multiplicative group is isomorphic to $T$. 
Let $W$ denote the space of row vectors $E^2$.
Let 
\begin{align*}
\G_\s = \G := \{g \in \GL_2(E) \mid \det(g) \in F^\times \}
\end{align*}
act on $W$ from the right so that the symplectic form $tr_{E/F} (w_1 w'_2 -w_2w'_1), \ w,w' \in W$ is preserved up to scalars.
Therefore $\G$ is embeddable into $\GSp_{W}(F)$.
Since $\GSp_{W}(F) \simeq G$, there are embeddings $\phi_\s: \G \to G$ such that 
\begin{align}
\begin{split}
\G \cap \phi_\s^{-1}(T) &= \{\begin{bmatrix}
a & \\
& a^c
\end{bmatrix} \mid a \in E^\times \} \\
\G \cap \phi_\s^{-1}(N(H_2(\p))) & \subset \{\begin{bmatrix}
1& x\\
&1
\end{bmatrix} \mid x \in \delta_{E/F}^{-1}\mathfrak{O} \} \subset \G \cap \phi_\s^{-1}(N(H_2(\mathfrak{o}))),
\end{split}
\label{eq:embphi}
\end{align}
where $\delta_{E/F}$ indicates $1$ (resp. the the relative different of $E/F$), and $\mathfrak{O}$ indicates $\mathfrak{o} \op \mathfrak{o}$ (resp. the ring of integers) in the split (resp. nonsplit) case.
Fixing such a $\phi_\s$, we will identify $\G$ with $\phi_\s(\G) \subset G$.
There is a unique element $\imath \in G$ up to scalars such that 
\begin{align}
g \langle \imath \rangle = g^c, g \in \G \label{eq:ALel}
\end{align} 
where $c$ indicates the standard involution of $E$ over $F$.
We call $\imath$ the {\it Atkin-Lehner element}.
Let $\Lambda$ be a smooth character of $T (\simeq E^\times)$ such that $\Lambda|_{F^\times} = w_\pi$.
Let $\psi$ be a nontrivial additive character of $F$.
Define a character 
\begin{align*}
\Lambda_\s^\psi: t n_y \longmapsto \Lambda(t) \psi(l_\s(y))
\end{align*} 
of the semidirect product $TN$.
{\it Bessel functions} relevant to $\Lambda_\s^\psi$ are $\C$-valued functions $\beta$ on $G$ with the following properties.
\begin{enumerate}[i)]
\item $\beta(t n g) = \Lambda_\s^\psi(tn) \beta(g)$.
\item $\beta$ is slowly increasing if $F = \Rb$.
\item $\beta$ is smooth and $K$-finite. 
\end{enumerate}
Here $K$ is the standard maximal compact subgroup of $G$.
Now, let $F$ be nonarchimedean, and let $(\pi,V) \in \Ir(G)$.
According to Piatetski-Shapiro and Novodovorsky \cite{PS-N}, the space $\Hom(V,\Lambda_\s^\psi)$ of Bessel functionals is at most one-dimensional. 
Roberts and Schmidt \cite{R-S2} also showed it for some representations.
Letting $\hat{F}^\times TN$ act on 
\begin{align}
V(N_\G,\Lambda) := \la \pi(tn)v - \Lambda(t)v \mid v \in V, t \in T, n \in N_\G \ra \label{eqn:defVNGL}
\end{align}
naturally, we can define the twisted Jacquet module 
\begin{align*}
V_{N_\G,\Lambda} := V/ V(N_\G,\Lambda)
\end{align*}
on which $\hat{F}^\times TN/N_\G$ acts.
Let $P_2$ denote the mirabolic subgroup of $\GL_2(F)$.
The group $\hat{F}^\times TN/N_\G$ is isomorphic to $E^\times \times P_2$ via the mapping 
\begin{align*}
\hat{u} t n_y  \longmapsto (t, \begin{bmatrix}
u & l_\s(y)\\
 & 1
\end{bmatrix}).
\end{align*} 
We can regard $V_{N_\G, \Lambda}$ as a $P_2$-module.
For Whittaker models, Roberts and Schmidt \cite{R-S} used $P_3$-structure of $V$ in order to give a proof for (and modify c.f. p.82. of loc. cit) Novodovorsky's local functional equation, and to construct the newform theory.
In this article, the $P_2$-structure of $V_{N_\G,\Lambda}$ will play similar roles for Bessel models.
Every irreducible smooth $P_2$-representation is isomorphic to a representation:
\begin{align*}
\begin{bmatrix}
u & * \\
&1
\end{bmatrix} \longmapsto \chi(u),
\end{align*}
or the compactly induced representation from the representation: 
\begin{align*}
\begin{bmatrix}
 1& x  \\
& 1
\end{bmatrix} \longmapsto \psi(x) 
\end{align*} 
(c.f. \cite{B-Z}). 
Here $\chi$ is a character of $F^\times$.
For a moment we denote these $P_2$-representations by $ext(\chi)$ and $ind(\psi)$ respectively.
Let $A_2$ and $N_2$ denote the diagonals and the unipotent radicals of $P_2$ respectively.
The Jacquet-Waldspurger module $V_{N,\Lambda}$ of $V$ defined in \cite{Sc-T} is obtained by replacing $N_\G$ with $N$ in (\ref{eqn:defVNGL}).
Observe that $V_{N,\Lambda}$ is isomorphic to the Jacquet module $(V_{N_\G, \Lambda})_{N_2}$.
\begin{lem}\label{lem:N2one}
With notations as above, 
\begin{enumerate}[i)]
\item $\Hom_{N_2}(V_{N_\G,\Lambda}, \1)$ is finite dimensional. 
\item $\Hom_{N_2}(V_{N_\G,\Lambda}, \psi)$ is at most one-dimensional.
\end{enumerate}
\end{lem}
\begin{proof}
i) follows from the isomorphism $\Hom_{N_2}(V_{N_\G,\Lambda}, \1) \simeq (V_{N_\G,\Lambda})_{N_2} \simeq V_{N,\Lambda}$ and Lemma 3.3.2. of loc. cit.
ii) follows from $\Hom_{N_2}(V_{N_\G,\Lambda}, \psi) \simeq \Hom_{TN} (V,\Lambda_\s^\psi)$. 
\end{proof}
\begin{lem}\label{lem:A_2one}
Let $\chi$ be a character of $F^\times$.
\begin{enumerate}[i)]
\item $\Hom_{A_2}(ind(\psi), \chi)$ is one-dimensional.
\item If $\xi \neq \chi$, then $\Hom_{A_2}(ext(\chi), \xi)$ is zero.
If $\xi = \chi$, then $\Hom_{A_2}(ext(\chi), \xi)$ is one-dimensional space spanned by 
\begin{align*}
f \longmapsto f(1).
\end{align*}
\end{enumerate}
\end{lem}
\begin{proof}
This is by the standard distributional technique (c.f. \cite{Wa} or 2.5 of \cite{R-S}).
\end{proof}
\nid
From Lemma \ref{lem:N2one}, it follows immediately 
\begin{prop}\label{prop:JHP2}
There exists a Jordan-H\"older sequence of $P_2$-modules $0 \subset V_0 \subset \cdots \subset V_n = V_{N_\G,\Lambda}$ such that: 
\begin{align*}
V_{i} \bs V_{i+1} \simeq ext(\chi_i), \ 
V_0 \simeq
\begin{cases}
ind(\psi) & \mbox{if $\Hom(V, \Lambda_\s^\psi) \neq \{0\}$,} \\
\{0\} & \mbox{otherwise.}
\end{cases}
\end{align*}
\end{prop}
\nid 
If $\Hom(V,\Lambda_\s^\psi)$ has a nontrivial element $\lambda$, then the space of functions on $G$
\begin{align*}
\B(\pi, \Lambda_\s^\psi) := \{\beta: g \mapsto \lambda(\pi(g)v) \mid v \in V\}
\end{align*}
endowed with the actions of $G$ given by right translations, is called the {\it Bessel model} of $\pi$ relevant to $\Lambda_\s^\psi$, or to $E$ roughly.
We call $\beta \in \B(\pi, \Lambda_\s^\psi)$ {\it Bessel vectors} of $\pi$.

Now, fix a $\Lambda_\s^\psi$.
For $\beta \in \B(\pi, \Lambda_\s^\psi)$ and $\vp \in \Ss(W)$, Piatetski-Shapiro defined 
\begin{align*}
Z(s,\beta,\vp) = \int_{N_{\G}\bs \G} \beta(g) \vp([0,1] g) |\det (g)|^{s + \frac{1}{2}} dg \in \C(X).
\end{align*}
These zeta integrals consist a fractional ideal of $\C[X^\pm] = \C[X,X^{-1}]$ admitting a generator in the form of $1/P(X)$ where $P$ is a polynomial in $X$ with constant term $1$.
This generator is called Piatetski-Shapiro's $L$-function of $\pi$ (of $\B(\pi,\Lambda_\s^\psi)$ rather than $\pi$ precisely), and denoted by $L(s,\pi)$.
Similarly, the set of 
\begin{align*}
Z(s,\beta) := \int_{F^\times} \beta(\hat{u})|u|^{s-\frac{3}{2}} d^\times u 
\end{align*} 
also defines a $L$-function, which is called the regular part of $L(s,\pi)$, and denoted by $L^{reg}(s,\pi)$.
By the argument in p. 466 of \cite{Sc-T}, the ratio $L^{reg}(s,\pi)/L(s,\pi)$ is a polynomial in $X$ dividing 
\begin{align}
\begin{cases}
(1-q^{-1}X)^2 & \mbox{if $E$ is split}, \\
1-(q^{-1}X)^2& \mbox{if $E/F$ is an unramified quadratic field extension,} \\
1-(q^{-1}X) & \mbox{if $E/F$ is a ramified quadratic field extension.}
\end{cases}
\label{eq:divLL}
\end{align} 
By the work of loc. cit, $L^{reg}(s,\pi)$ in the nonspilt case  for all $\pi \in \Ir(G)$ are determined, and do not vary with the choice of the Bessel models.
Let $\pi^\vee$ denote the contragredient to $\pi$.
Define $\beta^\imath$ by 
\begin{align*}
\beta^\imath(g) =\beta(\imath g)w_\pi(\mu(g))^{-1}
\end{align*} 
so that $\beta^\imath \in \B(\pi^\vee, (\Lambda^{-1})_\s^{\psi^{-1}})$.
Define $\vp^\sharp$ by 
\begin{align*}
\vp^\sharp(z) = \int_{W} \psi(tr_{E/F}(z_1w_2 -z_2 w_1) ) \vp(w) dw
\end{align*}
where $dw$ is a self-dual measure.
The following local functional equation is given by Piatetski-Shapiro \cite{PS2} (but, the definition for $\beta^\imath$ is modified).
\begin{prop}\label{prop:FE}
Let $\pi \in \Ir(G)$ with $\B(\pi, \Lambda_\s^\psi) \neq \{0\}$.
Then, there exists a monomial $\ep(s,\pi,\psi)$ in $q^{-s}$ such that 
\begin{align}
\frac{Z(1-s,\beta^\imath,\vp^\sharp)}{L(1-s,\pi^\vee)} = \ep(s,\pi,\psi)\frac{Z(s,\beta,\vp)}{L(s,\pi)} \label{eqn:LFE}
\end{align}
for any $\beta \in \B(\pi, \Lambda_\s^\psi)$ and $\vp \in \Ss(W)$.
\end{prop}
\nid
The above $\ep$-factor satisfies $\ep(s,\pi,\psi) \ep(1-s,\pi^\vee,\psi^{-1}) = 1$.
As far as (\ref{eq:embphi}) is satisfied, $\ep(s,\pi,\psi)$ does not depend on the choice of $\imath$ if $w_\pi$ is trivial.
For $a \in F^\times$, let $\psi^a$ denote the additive character defined by $\psi^a(z) = \psi(az)$.
If $\psi$ is replaced with $\psi^a$, then $\ep(s,\pi,\psi^a) = |a|^{4s-2} \ep(s,\pi,\psi)$, but $L(s,\pi)$ does not change. 
If $w_\pi$ is trivial, then $\pi^\vee \simeq \pi$.
In this case, for $\psi$ such that $\psi(\mathfrak{o}) = \{1\} \neq \psi(\p^{-1})$, the $\ep$-factor is in the form of $\ep_\pi X^{n_\pi}$ with 
\begin{align*}
\ep_\pi \in \{ \pm 1\},\ \ n_\pi \in \Z.
\end{align*} 
These quantities $\ep_\pi$ and $n_\pi$ are called the {\it analytic root number and analytic conductor of $\pi$} respectively.
 
A proof of (\ref{eqn:LFE}) is done by showing that, except for finitely many $s \in \C$, the space $\Hom_{\G}(V \ot \Ss(W), |\det|^{-s-1/2})$ is at most one-dimensional since it contains the functionals $(\beta,\vp) \mapsto Z(s,\beta,\vp)$ and $(\beta,\vp) \mapsto Z(1-s,\beta^\imath,\vp^\sharp)$. 
By Matrigne's argument at Proposition 4.2. of \cite{M}, except for finitely many $s \in \C$, we have
\begin{align*}
\Hom_{\G}(V \ot \Ss(W), |\det|^s) \hookrightarrow \Hom_\G(\pi \ot_\C {\rm ind}_{\hat{F}^\times N_\G}^\G(\1), |\det|^s).
\end{align*} 
The last space is isomorphic to $\Hom_{A_2}(V_{N_\G,\Lambda}, |\det|^{s})$ by the Frobenius reciprocity law 2.29 of \cite{B-Z}, and to $\Hom_{A_2}(V_0, |\det|^{s})$, except for finitely many $s \in \C$ by Lemma \ref{lem:A_2one} ii).
The last space is one-dimensional by Lemma \ref{lem:A_2one} i).

The following proposition will be used in sect. \ref{sec:NPG}.
Let $pr$ denote the projection $V= \B(\pi,\Lambda_\s^\psi) \to V_{N_\G,\Lambda}$.
Consider the functional 
\begin{align*}
\lambda_\chi: \beta \longmapsto Z(s,\beta \ot \chi) := \int_{F^\times} \beta(\hat{u})|u|^{s-\frac{3}{2}} \chi(u) d^\times u
\end{align*} 
where $\chi$ is a character of $F^\times$.
Since $\lambda_\chi(\pi(tn)\beta) = \Lambda(t) \lambda(\beta), tn \in TN_\G$, there is a functional $\lambda_\chi': V_{N_\G,\Lambda} \to \C$ such that $\lambda_\chi = \lambda_\chi' \circ pr$.
\begin{prop}\label{prop:prvan}
With notations as above, assume $\B(\pi, \Lambda_\s^\psi) \neq \{0\}$.
Then we have 
\begin{enumerate}[i)]
\item Let $\beta \in pr^{-1}(V_0)$.
Then $pr(\beta) = 0$ if and only if $\beta$ vanishes on $\hat{F^\times}$. 
\item There exists an $\hat{\mathfrak{o}}^\times$-invariant $\beta \in \B(\pi, \Lambda_\s^\psi)$ such that $\beta(1) \neq 0$.
\item Assume $Z(s,\beta,\chi) = 0$ for any $\chi$. Then $pr(\beta) = 0$.
\end{enumerate}
\end{prop}
\begin{proof}
i) 
Similar to the proof of Theorem 4.3.5. of \cite{R-S}.
ii) 
Similar to the proof of Proposition 2.6.4. of loc. cit.
Consider the Bessel functional $\beta \to \beta(1)$, and define $f_0 \in V_0$ by 
\begin{align*}
f_0(\begin{bmatrix}
x & y\\
 &1
\end{bmatrix}) = \psi(x^{-1}y) \Ch(x;\mathfrak{o}^\times).
\end{align*}
There exists a desired $\beta$ in $pr^{-1}(f_0)$.
iii) 
Let $\chi_n$ be the character such that $ext(\chi_n) \simeq V_n/V_{n-1}$ in Proposition \ref{prop:JHP2}.
The functional $\lambda_{\chi_n}'$ can be regarded as a nonzero constant multiple of the functional of $V_n/V_{n-1}$ given in Lemma \ref{lem:A_2one} ii).
The assumption implies $pr(\beta) \in V_{n-1}$ by the lemma below.
Iterating such arguments, we obtain $pr(\beta) \in V_0$.
Now the assertion follows from i).
\end{proof}
\begin{lem}
Let $\chi$ be a character of $F^\times$.
Let $f \in ext(\chi)$.
Then $f$ is identically zero if and only if $f(1) = 0$.
\end{lem}
\begin{proof}
Obvious.
\end{proof}
\subsection{$\theta$-lifts}\label{sec:th}
In this section, let $B$ denote $M_{2}(F)$, or the unique (up to isomorphism) division quaternion algebra over $F$, equipped with the nondegenerate symmetric form $Tr (b_1^* b_2/2)$, where $*$ indicates the main involution.
Let $GO_B = \GO_B(F)$ denote the generalized orthogonal group of $B$, and $\mu_B$ the similitude factor on $GO_B$.
Let $H = GSO_B := \ker(\mu_B^{-2} \det) \subset GO_B$.
Letting $B^\times \times B^\times$ act on $B$ by $(g_1,g_2) \cdot b = g_1 b g_2^{*}$, we have an isomorphism $H \simeq B^\times \times B^\times/\{(z, z^{-1}) \mid z \in F^\times \}$.
This isomorphism enable us to identify any object in $\Ir(H)$ with a outer product of two certain objects in $\Ir(B^\times)$, where these objects share a same central character.
Let $Z = B^2$.
For $f \in \Ss(Z)$, let $f^\sharp$ denote the Fourier transform defined by $f^\sharp(z) = \int_{Z}\psi(Tr(z ,y)) f(y) d y$ where $d y$ is a self-dual measure.
The Weil representation $w_\psi$ of the dual pair $\mathcal{S}:= Sp_4 \times O_B$ can be realized on $\Ss(Z)$ with the following formulas: 
\begin{align}
\begin{split}
w_\psi(1,h) f(z) &= f(h^{-1}\cdot z), \ \ h \in O_B, \\
w_\psi(a_g,1) f(z) &= |\det(g)|^{-2}f(z g), \ \ g \in \GL_2(F) \\
w_\psi(n_y,1) f(z) &= \psi(Tr (\frac{1}{2} \begin{bmatrix}
& 1 \\
1&
\end{bmatrix}(z,z)y)) f(z), \ \ y \in H_2\\
w_\psi(J,1) f(z) &= f^\sharp(-z).
\end{split} \label{eq:Weil}
\end{align}
Let $\Rc= G \times H$, and $\Rc_0= \ker(\mu^{-1}\mu_B) \subset \Rc$.
Following to \cite{R2}, we extend $w_\psi$ to $\Rc_0$ by
\begin{eqnarray*}
w_\psi(g,h)f(z) = |\mu_B(h)|^{-2} w_\psi(g_1,1)f(h^{-1} \cdot z), 
\end{eqnarray*}
where 
\begin{eqnarray*}
g_1 = g\begin{bmatrix}
1_2 & \\
 & \mu(g)^{-1}1_2
\end{bmatrix}.
\end{eqnarray*}
Observe that the central elements $(u,u) \in \Rc_0$ act on $\Ss(Z)$ trivially.

Now let $F$ be nonarchimedean.
Let $\Omega = \mathrm{ind}_{\Rc_0}^\Rc w_\psi$ be the compact induction.
Define $w_\psi(\rho) = w_\psi/ \displaystyle \cap_{\lambda \in \Hom_{SO_B}(w_\psi,\rho)} \ker(\lambda)$ for $\rho \in \Ir(SO_B)$, and define $\Omega(\rho)$ for $\rho \in \Ir(H)$ similarly.
By Lemme 2. III. 4. of \cite{MVW}, there exist $\Th_\psi(\rho)$ of the category $\Alg(\Sp_4(F))$ of smooth representations of $\Sp_4(F)$, and $\Th(\rho) \in \Alg(G)$, such that 
\begin{eqnarray}
w_\psi(\rho) \simeq \Th_\psi(\rho) \bt \rho, \ \ 
\Omega(\rho) \simeq \Th(\rho) \ot \rho. \label{eqn: defThtau}
\end{eqnarray}
It is known that these big thetas $\Th_\psi(\rho)$ and $\Th(\rho)$ are admissible of finite length.
The maximal semi-simple quotients of $\Th_\psi(\rho)$ and $\Th(\rho)$ are denoted by small thetas $\theta_\psi(\rho)$ and $\theta(\rho)$ respectively.
Let $e \in F$ be a nonsquare element, and $E = F(\sqrt{e})$.
Fixing an embedding $E \to B$, we identify elements of $E$ with those of $ B$. 
Let $z_ 0 = (1,\sqrt{e}) \in Z$.
The stabilizer subgroup of $z_0$ by $SO_B$ is isomorphic to $E^\times$.
Let $\Lambda$ be a continuous character of $E^\times$ such that $\Lambda|_{F^\times} = w_\pi$.
For $\pi \in \Ir(B^\times)$, let $\Ts_\Lambda(\pi)$ denote the Waldspurger model of $\pi$ relevant to $\Lambda$, i.e., the model of $\pi$ consisting of functions $\xi$ such that 
\begin{align*}
\xi(tg) = \Lambda(t) \xi(g), t \in E^\times
\end{align*} 
endowed with the actions of $\GL_2(F)$ given by right translations.
Let $\tau^1,\tau^2 \in \Ir(B^\times)$.
Assume that they have Waldspurger models relevant to $\Lambda$.
Take $\xi_i \in \Ts_\Lambda(\tau^i)$, and set a function $\xi(h_1,h_2) = \xi_1(h_1) \xi_2(h_2)$ on $GSO_B$.
For $f \in \Ss(Z)$, define a function $\xi_f$ on $G$ by 
\begin{align}
\xi_f(g) = \int_{E^\times \bs SO_B} w_{\psi}(g, h h_g)f(z_0)\xi(hh_g) d \dot{h}, \label{eq:xivp}
\end{align} 
where $h_g \in H$ is chosen so that $\mu(g) = \mu_B(h_g)$.
This integral is independent from the choice of $h_g$, and converges since the function $h \to f(h^{-1}\cdot z_0)$ has a compact support modulo $E^\times$.
By (\ref{eq:Weil}), one can see that $\xi_f$ is a Bessel function relevant to $\Lambda^\psi_\s$, where 
\begin{align*}
\s = \begin{bmatrix}
& -e  \\
1 &
\end{bmatrix}.
\end{align*} 
Let $\Xi(\Lambda_\s^\psi)$ denote the $G$-module generated by these $\xi_f$.
\begin{lem}\label{lem:G-surj}
With notations as above, there is a surjective $G$-homomorphism
\begin{align*}
\Th((\tau^1 \bt \tau^2)^\vee) \longrightarrow \Xi(\Lambda_\s^\psi).
\end{align*} 
\end{lem}
\begin{proof}
An modification of the proof of Lemme III.4. of \cite{MVW}.
Denote $\tau = \tau^1 \bt \tau^2$, and $\Xi = \Xi(\Lambda_\s^\psi)$.
Let $Z_B (\simeq F^\times)$ denote the center of $GSO_B$.
Since the central elements $(u,u) \in \Rc_0$ act on $\Ss(Z
)$ trivially, $\xi_f$ and $\tau$ have the same central character.
By this character, and Lemma 2.9 of \cite{B-Z}, there is an irreducible admissible $SO_B$-submodule $\tau_0$ of $\tau$ and a finite set $h_0 = 1, h_1,\ldots, h_r$ of representatives for $H/ Z_B SO_B$ such that 
\begin{align*}
\tau|_{SO_B} = \op_{i=0}^r \tau_i
\end{align*} 
where $\tau_i$ denotes the right translation of $\tau_0$ by $h_i$.
For $0 \le i \le r$, let $\Xi_i$ denote the $\Sp_4$-module generated by $\xi_f$ for $\xi \in \tau_i$ and $f \in \Ss(Z)$.
Choose $g_i \in G$ so that $\mu(g_i) = \mu_B(h_i)$.
By definition, 
\begin{align*}
\Xi|_{\Sp_4} = \op_{i=0}^r \Xi_i
\end{align*} 
where $\Xi_i$ denotes the $g_i$-translation of $\Xi_0$.
Denote by $\lambda_i \in \Hom_{\Sc}(w_\psi, \Hom_\C(\tau_i,\Xi_i))$ the mapping $f \mapsto (\xi \mapsto \xi_f)$.
Take an open subgroup $K \subset SO_B$ under which $f$ is invariant.
Obviously $\lambda_i(f)(\xi) = \lambda_i(f)(e_K \xi)$.
Since $\tau_i$ is admissible, $(\tau_i^\vee)^K \simeq (\tau_i^K)^\vee$ by Lemma 2.14. of loc. cit., and we may take a finite basis of $\tau_i^K$ and its dual of $(\tau_i^K)^\vee$, say $\{ \xi^{ij} \}$ and $\{ \xi_{ij}^{*}\}$.
Set $\lambda_i' \in \Hom_{\Sc}(w_\psi, \tau_i^\vee \bt \Xi_i)$ by 
\begin{align*}
\lambda_i'(f) = \sum_j \xi_{ij}^* \ot \xi^{ij}_f.
\end{align*}
Via the natural homomorphism $\tau_i^\vee \bt \Xi_i \to \Hom_\C(\tau_i, \Xi_i)$, $\lambda_i$ factors through $\lambda_i'$.
Obviously $\lambda_i'$ factors through the surjection $w_\psi \to w_\psi(\tau_i^\vee)$.
Now, from the former isomorphism at (\ref{eqn: defThtau}), we obtain a homomorphism $\tau_i^\vee \bt \Th_\psi(\tau_i^\vee) \to \tau_i^\vee \bt \Xi_i$, and 
\begin{align}
\Th_\psi(\tau_i^\vee) \longrightarrow \Xi_i \label{eq:surTh1}
\end{align}
naturally.
Let $\xi \in \tau_i$ and $f$ be arbitrary.
If $\xi_f \neq 0$, then we may assume that $\xi$ is $K$-invariant, and write $\xi = \sum_j c_j \xi^{ij}, c_j \in \C$.
By the former isomorphism at (\ref{eqn: defThtau}), $f$ corresponds to $\sum_j \xi_{ij}^* \ot v^{ij}$ for some $v_{ij} \in \Th_\psi(\tau_i^\vee)$.
The homomorphism (\ref{eq:surTh1}) sends $\sum_j c_j v^{ij}$ to $\xi_f$, and is surjective.
Since $\Th_\psi(\tau_i^\vee)$ is admissible, each $\Xi_i$ and $\Xi$ are admissible.
Let $\lambda \in \Hom_{\Rc_0}(w_\psi, \Hom_\C(\tau,\Xi))$ denote the mapping $f \mapsto (\xi \mapsto \xi_f)$.
Similar to $\lambda_i$, $\lambda$ factors through the $\lambda' \in \Hom_{\Rc_0}(w_\psi, \tau^\vee \ot \Xi)$ defined by 
\begin{align*}
\lambda'(f) = \sum_{i,j} \xi_{ij}^* \ot \xi^{ij}_f.
\end{align*} 
Since $\tau^\vee \ot \Xi$ is $\Rc$-admissible, by Lemma \ref{lem:Gref} i) below, $
((\tau^\vee \ot \Xi)^\vee|_{\Rc_0})^\vee \simeq \tau^\vee \ot \Xi$.
By the Frobenius reciprocity, 
\begin{eqnarray*}
\Hom_{\Rc_0}(w_\psi, \tau^\vee \ot \Xi) 
&\simeq& \Hom_{\Rc_0}(w_\psi, (\tau^\vee \ot \Xi)^\vee|_{\Rc_0})^\vee) \\
&\simeq& \Hom_{\Rc}(\Omega, \tau^\vee \ot \Xi).
\end{eqnarray*}
Let $\wt{\lambda} \in \Hom_{\Rc}(\Omega, \tau^\vee \ot \Xi)$ correspond to $\lambda'$.
By Lemma \ref{lem:Gref} ii), $\lambda'(f) \in \mathrm{Im}(\wt{\lambda})$.
Similar to (\ref{eq:surTh1}), we get the desired surjection by (\ref{eqn: defThtau}) again.
\end{proof}
\begin{lem}\label{lem:Gref}
Let $\Gr$ be an $l$-group in the sense of \cite{B-Z}, and $\Gr_0$ a closed subgroup of $\Gr$.
Let $(\pi,V) \in \Alg(\Gr)$.
Assume that $\Gr$ has a system of neighbourhoods $\mathscr{N} = \{K \}$ of the identity consisting of open compact subgroups such that $V^K = V^{K \cap \Gr_0}$. 
Then 
\begin{enumerate}[i)]
\item $(\pi|_{\Gr_0})^\vee = \pi^\vee$. 
\item Let $\rho \in \Alg(\Gr_0)$ and $\lambda \in \Hom_{\Gr_0}((\Delta_{\Gr_0}/\Delta_\Gr) \rho,(\pi|_{\Gr_0})^\vee)$, where $\Delta_\Gr$ denotes the modulus of $\Gr$.
Let $\wt{\lambda} \in \Hom_\Gr(\mathrm{ind}_{\Gr_0}^\Gr \rho,\pi^\vee)$ induced by the Frobenius reciprocity.
Then $\mathrm{Im}(\lambda) \subset \mathrm{Im}(\wt{\lambda})$.
\end{enumerate}
If $\Gr_0$ is a normal subgroup of $\Gr$, then for any $\Gr_0$-admissible $(\pi,V) \in \Alg(\Gr)$, there is a system of neighborhoods as above. 
\end{lem}
\begin{proof}
i) Let $V^*$ denote the full dual of $V$.
The restriction $\pi|_{\Gr_0}$ and $\pi$ have the same dual $V^*$.
By Lemma 2.14 of loc. cit., $(V^*)^{K \cap \Gr_0} = (V^{K \cap \Gr_0})^* = (V^K)^* = (V^*)^K$ for any $K \in \mathscr{N}$.
Therefore, 
\begin{align*}
 (\pi|_{\Gr_0})^\vee = \cup_{K \in \mathscr{N}} (V^*)^{K \cap \Gr_0} = \cup_{K \in N} (V^*)^{K} = \pi^\vee.
\end{align*}
ii)
For $\xi \in (\Delta_{\Gr_0}/\Delta_\Gr)\rho$, take a $K \in \mathscr{N}$ so that $\xi$ is $K \cap \Gr_0$-invariant.
Then, $\lambda(\xi) \in (V^*)^{K \cap \Gr_0} = (V^{K \cap \Gr_0})^* = (V^K)^*$. 
By 2.29 of loc. cit., $\wt{\lambda}$ is given by 
 \begin{eqnarray*}
\la \wt{\lambda}(f), v \ra = \int_{\Gr_0 \bs \Gr} \la \lambda(f(g)), \pi(g)v \ra d g, \ \ v \in V, f \in \mathrm{ind}_{\Gr_0}^\Gr \rho
\end{eqnarray*}
where $\la, \ra$ denotes the natural pairing of $V$ and $V^*$.
Since $\xi$ is invariant under $K \cap \Gr_0$, we can define $f_K \in \mathrm{ind}_{\Gr_0}^\Gr \rho$ by $f_K(hk) = \Delta_{\Gr_0}/ \Delta_{\Gr}(h)\xi(h), h \in \Gr_0, k \in K$.
By definition, $f_K$ is $K$-invariant, and therefore $\wt{\lambda}(f_K)$ lies in $(V^*)^K = (V^K)^*$.
For $v \in V^K$, 
\begin{eqnarray*}
\la \wt{\lambda}(f_K), v \ra &=& \int_{\Gr_0 \bs \Gr_0K} \la \lambda(f_K(g)), \pi(g)v \ra d g \\
&=& \int_{\Gr_0 \bs \Gr_0K} \la \lambda(\xi), v \ra d g \\
&=& \vol(\Gr_0 \bs \Gr_0K)\la \lambda(\xi), v \ra.
\end{eqnarray*} 
Hence $\wt{\lambda}(\vol(\Gr_0 \bs \Gr_0K)^{-1}f_K) = \lambda(\xi)$.
This completes the proof of ii).
For the last assertion, let $L \subset \Gr$ be an open compact subgroup.
Fix an isomorphism $\mu: L /L \cap \Gr_0 \simeq A$ for a compact group $A$.
Since $\pi$ is $\Gr_0$-admissible, $V^{L \cap \Gr_0}$ is finite dimensional.
Therefore, there is an open subgroup $B \subset A$ such that $V^{L \cap \Gr_0} \subset V^{L_B}$ for $L_B := \{k \in L \mid \mu(k) \in B \}$.
Then, $L_B \cap \Gr_0 = \{k \in L \mid \mu(k) = 1 \} = L \cap \Gr_0$, and hence $V^{L_B \cap \Gr_0} = V^{L_B}$.
So, $\mathscr{N} := \{L_B\}$ is the desired system of neighbourhoods.
\end{proof}
\subsection{Saito-Kurokawa packet}\label{sec:packet}
Let $\Fb$ be a totally real number field, $\tau = \ot_v \tau_v$ be an irreducible cuspidal automorphic representation of $\PGL_2(\A_\Fb)$.
Let $S_\tau$ denote the set of all places $v$ at which $\tau_v$ is discrete. 
The Saito-Kurokawa packet (we will abbreviate to SK-packet) of $\tau$ is the set of irreducible cuspidal automorphic representations $\pi = \ot_v \pi_v$ of $\PGSp_4(\A_\Fb)$ whose $L$-parameters are $\{\alpha_v^\pm, |*|_v^{\pm 1/2}\}$ for almost all $v$ where $\pi_v$ is unramified.
Here $\{\alpha_v^{\pm} \}$ indicates the $L$-parameter of $\tau_v$.
By \cite{Sc}, \cite{G-T}, if $v$ is nonarchimedean, then
\begin{align*}
\pi_v = \begin{cases}
\theta(\tau_v \bt \1) \ \mbox{or} \ \theta(\tau_v^{JL} \bt \1) & \mbox{$v \in S_\tau$}, \\
\theta(\tau_v \bt \1) & \mbox{otherwise.} 
\end{cases} 
\end{align*} 
where $\tau^{JL}$ indicates the Jacquet-Langlands transfer of $\tau$.
We denote $\theta(\tau_v \bt \1)$ and $\theta(\tau_v^{JL} \bt \1)$ by $SK(\tau_v)$ and $SK(\tau_v^{JL})$, respectively.
There are other descriptions for $SK(\tau_v)$ if $\tau_v$ lies in 
\begin{align*}
\Ir'(\PGL_2(\Fb_v)) := \Ir(\PGL_2(\Fb_v)) \setminus  |*|_v^{3/2} \times |*|_v^{-3/2}. 
\end{align*} 
Let $I(\tau_v) = I_+(\tau_v)$ and $I_-(\tau_v)$ denote the representation of $G$ induced from the representations 
\begin{align*}
\begin{bmatrix}
h & * \\
& u h^\dag
\end{bmatrix} \longrightarrow \left|\frac{\det(h)}{u}\right|_v^{\pm 1/2}\tau_v(h)
\end{align*}  
respectively.
Then, $SK(\tau_v)$ is a unique irreducible nongeneric quotient of $I(\tau_v)$ and a unique nongeneric subrepresentation of $I_-(\tau_v)$.
Let $St$ denote the Steinberg representation of $\PGL_2(\Fb_v)$, and let $G(\tau_v) = \theta(\tau_v \bt St)$.
Then the following sequences are exact.
\begin{align}
\begin{split}
0 \to G(\tau_v) \to I(\tau_v)  \to SK(\tau_v) \to 0, \\
0 \to SK(\tau_v) \to I_-(\tau_v)  \to G(\tau_v) \to 0. 
\end{split}
\label{eq:exsSK}
\end{align}
If $v$ is a real archimedean place and $\tau_v$ is a holomorphic discrete series of minimal weight $2\kappa (\ge 2)$, $\pi_v$ is an irreducible constituent of a degenerate principal series, or the (limit of) holomorphic discrete series of minimal wight $(\kappa+1,\kappa+1)$ (c.f. sect. 4 of \cite{Sc}), which will be also denoted by $SK(\tau_v)$ and $SK(\tau_v^{JL})$ respectively.
The set 
\begin{align*}
\begin{cases}
\{SK(\tau_v), SK(\tau_v^{JL}) \} & \mbox{if $v \in S_\tau$}, \\
\{SK(\tau_v) \} & \mbox{otherwise} 
\end{cases}
\end{align*} 
is called the (local) SK-packet of $\tau_v$.
If all archimedean components of $\tau$ are holomorphic discrete series, then by the main lifting theorem of loc. cit, the SK-packet consists of 
\begin{align*}
\Pi(\tau \bt \pi_S): = \bigotimes_{v \in S} SK(\tau_v^{JL}) \ot \bigotimes_{v \not\in S} SK(\tau_v)
\end{align*} 
for $S \subset S_\tau$ such that $\ep(1/2,\tau) = |-1|^{|S|}$, where $S$ is possibly empty if $L(1/2,\tau) = 0$ and $\ep(1/2,\tau) = 1$.
The $L$- and $\ep$-factors of the Langlands parameter $\phi_{\pi_v}$ attached to $\pi_v = \Pi(\tau \bt \pi_S)_v$ are
\begin{align*}
L(s, \tau_v) \zeta_v(s+1/2) \times \begin{cases}
1& \mbox{if $v \in S$}, \\
\zeta_v(s-1/2) & \mbox{otherwise,} 
\end{cases}
\end{align*} 
and 
\begin{align*}
\ep(s, \tau_v,\psi_v) \times \begin{cases}
-|a_v|^{4s-2}& \mbox{if $v \in S$ is archimedean}, \\
- q_v^{(-4l_v-1)(s-1/2)} & \mbox{if $v \in S$ is nonarchimedean,} \\
1 & \mbox{otherwise.} 
\end{cases}
\end{align*} 
Here $\zeta_v$ indicates the $v$-factor of the complete Dedekind zeta function of $\Fb$, $a_v$ is the real number such that $\psi_v(x) = \exp(2\pi \sqrt{-1}a_v x)$, and $l_v$ is the integer such that $\psi_v(\p^{-l_v}) = \{1\}$ and $\psi_v(\p^{-l_v-1}) \neq \{1\}$.
Defining $L(s,\phi_\pi) = \prod_v L(s,\phi_{\pi_v})$ and $\ep(s,\phi_\pi) = \prod_v \ep(s,\phi_{\pi_v},\psi_v)$, we have a global functional equation:
\begin{align}
\ep(s,\phi_\pi)L(1-s,\phi_\pi) = L(s,\phi_\pi). \label{eq:GlFE}
\end{align}

Now let $F$ be a nonarchimedean local field.
The following result due to Roberts and Schmidt is fundamental.
\begin{thm}[\cite{R-S2}]\label{thm:1dimSK}
Let $\tau \in \Ir'(\PGL_2(F))$, and $\pi$ be in the SK-packet of $\tau$.
\begin{enumerate}[i)]
\item In the nonsplit case, if $\B(\pi,\Lambda_\s^\psi) \neq 0$, then $\Lambda = \1$.
\item In the split case, if $\B(\pi, \Lambda_\s^\psi) \neq \{ 0\}$, then $\Lambda = \1 \ \mbox{and}\ \pi = SK(\tau)$.
\end{enumerate}
\end{thm}
In the remainder of this article, we will treat only special Bessel models mainly, and let 
\begin{align*}
\B_{\s}(\Pi) = \B(\Pi, \1_\s^\psi)
\end{align*} 
for a regular $\s \in H_2$, a general $\Pi \in \Ir(PG)$ and a fixed $\psi$.

Now let $\tau \in \Ir'(\PGL_2(F))$, and $\pi = SK(\tau)$.
We want to show the next theorem.
\begin{thm}\label{thm:LSK}
It holds that
\begin{align*}
L(s,\pi) = L(s,\phi_\pi), \ \ \ep(s,\pi,\psi) = \ep(s,\phi_\pi,\psi)
\end{align*} 
for any (split or nonsplit) Bessel model of $\pi = SK(\tau)$.
\end{thm}
\nid
By Theorem \ref{thm:1dimSK}, it suffices to consider the special Bessel models of $\pi$.
Assume that $\B_\s(\pi) \neq \{0\}$ for a regular $\s$.
Let 
\begin{align}
\Xi = \Xi(\1_\s^\psi) \ \mbox{(resp. $\Xi' = \Xi(\1_\s^\psi)$} \label{eq:defXi}
\end{align} 
be the $G$-module generated by $\xi_f$ (c.f. (\ref{eq:xivp})) where $\xi \in \tau \bt \1$ (resp. $\xi \in \tau \bt St$).
Since $\pi = \theta(\tau \bt \1)$(resp. $G(\tau) = \theta(\tau \bt St)$), there is a surjection from the Siegel induction $I(\tau)$ (resp. $I_-(\tau)$) to the big theta $\Th(\tau \bt \1)$ (resp. $\Th(\tau \bt St)$) by the proof of Theorem 8.2 of \cite{G-T}.
By Lemma \ref{lem:G-surj}, we have 
\begin{align}
I(\tau) \twoheadrightarrow \Xi, \ \ I_-(\tau) \twoheadrightarrow \Xi'. \label{eq:srjsXiXi'}
\end{align} 
In order to show the theorem we need the following lemma.
\begin{lem}\label{lem:gammacoinSK}
With notations as above, we have the followings:
\begin{enumerate}[i)]
\item $SK(\tau) \simeq \Xi$.
\item$\gamma(s,SK(\tau),\psi) = \gamma(s,\tau,\psi) \gamma(s,\1_{GL(2)},\psi)$, where the $\gamma$-factors are defined as usual.
\end{enumerate}
\end{lem}
Suppose that $E_\s$ is in the nonsplit case.
In this case, the proofs for the lemma and the theorem are as follows.
By the table of Theorem 6.2.2. of \cite{R-S2}, $G(\tau)$ has no special Bessel model relevant to $\s$ and contained in the kernel of the former surjection of (\ref{eq:srjsXiXi'}).
Now i) of the lemma is obvious by (\ref{eq:exsSK}), and ii) follows from i) and the computation of Piatetski-Shapiro and Soudry \cite{PS-S}.
According to Table 5 of \cite{Sc-T}, $L^{reg}(s,\pi)^{-1}= (1-X)L(s,\tau)^{-1}$.
Put $P(X) = L^{reg}(s,\pi)/L(s,\pi)$.
The $\gamma$-factor of $\pi$ is, by Lemma \ref{lem:gammacoinSK}, 
\begin{align*} 
& \ep(s,\pi,\psi)\frac{P(X)}{P(X^{-1})} \frac{L^{reg}(1-s,\pi)}{L^{reg}(s,\pi)} \\
&= -\ep(s,\pi,\psi)X \frac{P(X)}{P(X^{-1})} \frac{L(1-s,\tau)}{L(s,\tau)}
\end{align*} 
and that of $\phi_\pi$ is
\begin{align*}
-\ep(s,\tau,\psi)X \frac{(1-q^{-1}X)}{(1-q^{-1}X^{-1})} \frac{L(1-s,\tau)}{L(s,\tau)}.
\end{align*}
Therefore,
\begin{align*}
\frac{P(X)}{P(X^{-1})} = \frac{\ep(s,\tau,\psi)}{\ep(s,\pi,\psi)}\frac{1-q^{-1}X}{1-q^{-1}X^{-1}}.
\end{align*}
Taking (\ref{eq:divLL}) into account, and looking the location of the poles of both sides, we conclude that $\ep(s,\pi,\psi) = \ep(s,\tau,\psi)$ and $P(X) = (1-q^{-1}X)$.
This completes the proof of the theorem in the nonsplit case.
Those for the split case will be given in the next section.
\section{Local newform (split case)}\label{sec:NFsp}
In this section, let $F$ be nonarchimedean, and 
\begin{align*}
\s = \begin{bmatrix}
1&  \\
&1
\end{bmatrix}.
\end{align*} 
The corresponding algebra $E_\s = E$ is split.
In this case, we can recognize $\G_\s = \G$ as a group $\{ (g_1, g_2) \in \GL_2(F) \times \GL_2(F) \mid \det(g_1) =\det(g_2) \}$, and define the embedding $\phi_\s$ into $G$ by 
\begin{align*}
(\begin{bmatrix}
a_1 & b_1  \\
c_1 & d_1
\end{bmatrix},  \begin{bmatrix}
a_2 & b_2  \\
c_2 & d_2
\end{bmatrix}) \longmapsto \begin{bmatrix}
a_1& & &b_1  \\
& a_2& b_2 & \\
& c_2 &d_2 &  \\
c_1& & & d_1 \\
\end{bmatrix}.
\end{align*} 
We choose an Atkin-Lehner element (c.f. \ref{eq:ALel}) 
\begin{align*}
\imath = \begin{bmatrix}
& 1& &  \\
1& & & \\
& & &1  \\
& &1 &  \\
\end{bmatrix}.
\end{align*} 
Let 
\begin{align}
s = \begin{bmatrix}
1& & &  \\
& &-1 & \\
& 1& &  \\
& & &1  \\
\end{bmatrix}, 
\end{align} 
and 
\begin{align*}
n_x' = \begin{bmatrix}
1& & &  \\
&1 &x & \\
& &1 &  \\
& & &1  \\
\end{bmatrix}, \ \bn_x' = \begin{bmatrix}
1& & &  \\
&1 & & \\
& x&1 &  \\
& & &1  \\
\end{bmatrix},\ \ x \in F.
\end{align*} 
Let $\tau \in \Ir'(\PGL_2(F))$, and $\pi =SK(\tau)$.
Let $\psi$ be an additive character on $F$ such that $\psi(\mathfrak{o}) = \{1\} \neq \psi(\p^{-1})$.
We will construct a paramodular Bessel vector $\beta \in \B_\s(\pi)$ and compute its zetas.
Let $n_\tau$ and $\ep_\tau (\in \{ \pm 1\})$ be the conductor and root number of $\tau$, respectively.
It is known by \cite{C} that there exists a Whittaker function $\omega$ with respect to $\psi$ such that 
\begin{align*}
L(s,\tau) = \sum_{i = 0}^\infty X^{i} \omega(\begin{bmatrix}
\varpi^{i}&  \\
& 1
\end{bmatrix}).
\end{align*} 
Recall the definition of the paramodular group of level $n$.
It consists of elements $k$ with $\mu(k) \in \mathfrak{o}^\times$ in the set 
\begin{align*}
\begin{bmatrix}
\mathfrak{o} &\mathfrak{o} & \mathfrak{o}& \p^{-n}\\
\p^n &\mathfrak{o} &\mathfrak{o} &\mathfrak{o} \\
\p^n &\mathfrak{o} &\mathfrak{o} &\mathfrak{o} \\
\p^n & \p^n & \p^n & \mathfrak{o}
\end{bmatrix}.
\end{align*}  
Let $K$ be the paramodular group of level $n_\tau$.
By the results in 5.5. of \cite{R-S}, there is a unique (up to scalars) $K$-invariant $\delta \in \pi \subset I_-(\tau)$ (c.f. \ref{eq:exsSK}) defined by 
\begin{align*}
\delta(g) = 
\begin{cases}
|\det(h)u^{-1}|\omega
(h) & \mbox{if $g \in \begin{bmatrix}
h &  *\\
& u h^\dag
\end{bmatrix}K$}, \\
0 & \mbox{otherwise}.
\end{cases}
\end{align*}
It has the property 
\begin{align*}
\pi(\begin{bmatrix}
& 1& &  \\
\varpi^{n_\tau} & & & \\
& & &1  \\
& &\varpi^{n_\tau} &  \\
\end{bmatrix}) \delta = \ep_\tau \delta.
\end{align*}
We set
\begin{align*}
\beta(g) = \int_F \delta(sn_x' g) dx.
\end{align*}
It is easy to see that this integral converges, and $\beta$ is a special split Bessel vector invariant under $K$.
Let $\check{\varpi} = \hv\la s\ra$.
Since $\beta$ is special, we have
\begin{align*}
\beta(\imath \hat{\varpi}^i) & = \beta(\check{\varpi}^{n_\tau} \hat{\varpi}^i \imath) = \beta(\hat{\varpi}^i \imath_{n_\tau}) = \ep_\tau \beta(\hat{\varpi}^i)
\end{align*}
by the property of $\delta$.
Since $K$ contains the element $s$, we have  
\begin{align*}
\beta(\hat{\varpi}^i) &= \int_F \delta(\bar{n}'_x\check{\varpi}^is) dx = \int_F\delta(\check{\varpi}^i \bn'_{\varpi^{-i}x} ) dx = q^{-i} \int_F \delta(\check{\varpi}^i \bn'_x) dx
\end{align*} 
by definition of $\beta$.
The last integral is 
\begin{align*}
& \delta(\check{\varpi}^i) + \sum_{j = 1}^{\infty} \int_{x \in \mathfrak{o}^\times} \delta(\check{\varpi}^i \bn'_{\varpi^{-j}x}) dx \\
&=  \delta(\check{\varpi}^i) +  \sum_{j = 1}^{i}q^{j-1}(q-1) \delta( \hv^j \check{\varpi}^{i-j})
\end{align*} 
by the $K$-invariance property of $\delta$, and the identity
\begin{align}
\begin{bmatrix}
1_n&  \\
x &1_n
\end{bmatrix} = \begin{bmatrix}
1_n& x^{-1} \\
 &1_n
\end{bmatrix}
\begin{bmatrix}
& -x^{-1} \\
x&
\end{bmatrix}
 \begin{bmatrix}
1_n& x^{-1} \\
 &1_n
\end{bmatrix}, \ \ x \in \GL_n(F).
\label{eq:ufid}
\end{align} 
Therefore, 
\begin{align*}
\int_F \delta(\check{\varpi}^i \bn'_x) dx = \omega
(\begin{bmatrix}
\varpi^{i}&  \\
& 1
\end{bmatrix}) +  \sum_{j = 1}^{i}q^{j-1}(q-1)|\varpi|^j \omega
(\begin{bmatrix}
\varpi^{i-j}&  \\
& 1
\end{bmatrix}). 
\end{align*}
Hence, 
\begin{align}
Z(s,\beta) &= \sum_{i=0}^\infty \omega
(\begin{bmatrix}
\varpi^{i}&  \\
& 1
\end{bmatrix}) (X^i +  (1-q^{-1})\sum_{j = i+1}^\infty X^j ) \notag\\
& = \sum_{i=0}^\infty \omega
(\begin{bmatrix}
\varpi^{i}&  \\
& 1
\end{bmatrix}) X^{i} (1 +  (1-q^{-1})\sum_{j = 1}^\infty X^j) \notag\\
&= L(s,\tau)\frac{1-X'}{1-X}. \label{eq:Zregsp}
\end{align}  
Let $\vp_n \in \Ss(F^4)$ be the characteristic function of the lattice $\p^n \op \mathfrak{o} \op \mathfrak{o} \op \mathfrak{o}$, which is invariant under $\phi_\s^{-1}(K)$, a maximal compact subgroup of $\G$.
Observe that $\vp_n^\sharp$ is the characteristic function of $\mathfrak{o} \op \mathfrak{o} \op \mathfrak{o} \op \p^{-n}$, and invariant under the same subgroup of course.
From the proof of Lemma 5.3.2. \cite{Sc-T}, it follows that 
\begin{align*}
Z(s,\beta, \vp_{n_\tau}) &= \frac{Z(s,\beta)}{(1-X')^{2}}, \\
Z(s,\beta^\imath, \vp_{n_\tau}^\sharp) &= \ep_\tau X^{-n_\tau} \frac{Z(s,\beta)}{(1-X')^{2}}.
\end{align*}
\begin{thm}\label{label:Lsp}
With notations as above, $\pi =SK(\tau)$ has a unique (up to scalars) nontrivial a special split Bessel vector invariant under the paramodular group of level $n_\tau$.
If $\beta$ is such a Bessel vector, then  
\begin{align*}
\frac{Z(s,\beta,\vp_{n_\tau})}{L(s,\phi_\pi)} &= \ep_\tau X^{-n_\tau} \frac{Z(1-s,\beta^\imath,\vp_{n_{\tau}}^\sharp)}{L(1-s,\phi_\pi)} \in \C^\times.
\end{align*}
\end{thm}
\begin{proof}
The uniqueness is proved in 5.5. of \cite{R-S}.
The last statement is an immediate consequence of this and the above computation.
\end{proof}
In the remainder of this section, we will devote to prove Lemma \ref{lem:gammacoinSK} and Theorem \ref{thm:LSK} in the split case.
Assume that $SK(\tau) \subset I_-(\tau)$ is not contained in the kernel of the latter surjection of (\ref{eq:srjsXiXi'}).
Then, $\Xi'$ defined at (\ref{eq:defXi}) has an irreducible submodule isomorphic to $SK(\tau)$.
But, by the computation of Piatetski-Shapiro and Soudry \cite{PS-S}, it holds that  
\begin{align*}
Z(1-s,\beta^\imath, \vp^\sharp) = \gamma(s,\tau,\psi)\gamma(s,St,\psi) Z(s,\beta,\vp)
\end{align*} 
for any $\vp \in \Ss(F^4)$ and $\beta \in \Xi'$.
This conflicts to Theorem \ref{label:Lsp}. 
Hence, for the split special Bessel model of generic irreducible quotient $G(\tau)$ of $I_-(\tau)$, it holds that
\begin{align*}
G(\tau) \simeq \Xi'
\end{align*} 
and, by Proposition \ref{prop:FE}, that
\begin{align*}
\gamma(s,G(\tau),\psi) = \gamma(s,\tau,\psi) \gamma(s,St,\psi).
\end{align*}
Taking the former surjection of (\ref{eq:srjsXiXi'}) into account, we obtain the lemma.
Then by the argument for the nonsplic case, the theorem is reduced to show that $L^{reg}(s,\pi)^{-1}$ equals 
\begin{align}
(1-X) \times
\begin{cases}
1 & \mbox{if $\tau =St$,} \\
L(s,\tau)^{-1} & \mbox{otherwise}
\end{cases}
\label{eq:estSKX}
\end{align}  
in the split case. 

Now we will show the above equality.
Let $(\pi,V) \in \Ir(PG)$ be a general representation having a split special Bessel model.
Assume $\pi$ is nongeneric for the sake of simplicity.
Let $\beta \in \B_\s(\pi)$.
It holds that 
\begin{align}
Z(s,\pi(\begin{bmatrix}
a b & & & * \\
 & a c& * & 
 \\
& & b& \\
& & & c
\end{bmatrix}
)\beta) = |a|^{-s + \frac{3}{2}}Z(s,\beta). \label{eqn:regfunc}
\end{align}
In particular, the functional $\beta \mapsto Z(s,\beta)$ is invariant under the center $Z^J$ of the Jacobi subgroup of the standard Klingen subgroup $Q \subset G$.
Therefore, we can apply Roberts and Schmidt's $P_3 (\simeq Q/Z^J)$-technique to analyze $Z(s,\beta)$ (c.f. p. 130-135 of \cite{R-S}).
Consider the principal part of the Laurent expansion 
\begin{align*}
Z(s,\beta) = \frac{\lambda_{n_j}^{j}(\beta)}{(s-s_j)^{n_j}} + \cdots + \frac{\lambda_{1}^{j}(\beta)}{(s-s_j)} + \mbox{(holomorphic part)}
\end{align*}
for each pole $s_j$.
Put $\alpha_j = \exp(s_j)$.
\begin{prop}
Each functional $\lambda_i^j$ induces a linear functional $\mu_i^j$ on the $P_3$-module $V$ such that
\begin{align}
\mu_i^j(\begin{bmatrix}
u &*  &  \\
 & *&  \\
 & & 1
\end{bmatrix}\beta) &= \mu_i^j(\beta), \ \ \ u \in \mathfrak{o}^\times \label{eq:muij}, \\
\mu_i^j(\begin{bmatrix}
\varpi & &  \\
 & 1&  \\
 & & 1
\end{bmatrix}\beta) &= q^{-3/2} \alpha_j \mu_i^j(\beta) + \sum_{k = i +1}^{n_j} c^i_k \mu_k^j(\beta). \label{eqn:lami'}
\end{align}
Here $\alpha_j$ and $c_k^i \neq 0$ are constants.
\end{prop}
\begin{proof}
By (\ref{eqn:regfunc}), the isomorphism $P_3 \simeq Q/Z^J$ given in Lemma 2.3.1 of loc. cit, and the Taylor expansion $q^{s} = \sum_{i = 0}^\infty \alpha_j \frac{(\log q)^i}{i !}(s-s_j)$.
\end{proof}
\nid
By Theorem 2.5.3. of loc. cit, nongeneric $V$ has a Jordan-H\"older sequence of $P_3$-modules $\{0\} = U_0 \subset \cdots \subset U_L = V$.
Here each $U_m \bs U_{m+1}$ is isomorphic to a representation
\begin{align*} 
&\begin{bmatrix}
g & * \\
 & 1
\end{bmatrix} \longmapsto \rho(g), \ g \in \GL_2(F)
\end{align*} 
or the compact induction from a representation
\begin{align*}
&\begin{bmatrix}
t &* & * \\
 & 1 & x \\
 & & 1
\end{bmatrix} \longmapsto \chi(t) \psi(x), \  x \in F, t \in F^\times
\end{align*}
where $\chi \in \Ir(\GL_1(F))$ and $\rho \in \Ir(\GL_2(F))$.
These representations of $P_3$ will be denoted by $ext(\rho)$ and $ind(\chi)$ respectively. 
Let $B_2'$ denote the subgroup of $P_3$ consisting of matrices in (\ref{eq:muij}), and let $B_2'$ act on $\C_{B_2'} =\C$ trivially. 
Let $\hat{a}$ denote also $diag(a,1,1) \in P_3$ for a moment.
\begin{lem}
If $\chi$ is ramified, then $\Hom(ind(\chi), \C_{B_2'}) = \{0\}$.
If $\chi$ is trivial, then $\Hom(ind(\chi), \C_{B_2'})$ is infinite dimensional.
If $\chi$ is unramified nontrivial, then $\Hom(ind(\chi), \C_{B_2'})$ is one dimensional, and it holds that
\begin{align*}
\mu(\hat{a} \cdot f) =\chi(a)\mu(f),  \ a \in F^\times.
\end{align*} 
Here $\mu \in \Hom(ind(\chi), \C_{B_2'})$ and $f \in ind(\chi)$ are arbitrary.
\end{lem}
\begin{proof}
This is proved by the standard distributional technique (c.f. \cite{Wa}).
Consider the property (\ref{eq:muij}), and the double coset space:
\begin{align*}
H \bs P_3 /\{\begin{bmatrix}
1 &*  &  \\
 & *&  \\
 & & 1
\end{bmatrix}\},
\end{align*} 
where $H$ is the subgroup on which the inducing representation is defined.
A realization of this space is 
\begin{align*} 
\{ 1 \} \sqcup \{ \begin{bmatrix}
 &1 &  \\
* &  &  \\
 & & 1
\end{bmatrix} \}.
\end{align*}
The support of the distribution corresponding to $\mu$ is only the orbit of $1$.
\end{proof}
\begin{lem}
Let $f \in ext(\rho)$, $a \in F^\times$ and $\mu \in \Hom(ext(\rho), \C_{B_2'})$ be arbitrary.
Then we have the followings.  
\begin{enumerate}[i)]
\item Let $\rho = \chi_+ \times \chi_-$.
If $\chi_\pm$ is unramified and $\chi_\mp = |*|^{1/2}$, then $\dim \Hom(ext(\rho), \C_{B_2'}) = 1$ and $
\mu(\hat{a}\cdot f) = \chi_\pm(a)|a|^{1/2}\mu(f)$.  
Otherwise, $\Hom(ext(\rho), \C_{B_2'})$ is zero.
\item Let $\rho=\chi St$. 
If $\chi$ is trivial, then $\dim \Hom(ext(\rho), \C_{B_2'}) = 2$.
Otherwise, $\Hom(ext(\rho), \C_{B_2'})$ is zero.
\item Let $\rho = \chi \circ \det$.
If $\chi$ is trivial, then $\dim \Hom(ext(\rho), \C_{B_2'}) =1$, and $
\mu(\hat{a} \cdot f) =\mu(f)$.
Otherwise, $\Hom(ext(\rho), \C_{B_2'})$ is zero.
\item If $\rho$ is supercuspidal, then $\Hom(ext(\rho), \C_{B_2'})$ is zero.
\end{enumerate}
\end{lem}
\begin{proof}
Similar to the previous lemma.
\end{proof}
Now fix $j$.
Let $m_i$ be the unique integer such that $\lambda_i^j(U_{m_i}) = 0$ and $\lambda_i^j(U_{m_i+1}) \neq 0$.  
Assume that $\Hom (U_{m_i} \bs U_{m_i+1}, \C_{B_2'}) = 1$ as in the above lemmas.
Let $\gamma_i$ be the constant such that $\mu(\hv f) = \gamma_i \mu(f)$.
Let $\beta \in U_{m_i+1}$.
By (\ref{eqn:lami'}),
\begin{align*}
\gamma_i^l(\gamma_i -  q^{-3/2} \alpha_{j} )\mu_{i}^j(\beta) = \sum_{k = i +1}^{n_j} c^i_k \mu_k^j(\pi(\hv^l)\beta), \ \ l\in \Z.
\end{align*} 
Now, it follows that $\gamma_i = q^{-3/2} \alpha_{j}$, and that $\mu_k^{j}(\beta) = 0$ if $i+ 1 \le k \le n_j$ (use the induction on $i$, starting from $i = n_j$ to $1$).
In particular, $\gamma_1 = \cdots = \gamma_{n_j}$.
From this argument, one can deduce: 
\begin{prop}
Let $(\pi,V) \in \Ir(G)$ be nongeneric.
Assume that for the $P_3$-filtration $\{0\} = U_0 \subset \cdots \subset U_L = V$, there is no quotient $U_i \bs U_{i+1}$ isomorphic to $ext(St)$ or $ind(\chi)$ with $\chi = \1$.
Then, as polynomials in $X$, 
\begin{align*}
L^{reg}(s,\pi)^{-1} | \prod_{0 \le i \le L-1}(1 - \gamma_i q X).
\end{align*} 
Here $\gamma_i$ is the constant such that $\mu(\hv f) = \gamma_i \mu(f)$ for any $\mu \in \Hom(U_i \bs U_{i+1},\C_{B_2'})$ and $f \in U_i \bs U_{i+1}$ if  $\dim \Hom(U_i \bs U_{i+1},\C_{B_2'}) = 1$, and $0$ otherwise.
\end{prop}
\nid
Viewing Table A.5., A.6. of \cite{R-S}, one can find that $L^{reg}(s,\pi)^{-1} \in \C[X]$ for $\pi =SK(\tau)$ divides (\ref{eq:estSKX}).
However, at (\ref{eq:Zregsp}), we have seen that (\ref{eq:estSKX}) is attained by $1/Z(s,\beta)$ if $\tau = St$ and $(1-X')/Z(s,\beta)$ otherwise.
Noting that $\deg L(s,\tau)^{-1} \le 2$, and that $L(s,\tau)^{-1} \neq (1- X)(1-X')$ if $\deg L(s,\tau)^{-1} = 2$, and $L(s,\tau)^{-1} = (1\pm X')$ if $\deg L(s,\tau) = 1$, we conclude that $L^{reg}(s,\pi)^{-1}$ equals (\ref{eq:estSKX}).
\section{Nonsplit paramodular groups}\label{sec:NPG}
To define the nonsplit paramodular groups in the next subsection, we need some orders of $M_2(F)$.
Let $e \in \mathfrak{o}$ be a non-square element, and $E = F(\sqrt{e})$.
Let $\F$ and $\E$ indicate the residual field of $F$ and $E$ respectively.
Let $\f$ denote the degree of the field extension $\E/\F$.
The quadratic extension $E$ of $F$ falls into the following cases.
\begin{description}
\item[Case U-i)] $\f = 2$, and $F$ is nondyadic. 
\item[Case U-ii)] $\f = 2$, and $F$ is dyadic. 
\item[Case R-i)] $e$ lies in $\p$, $\f = 1$. 
\item[Case R-ii)] $e$ lies in $\mathfrak{o}^\times$, $\f =1$. 
\end{description}
In the case R-ii), $F$ is always dyadic.
In the case U-ii), there is an element $b \in \mathfrak{o}$ such that $1 - b^2 e \in 4 \mathfrak{o}$.
Except for the case R-i), $e$ lies in $\mathfrak{o}^\times$.

Let 
\begin{align*}
\dv = \begin{cases}
2& \mbox{in the case U-ii)}, \\
1& \mbox{otherwise.} 
\end{cases}
\end{align*} 
Let
\begin{align}
\s = \begin{bmatrix}
 & e \\
1 &
\end{bmatrix}, \iota = \begin{bmatrix}
1&  \\
&-1
\end{bmatrix}, 
\h = \begin{bmatrix}
 & 1\\
 &
\end{bmatrix}, \label{eq:defsigma}
\end{align}
where $e$ is a nonsquare element of $\mathfrak{o}$.
The subalgebra $F + F \s \subset M_2(F)$ is isomorphic to $E$, and we will identify them.
Let $*$ indicate the main involution of $M_2(F)$.
The Galois conjugate of $x$ equals $x^* = x\langle \iota \rangle$ if $x \in E$.
Let $\mathfrak{O}$ denote the ring of integers of $E$.
Explicitly, 
\begin{align*}
\mathfrak{O} = 
\begin{cases}
\mathfrak{o} \op \mathfrak{o} \frac{1+ b \s}{2} & \mbox{in the case U-ii)}, \\
\mathfrak{o} \op \mathfrak{o} \s & \mbox{otherwise.} 
\end{cases}
\end{align*}
Except for the case U-ii), $\mathfrak{O}$ coincides with $\mathfrak{o} \op \mathfrak{o} \s$.
Let $\mathfrak{P}$ denote the prime ideal of $\mathfrak{O}$.
Fix an element of $\alpha \in \mathfrak{o}^\times$ chosen so that $\alpha^2 -e = \varpi$ in the case R-ii).
We fix the generator $\varrho$ of $\mathfrak{P}$ as follows:
\begin{align*}
\varrho = \begin{cases}
\varpi & \mbox{in the case U}, \\
 \s & \mbox{in the case R-i)}, \\
\alpha + \s & \mbox{in the case R-ii).}
\end{cases} 
\end{align*}
Since $M_2(F) = E + E \h$, we have a decomposition 
\begin{align} 
\GL_2(F) = P_2 E^\times = E^\times P_2. \label{eq:P2T}
\end{align}
The following identities are useful.
\begin{align}
\begin{split}
\h^2 = 0, \  \h \s \h = \h, \ \s \h \s = \s -e \h, \\
\h t - t^* \h =t\h  - \h t^* = \frac{t - t^*}{2\s}.
\end{split}
\label{eq:usid2}
\end{align} 
Set an $\mathfrak{O}$-module  
\begin{align*}
R = \mathfrak{O} + \mathfrak{O} \iota = \mathfrak{O} + \dv \mathfrak{O} \h.
\end{align*} 
It is easy to see using (\ref{eq:usid2}) that $R$ is an order.
In particular, $R$ coincides with $M_2(\mathfrak{o})$ except for the case U-ii).
Additionally, define an order
\begin{align*}
R^\varrho = R \cap R \langle \varrho \rangle
\end{align*} 
so that 
\begin{align*} 
\varrho R^\varrho = R^\varrho \varrho, \ \varrho R, R \varrho, \mathfrak{O} \subset R^\varrho; \ \iota \in R^\varrho \subset R.
\end{align*}
By definition, $R = R^\varrho$ if and only if $E$ is in the case U.
If $E$ is in the case R-i), then we have an easy description 
\begin{align*}
R^\varrho = \begin{bmatrix}
 \mathfrak{o}& \p  \\
\mathfrak{o} & \mathfrak{o}
\end{bmatrix}.
\end{align*} 
For these orders $S$, we will study the structure of the Hankel part  
\begin{align}
S^H = S \cap H_2. \label{eq:S'}
\end{align} 
\begin{lem}\label{lem:R^0}
With notations as above, we have the followings.
\begin{enumerate}[$i)$]
\item 
As $\mathfrak{o}$-modules, we have identities:
\begin{align*}
R^H = \mathfrak{O} \op \mathfrak{o} \h, \ (\varrho R)^H = \varrho\mathfrak{O} \op \p \h.
\end{align*} 
\item 
If $E$ is in the case R, then 
\begin{align}
\h &\in \varrho^{-1}R^\varrho \setminus R^\varrho. \label{eqn:hRr}
\end{align} 
In this case, it holds that $(R^\varrho)^H= \mathfrak{O} + (\varrho R)^H$, and that
\begin{align*} 
(R^\varrho)^H = \mathfrak{O} \op \p \h,\ (\varrho R^\varrho)^H = \varrho\mathfrak{O} \op \p \h
\end{align*}
as $\mathfrak{o}$-modules.
\item If $E$ is in the case R, then 
\begin{align*}
R^\varrho = \mathfrak{O} \op \mathfrak{O} \varrho \h.
\end{align*}
\end{enumerate}
\end{lem}
\begin{proof}
i) 
obvious since $\{1, \h \}$ is an $\mathfrak{O}$-basis of $R$.
ii) 
In this case, $R^\varrho \neq R$.
Hence $\h \not\in R^\varrho$.
But, $\varrho \h \in \varrho R \subset R^\varrho$ and follows (\ref{eqn:hRr}). 
Other statements follow from it and i).
iii)
reduces to ii).
Indeed, for any element $r \in R^\varrho$, we may take an element $u \in \mathfrak{O}$ so that $r - u \varrho \h \in (R^\varrho)^H$.
\end{proof}
Let $m$ be a nonnegative integer.
Set a lattice 
\begin{align*}
R_m = \mathfrak{O} + \varrho^m R = \mathfrak{O} + \varrho^m \mathfrak{O} \iota
\end{align*}
for $m \in \Z$.
This is an order if $m \ge 0$.
This definition does not depend on the choice of the uniformizer $\varrho$.
Any element of $E^\times$ normalizes $R_m$.
If $E$ is in the case R, then $R_1 = R^\varrho$ by Lemma \ref{lem:R^0} iii).
The structure of the units group $R_m^\times$ is important to our Hecke theory in \ref{sec:Hecke}.
\begin{lem}\label{lem:unitorder}
Let
\begin{align*}
m_0 = \begin{cases}
1& \mbox{if $E$ is in the case of U}, \\
2& \mbox{if $E$ is in the case of R.} 
\end{cases}
\end{align*} 
Then, the followings are true.
\begin{enumerate}[i)]
\item 
If $m \ge m_0$, then $R_m^\times = \mathfrak{O}^\times + \dv\varrho^m R$. 
\item 
As a complete system of the representatives for $R_m^\times/R_{m+1}^\times$, we can take 
\begin{align*}
\begin{cases}
\{u + \dv s \h \mid s \in \E, \det (u+s\h) \neq 0\} & \mbox{if $m =m_0 -1$,} \\
u + \E \dv \varrho^m \h & \mbox{if $m \ge m_0$}.
\end{cases}
\end{align*} 
Here $u$ is an arbitrary fixed element of $\E^\times$.
\end{enumerate}
\end{lem}
\begin{proof}
i) Observe the norm of $u +  \dv\varrho^m r$ for $u \in E, r \in R$.
ii) follows from i) immediately.
\end{proof}
\subsection{Nonsplit paramodular forms}\label{sec:Nonparaform}
Keep the identification $E$ with the subalgebra of $M_2(F)$ as in the previous subsection.
We define the embedding 
\begin{align}
\phi_\s: \begin{bmatrix}
x & y \\
z & w
\end{bmatrix} \longmapsto \begin{bmatrix}
x & y \s^{-1}/2\\
2\s z  & w
\end{bmatrix} \label{eq:embG}
\end{align}
so that (\ref{eq:embphi}) are satisfied. 
The subgroup $\{diag(a,a^c) \mid a \in E^\times\} \subset \G$ will be denoted by $T$.
We choose an Atkin-Lehner element 
\begin{align*}
\imath = \begin{bmatrix}
\iota &  \\
& -\iota
\end{bmatrix}.
\end{align*} 
For subsets $S_i \subset M_2(F)$, we will denote 
\begin{align*}
\begin{bmatrix}
S_1 & S_2  \\
S_3 & S_4
\end{bmatrix} = \{g = \begin{bmatrix}
s_1 & s_2 \\
s_3 & s_4
\end{bmatrix} \in G \mid \mu(g) \in \mathfrak{o}^\times, s_i \in S_i \}. 
\end{align*}
Now define the {\it complete nonsplit paramodular group}
\begin{align*}
K_{2m} &= \begin{bmatrix}
R_m & \dv^{-1} \varrho^{-m} R_m  \\
\dv  \varrho^{m}R_{m} & R_m
\end{bmatrix}
\end{align*} 
for $m \ge 0$.
By definition, $K_0$ is equal to $G(\mathfrak{o})$ unless $E$ is in the case U-ii).
Additionally, we define 
\begin{align*}
K_{2m+1}^{\flat} =  
\begin{bmatrix}
R_m & \dv^{-1} \varrho^{-m} R_m  \\
\dv \varrho^{m+1} R_{m-1} & R_m
\end{bmatrix}, \ \ m \ge 1
\end{align*} 
and  
\begin{align*}
K_{2m+1} &= 
\begin{bmatrix}
R_{m} & \dv^{-1} \varrho^{-m} R_m \\
\dv  \varrho^{m+1}R_m & R_{m}
\end{bmatrix}, \\
K_{2m+1}^{\sharp} &=
\begin{bmatrix}
R_{m+1} & \dv^{-1} \varrho^{-m} R_m \\
\dv \varrho^{m+1}R_m & R_{m+1}
\end{bmatrix}, 
\end{align*}
for 
\begin{align*}
m \ge \begin{cases}
0 & \mbox{if $E$ is in the case U,} \\
1 & \mbox{if $E$ is in the case R.}
\end{cases}
\end{align*} 
We call these compact open subgroups {\it nonsplit paramodular groups of principal level $m$ over $E$} or paramodular groups, briefly.
Paramodular groups are normalized by $\imath$ and elements in $T$.
Observe that if we define $K_{2m+1}^\flat$ also for $m =0$ similarly, then it coincides with $K_{0}$.
When $E$ is in the case U-i), $K_1$ coincides with the Hecke subgroup 
\begin{align}
\Gamma_0(\p) = \begin{bmatrix}
M_2(\mathfrak{o}) & M_2(\mathfrak{o})  \\
\p M_2(\mathfrak{o}) & M_2(\mathfrak{o})
\end{bmatrix}. \label{eq:defScs}
\end{align}
When $E$ is in the case R-i), 
\begin{align*}
K_2 = \{k \in \begin{bmatrix}
\mathfrak{o} &\p &\mathfrak{o} & \mathfrak{o}  \\
\mathfrak{o} & \mathfrak{o} & \p^{-1} &\mathfrak{o} \\
\p& \p & \mathfrak{o} & \p  \\
\mathfrak{o} & \p & \mathfrak{o} & \mathfrak{o}  \\
\end{bmatrix} \mid \mu(k) \in \mathfrak{o}^\times \}
\end{align*}
is isomorphic to the original paramodular group of level $\p$, and 
\begin{align*}
K_3^\flat = \{k \in \begin{bmatrix}
\mathfrak{o} &\p &\mathfrak{o} & \mathfrak{o}  \\
\mathfrak{o} & \mathfrak{o} & \p^{-1} &\mathfrak{o} \\
\p& \p & \mathfrak{o} & \p  \\
\p & \p & \mathfrak{o} & \mathfrak{o}  \\
\end{bmatrix} \mid \mu(k) \in \mathfrak{o}^\times \}
\end{align*}
is isomorphic to the Klingen subgroup of level $\p$, in the sense of \cite{R-S}. 
The complete paramodular group $K_{2m}$ contains the Weyl element
\begin{align*}
w_{m} : = \begin{bmatrix}
 & - \dv^{-1}\varrho^{-m} \\
\dv \varrho^{m} &
\end{bmatrix}.
\end{align*}
Here the identification $E$ with the subalgebra of $M_2(F)$ is used. 
When $E$ is in the case R-i), $K_2$ and $K_3^\flat$ contain the Weyl element
\begin{align}
s_\p := \begin{bmatrix}
1& & &  \\
& & -\varpi^{-1} & \\
& \varpi& &  \\
& & & 1  \\
\end{bmatrix}. \label{def:sp}
\end{align} 
When $E$ is in the case R-ii), $K_2$ and $K_3^\flat$ contain $s_\p \langle u_\alpha \rangle$, where 
\begin{align}
u_\alpha :=
\begin{bmatrix}
1 & \alpha & & \\
 & 1 & & \\
& & 1& -\alpha \\
& & &1 
\end{bmatrix}. \label{eq:ualpha}
\end{align}
When $E$ is in the case U, $K_{2m+1}$ and $K_{2m+1}^{\sharp}$ are normalized by 
\begin{align*} 
w_{m}' := \begin{bmatrix}
 & - \dv^{-1} \varpi^{-m} \\
\dv \varpi^{m+1} &
\end{bmatrix}.
\end{align*}
For complete paramodular groups $K = K_{2m}$, we have a decomposition 
\begin{align}
K = N_K A_K \bar{N}_K \sqcup w_m N_K A_K \bar{N}_K \label{eq:decK1}
\end{align} 
if 
\begin{align*}
m \ge \begin{cases}
0 & \mbox{if $E$ is in the case U-ii)} \\
1 & \mbox{if $E$ is in the case U-i)} \\
2 & \mbox{otherwise}
\end{cases}
\end{align*}
For noncomplete paramodular groups $K$, if $K$ is not $K_3^\flat$ with $E$ in the case R, then  
\begin{align}
K = N_K A_K \bar{N}_K. \label{eq:decK}
\end{align} 

Let $\pi \in \Ir(PG)$.
Assume that $\B_\s(\pi) \neq \{0\}$.
Here the additive character $\psi$ is taken so that $\psi(\mathfrak{o}) = \{1\} \neq \psi(\p^{-1})$.
Define {\it paramodular subspaces (over $E$) of principal level of $m$}
\begin{align*}
\B_{2m} = \B_\s(\pi)^{K_{2m}} 
\end{align*} 
and $\B_{2m+1}^\flat, \B_{2m+1}, \B_{2m+1}^\sharp$, similarly.
Additionally, when $E$ is in the case U, define
\begin{align*} 
\B_{2m+1,\kappa} = \{\beta \in \B_{2m+1} \mid \pi(w_{m}') \beta =  \kappa\beta \}
\end{align*}
and $\B_{2m+1,\kappa}^\sharp$ similarly where $\kappa \in \{\pm 1\}$.
In this case, since $K_{2m+1}, K_{2m+1}^\sharp$ are normalized by $w_{m}'$, there are natural decompositions 
\begin{align*}
\B_{2m+1} = \op_{ \kappa}\B_{2m+1,\kappa}, \  \B_{2m+1}^\sharp = \op_{ \kappa}\B_{2m+1,\kappa}^\sharp.
\end{align*}  
Bessel vectors in the above subspaces are called {\it paramodular forms (over $E$) of principal level of $m$}.
In particular, vectors in $\B_{2m}$ are called {\it complete paramodula forms}.
There are obvious inclusive relationships 
\begin{align*}
\B_{2m} \subset \B_{2m+1}^{\flat} \subset \B_{2m+1} \subset \B_{2m+1}^{\sharp}.
\end{align*} 
The idempotent 
\begin{align}
e_m = e_{K_{2m+2}} \label{eq:emidm}
\end{align} 
of the Hecke algebra of $K_{2m+2}$ defines a mapping $\B_{2m+1}^{\sharp} \to \B_{2m+2}$.
But we do not know whether $e_m$ is injective and whether there is an inclusion map $\B_{2m+1}^\sharp \to \B_{2m+2}$.
Since the set $\{w_{m+1}, n_{x \varrho^{-m}} (x \in \F)\}$ is a complete system of representatives for the coset space $K_{2m+2}/(K_{2m+2} \cap K_{2m+1}^\sharp)$, if $\beta$ lies in $\B_{2m+1}^\sharp$, then 
\begin{align}
\begin{split}
e_m \beta(g) &= \frac{1}{q^\f + 1} (\beta(g w_{m+1}) + \sum_{x \in \F} \beta(g n_{x \varrho^{-m}})),  \\
e_m \beta(\hv^i) &=  \frac{1}{q^\f + 1}(\beta(\hv^i w_{m+1}) + q^\f\beta(\hv^i)), \ i \ge 0. 
\end{split}\label{eqn:e_m}
\end{align}
Using $e_m$, we can show the existence of complete paramodular forms as follows.
\begin{prop}\label{prop:extpf}
Let $\pi \in \Ir(PG)$ with $\B_\s(\pi) \neq \{0\}$. 
There exists a complete paramodular form not vanishing at the identity.
\end{prop}
\begin{proof}
By Proposition \ref{prop:prvan} ii), there exists an $\hat{\mathfrak{o}}^\times$-invariant $\beta \in \B_\s(\pi)$ not vanishing at $1$.
Since $\beta$ is a special Bessel vector, we may assume $\beta$ is invariant also under $T$.
Therefore, $\beta$ is invariant under the subgroup 
\begin{align}
\begin{bmatrix} 
R_{m} & \varrho^{m}R \\
\varrho^{m+1}R_{m} & R_{m}
\end{bmatrix}, \ m >>0 \label{eq:subgpex}
\end{align}
by the smoothness of $\pi$.
Now we will construct a complete paramodular form.
Suppose that $E$ is in the case U.
Consider the integral 
\begin{align*}
\beta' := \frac{1}{\vol(K_{2m+1})} \int_{K_{2m+1}} \pi(k) \beta d k = \frac{1}{\vol(N_{K_{2m}})} \int_{N_{K_{2m}}} \pi(n) \beta d n .
\end{align*} 
This is a paramodular form in $\B_{2m+1}$ not vanishing at $1$.
By the lemma below, $\beta'$ is vanishing at $w_{m+1}$ since $n_{x \h} \la w_{m+1} \ra = \bar{n}_{\varpi^{2m+2} x\h}$ lies in $K_{2m+1}$ if $x \in \p^{-1}$.
By (\ref{eqn:e_m}), $e_m \beta'$ is not vanishing at $1$, and a complete paramodular form.
Suppose that $E$ is in the case R.
The similar integral over $N_{K_{2m+3}^{\flat}}$ lies in $\B_{2m+3}^{\flat}$ and its image by $e_{m+1}$ is a desired one since $n_{x \h}  \la w_{m+2} \ra = \bar{n}_{x \varrho^{m+2}\h\varrho^{m+2}}$ lies in $K_{2m+4}$ if $x \in \p^{-1}$.
\end{proof} 
\begin{lem}\label{lem:vanishlem}
Let $\mathrm{G}$ be a group and $N, K$ be subgroups of $\mathrm{G}$.
Let $\Psi: N \to \C^\times$ and $\Omega: K \to \C^\times$ be homomorphisms.
Let $f$ be a $\C$-valued function on $\mathrm{G}$ such that 
\begin{align*}
f(ngk) = \Psi(n)\Omega(k) f(g), \ \ n \in N, g \in \mathrm{G}, k \in K
\end{align*} 
If there exists an element $n \in N$ such that $n \langle g \rangle \in K$ and $\Psi(n) \neq \Omega(n \langle g \rangle)$, then $f(g) = 0$.
\end{lem} 
\begin{proof}
Obvious.
\end{proof}
\nid
The following is an analogue of Theorem 3.1.3. of \cite{R-S} not only for complete paramodular forms but also complete paramodular vectors.
\begin{thm}\label{thm:indpara}
Let $(\pi,V) \in \Ir(PG)$ be infinite dimensional. 
Let $v_1, \ldots, v_r \in V$ be nontrivial vectors invariant under complete paramodular groups of different levels over a fixed field.
Then $v_1, \ldots, v_r$ are linearly independent.
\end{thm}
\nid
By the proof of Theorem 3.1.3 of loc. cit., this theorem follows from:
\begin{lem}
Two complete paramodular groups of different levels over a fixed field generate a subgroup containing $Sp_4(F)$.
\end{lem}
\begin{proof}
Since the proofs are similar, we only give that for the case U.
Let $K = K_{2m}, L = K_{2n}$ with $m > n$.
It suffices to show that two subgroups $L$ and $N_{K}$ generate a subgroup containing $\Sp_4(F)$.
Using (\ref{eq:ufid}) for elements of $N_{K}$, we find that $w_{m}$ is expressed as a product of elements of $L$ and $N_{K}$, and so is $w_{n} w_{m}$.
Therefore, any elements of $N$ and $\bar{N}$ are also expressed by those in $L$ and $N_{K}$.
Now the assertion follows from the fact that $\Sp_4(F)$ is generated by $N, \bar{N}$, and $\Sp_4(F) \cap L$.
\end{proof}
We will see some good properties of paramodular forms when considering `canonical' Piatetsk-Shapiro zeta integral of them.
When $K$ is a paramodular group of principal level $m$, and $\beta$ lies in $\B_\s(\pi)^K$, let 
\begin{align*}
K^*  = K \la w_{m} \ra, \ \ \beta^* = \pi(w_{m})\beta.
\end{align*} 
Of course, $K^* = K$ and $\beta^* = \beta$ when $K$ is complete.
By definition, 
\begin{align}
\psi(l_\s(n_{u + x \h})) = \psi(x), \ \ u \in E, x \in F.  \label{eq:psiM'}
\end{align}
Observing $N_K, N_{K^*}$, Lemma \ref{lem:R^0}, \ref{lem:vanishlem}, we find that 
\begin{align}
\begin{split}
Z(s,\beta) &\in \C[[X]], \\
Z(s,\beta^*) &\in 
\begin{cases}
\C[[X]] & \mbox{if $K = K_{2m+1}^\flat$ or $K_{2m}$}, \\
X^{-1}\C[[X]] & \mbox{otherwise}.
\end{cases} 
\end{split}
\label{eq:ZbZ}
\end{align} 
Let $m \in \Z$.
Let $\df$ denote the order of the relative discriminant of $E/F$.
Define
\begin{align*}
\vp_m(x,y) =  q^{\f (\df-m)} \Ch(x,y; \mathfrak{P}^{m-\df} \op \mathfrak{O}) \in \Ss(E^2).
\end{align*}
The stabilizer subgroup of $\vp_m$ by $\G$ is
\begin{align*}
\K_m := 
\{k = \begin{bmatrix}
x &y \\
z &w
\end{bmatrix} \mid \det(k) \in \mathfrak{o}^\times, x,w \in \mathfrak{O}, y \in \mathfrak{P}^{\df-m}, z \in \mathfrak{P}^{m- \df} \}.
\end{align*}
Observe that $\K_m \subset K_{2m}$.
Now define the canonical $m$-th Piatetski-Shapiro zeta integral of $\beta \in \B_\s(\pi)$ by   
\begin{align*}
Z_m(s,\beta) = Z(s,\beta,\vp_m).
\end{align*} 
Let $d u, dt, dn, dk$ be the Haar measures on $\hat{F^\times}, T, N_{\G}, \K_0$ respectively such that $\vol(\hat{\mathfrak{o}}^\times) = \vol(\mathfrak{O}^\times) = \vol(N_{\K_0}) = \vol(\K_0) = 1$. 
Then 
\begin{align*}
\int_{N_\G} \int_{F^\times} \int_{A_\G} \int_{\K_0} \int  f(n\hat{u} t k) |u|^{-3} d k dt d u dn
\end{align*} 
defines a Haar measure on $\G$ such that $\vol(\K_0) = 1$.
By this measure, for a paramodular form $\beta$ of principal level $m$, it holds
\begin{align}
Z_n(s,\beta)= \frac{\zeta_E(s+\frac{1}{2})}{(1+q^\f)} \times
\begin{cases}
(Z + q^\f Z^*) & \mbox{if $n = m$,}\\
(Z + X^{\f} Z^*) & \mbox{if $n = m+1$,} 
\end{cases}
\label{eq:PSzeta}
\end{align} 
where $Z =  Z(s,\beta)$, $Z^* = Z(s,\beta^*)$.
\begin{lem}\label{lem:PSzeta}
With notations as above, for a paramodular form $\beta$ of principal level $m$ we have the followings. 
\begin{enumerate}[$i)$]
\item If $Z_m(s,\beta) = 0$, then
\begin{align*}
Z_{m+1}(s,\beta) = \frac{Z}{1+q^\f}.
\end{align*}
\item If $Z_{m+1}(s,\beta) = 0$, then
\begin{align*}
Z_m(s,\beta) &= - \frac{Z}{X'^{\f}(1+q^\f)}.
\end{align*}
\item
If $\beta$ lies in $\B_{2m}$, then   
\begin{align*}
Z_n(s,\beta) = \zeta_E(s+\frac{1}{2}) Z \times 
\begin{cases}
1 & \mbox{in case of $n =m$,} \\
\frac{1 + X^{\f}}{1+q^\f} & \mbox{in case of $n = m+1$.}
\end{cases}
\end{align*} 
\item
If $E$ is in the case U, and $\beta$ lies in $\B_{2m+1,\kappa}$ or $\B_{2m+1,\kappa}^\sharp$, then 
\begin{align*}
Z_n(s,\beta) = \frac{Z}{1+q^2} \times 
\begin{cases}
\frac{1}{1-\kappa X'^{-1}} & \mbox{in case of $n = m$,} \\
\frac{1}{1- \kappa X'} & \mbox{in case of $n = m+1$.}
\end{cases}
\end{align*}
\end{enumerate}
\end{lem}
\begin{proof}
Follows from (\ref{eq:PSzeta}).
\end{proof}
\noindent
To consider the functional equation (\ref{eqn:LFE}) for paramodular forms, it is convenient to introduce the following zeta polynomial and sign.
For a paramodular form $\beta \in \B_\s(\pi)$, we call the ratio 
\begin{align*}
P_m(X,\beta) := \frac{Z_m(s,\beta)}{L(s,\pi)},
\end{align*}
which lies in $\C[X^{\pm}]$ by definition, the {\it $m$-th zeta polynomial} of $\beta$.
Since paramodular groups $K$ are normalized by $\imath$, $\pi(\imath)$ acts on $\B_\s(\pi)^K$ and has eigenvalues $\ep \in \{ \pm 1\}$.
Therefore, we have a natural decomposition 
\begin{align*}
\B_\s(\pi)^K = \op_{\ep} \B_\s(\pi)^{K,\ep},
\end{align*}  
where $\B_\s(\pi)^{K,\ep}$ denotes the eigenspace corresponding to $\ep$, which will be called paramodular subspace of sign $\ep$ and denoted by $\B_{2m}^\ep$, etc.
We say a paramodular form is {\it of sign $\ep$} if it belongs to $\B_\s(\pi)^{K,\ep}$.
\begin{prop}\label{prop:FEHT} 
If $\beta$ is of sign $\ep$, then 
\begin{align}
P_m(X^{-1}, \beta) = \ep\ep_\pi X^{(n_\pi-\f m)} P_m(X,\beta). 
\label{eq:FNEQP}
\end{align}
\end{prop}
\begin{proof}
By the functional equation (\ref{eqn:LFE}), and the fact $\vp_m^\sharp(z) = q^{-\f m}\vp_m(\varrho^{m} z)$, 
\end{proof}
\nid 
If $P(X) \in \C[X^\pm]$ is in the form of $c_{-n}X^{-n} + \cdots + c_m X^m$ with $c_{-n}c_m \neq 0$, then we call $m-n$ the {\it diameter of $P$}, and denote it by ${\rm dia} P$.
We say $P(X) \in \C[X^\pm]$ has {\it sign $\ep \in \{\pm \}$}, if 
\begin{align*}
X^{{\rm dia} P} P(X^{-1})= \ep P(X).
\end{align*} 
The above proposition says:
\begin{lem}\label{lem:sdzp}
If $\beta$ is of sign $\ep$ and $P_m(X,\beta)$ is not zero, then $P_m(X,\beta)$ has sign $\ep \ep_\pi$, and 
\begin{align*}
{\rm dia} P_m(X,\beta) = \f m -n_\pi.
\end{align*} 
\end{lem}
\nid
Applying this lemma to $P_m(X,\beta)$ and $P_{m+1}(X,\beta)$ for a paramodular form $\beta$ of principal level of $m$, we obtain from (\ref{eq:PSzeta}) the following lemma. 
\begin{lem}\label{lem:Z*Z}
If $\beta$ is of a sign and $Z(s,\beta^*) = 0$, then $Z(s,\beta) = 0$.
\end{lem}
Now we will prove the main theorem below in this subsection, which is an analogue of Corollary 4.3.8. of loc. cit., and play a crucial role for one-dimensionality of newforms in sect. \ref{sec:NFN}.
We need the following linear operators $z_m: \B_{2m+ 1}^\sharp \to \B_{2m+ 3}^\sharp$ defined by 
\begin{align*}
\beta \longmapsto \frac{1}{q^{\f}}\sum_{n \in N_{K_{2m+2}}/N_{K_{2m}}} \pi(n)\beta
\end{align*} 
and  
\begin{align}
\eta: \B_{n}^\bullet \ni \beta \longmapsto \pi(\hv^\f)\beta \in \B_{n+ 2}^\bullet.\label{eq:defeta}
\end{align}
Here note that 
\begin{align*}
\pi(w_{m+1}w_m) = \eta.
\end{align*} 
By (\ref{eqn:e_m}), if $\beta \in \B_{2m}$, then  
\begin{align*}
e_m \beta = \frac{z_m + \eta }{q^\f+1}\beta.
\end{align*} 
\begin{thm}\label{thm:paravn}
Let $\pi \in \Ir(PG)$.
Let $\beta \in \B_\s(\pi)$ be a nonsplit paramodular form of a sign.
Assume that $Z(s,\beta^*)$ is zero.
Then $\beta$ is identically zero.
\end{thm}
\begin{proof}
Since the proofs are similar, we treat only the situation where $E$ is in the case U.
By the last lemma, $Z(s,\beta) = 0$.
By Proposition \ref{prop:prvan}, $pr(\beta) =0$.
Let $n$ be the principal level of $\beta$.
By the smoothness of $\pi$, there exists a sufficiently large $r$ such that $z_{n+r} \beta$ is identically zero.
Here observe that $z_{n+r} \beta \in \B_{2(n+r)+3}^\sharp$.
First, suppose that $\beta \in \B_{2n}$.
We will claim by induction that, for $r \ge 0$, there exist linear operators $b_r: \B_{2n} \to \B_{2(r+n+2)}$ and $c_r: \B_{2n} \to \B_{2(r+n+1)}$ such that $z_{n+r} = b_r+ c_r$.
The claim for $r = 0$ is true, indeed, $b_0 = - \eta, c_0 = (q^\f+1)e_n$.
Write $e_m' = (q^\f+1) e_m$.
Assume the claim for $r \ge 0$.
Then since $\beta$ lies in $\B_{2n}$, it hold that 
\begin{align*}
z_{r+n+1}\beta &= z_{r+n+1} \circ z_{r+n}\beta\\
&= (e_{r+n+1}' -\eta) \circ (b_r+ c_r)\beta \\
&= - \eta c_r \beta + (e_{r+n+1}' \circ c_r + z_{r+n+1}\circ b_r)\beta.
\end{align*}
Here since $b_r \beta \in \B_{2(r+n+2)}$ by assumption, $z_{r+n+1}\circ b_r \beta = b_r \beta$.
Therefore, 
\begin{align*} 
z_{r+n+1}\beta = -\eta c_r \beta + (e_{r+n+1}' \circ c_r + b_r)\beta.
\end{align*}
So, for $r \ge 0$,
\begin{align}
b_{r+1} := - \eta c_r, \ c_{r+1} :=(e_{r+n+1}' \circ c_r + b_r) \label{eq:defbrcr}
\end{align} 
are the desired operators.
This proves the claim.
Now consider
\begin{align*} 
(e_{r+n}' -\eta)\circ \cdots \circ (e_n' -\eta) \beta &= z_{n+r}\circ \cdots \circ z_n \beta \\
&= z_{n+r} \beta =  (b_r+ c_r)\beta =  0.
\end{align*}
Since $b_{r} \beta$ and $c_r \beta$ are complete paramodular forms of different levels, they are linearly independent by Theorem \ref{thm:indpara}.
Therefore, $b_{r} \beta = c_r \beta = 0$.
Assume $r = 0$, then this means $\eta \beta = 0$ and $\beta = 0$ since $\eta$ is injective.
Assume that $r > 0$.
By (\ref{eq:defbrcr}), $\eta c_{r-1} \beta = 0$, and thus $c_{r-1}\beta = 0$.
Therefore $e_{r+n}' \circ c_{r-1}\beta = 0$.
By (\ref{eq:defbrcr}) again, $b_{r-1} \beta = 0$.
Thus $z_{r+n-1} \beta = b_{r-1}\beta + c_{r-1}\beta = 0$.
Hence, $\beta = 0$ by induction.
Next suppose that $\beta$ is not a complete paramodular form.
By induction, under the situation where $\beta$ lies in $\B_{2n+1}^{\sharp}$, it suffices to show that $\beta =0$ if $z_n \beta = 0$.
Since $Z(s,\beta) = 0$, $e_n \beta$ is identically zero by (\ref{eqn:e_m}) and the above argument.
Thus, $z_n \beta = -\eta \beta$ is identically zero and so is $\beta$.
This finishes the proof.
\end{proof}
\nid
By (\ref{eq:PSzeta}) and Lemma \ref{lem:Z*Z}, 
\begin{Cor}\label{cor:Z}
Let $\beta$ be a nonzero paramodular form of principal level $m$ of a sign.
Then, at least one of $Z_m(s,\beta)$ and $Z_{m+1}(s,\beta)$ is not zero.
In particular, if $\beta$ is a complete paramodular form of principal level $m$, then $Z_m(s,\beta)$ is not zero.
\end{Cor}
Here, we introduce some notations.
The lowest principal level of nontrivial complete paramodular subspace is called the {\it minimal level of $\pi$} and denoted by $M_\pi$.
The subspace $\B_{2M_\pi}$ and its nontrivial vectors are called the {\it minimal space} and {\it newforms of $\pi$}.
We will show in sect. \ref{sec:old} that the mappings $e_m: \B_{2m} \to \B_{2m+2}$ are injective in case that $\pi$ is in a SK-packet, and it makes sense to introduce this vocabulary at least in this case.
Further, consider the following sequence of paramodular subspaces of sign $\ep$: 
\begin{align}
\B_0^\ep
\begin{cases}
\subset \B_1^\ep \subset \B_1^{\sharp,\ep} \to \B_2^\ep \subset \B_3^{\flat,\ep} \subset \B_3^\ep \subset \B_3^{\sharp,\ep} \cdots  & \mbox{if $E$ is in the case U}, \\
\to \B_2^{\ep} \subset \B_3^{\flat,\ep} \subset \B_3^{\ep} \subset \B_3^{\sharp,\ep} \to \B_5^{\flat,\ep} \cdots & \mbox{if $E$ is in the  case R},
\end{cases} \label{eq:secpfm}
\end{align}
where the arrows indicate the mappings $e_m$.
The first nontrivial subspace is called the {\it strict minimal space} of sign $\ep$. 
If $E$ is in the case R, the subspaces $\B_4^{\ep}, \B_6^{\ep}, \ldots$ are excluded from the sequence for the following reason.
\begin{lem}\label{lem:B2mB2m-1}
Assume that $E$ is in the case R.
If $\B_{2m}^{\ep} \neq \{ 0 \}$ for $m \ge 2$, then $\B_{2m-1}^{\sharp, \ep} \neq \{ 0 \}$.
\end{lem}
\begin{proof}
Let $K = K_{2m-1}^\sharp$.
By Lemma \ref{lem:R^0}, 
\begin{align*}
(\varrho^m R_m)^H = \mathfrak{P}^m \op \p^{m}\h =  (\varrho^{m} R_{m-1})^H.
\end{align*}  
Therefore, $\bar{N}_{K_{2m}} = \bar{N}_{K}$.
By definition, $A_{K_{2m}} = A_{K}$.
If there is a nontrivial $\beta \in \B_{2m}^{\ep}$, then $\beta(\hv^j)$ is not zero for some $j$ by Theorem \ref{thm:paravn}, and 
\begin{align*} 
\vol(K)^{-1}\int_{K} \pi(k) \beta dk = \vol(N_{K})^{-1}\int_{N_{K}} \pi(n) \beta dn
\end{align*}
(c.f. (\ref{eq:decK})) is also not zero at $\hv^j$.
\end{proof}
\nid
By Proposition \ref{prop:extpf}, at least there is a nontrivial paramodular form of sign plus or minus.
Nontrivial vectors of the strict minimal space of sign $\ep$ are called {\it strict newforms of sign $\ep$}.
Its principal level is called the {\it strict minimal level of sign $\ep$} if it exists, and denoted by $m_\pi^\ep$.
If it does not exist, then write $m_\pi^\ep = \infty$.
By definition, 
\begin{align} 
\min\{m_\pi^+, m_\pi^-\} \le M_\pi. \label{eq:mmM}
\end{align} 
Of course, $M_\pi$ and $m_\pi^\ep$ depend on the choice of $E$.
By (\ref{eq:ZbZ}) and (\ref{eq:PSzeta}), $XP_m(X,\beta)$ and $P_{m+1}(X,\beta)$ are polynomials in $X$, and at least one of them is not zero by Corollary \ref{cor:Z}.
Therefore we obtain from Lemma \ref{lem:sdzp} an estimation:
\begin{align}
m_\pi^\ep \ge \frac{n_\pi-2}{\f}. \label{eq:m>npre}
\end{align} 
\begin{lem}\label{lem:m>n}
Let $\beta \in \B_{2m}^\ep$.
Then, 
\begin{align*}
m = m_\pi^\ep = \frac{n_\pi}{\f}  \Longleftrightarrow P_m(X,\beta) \in \C^\times. 
\end{align*} 
In this case, $\ep = \ep_\pi$.
\end{lem}
\begin{proof}
Follows from Lemma \ref{lem:PSzeta} iii).
\end{proof}
\begin{lem}\label{lem:flatmin}
Suppose that $\beta \in \B_{2m+1}^{\flat,\ep}$ is a strict newform.
Then $Z_m(s,\beta) = 0$, and $Z_{m+1}(s,\beta) = Z (s,\beta)/(1+q^\f)$, which is not zero.
\end{lem}
\begin{proof}
There is an isomorphism $K_{2m}/K_{2m+1}^\flat \simeq \K_m/ \K_m \cap K_{2m+1}^\flat$. 
Since $\B_{2m}^\ep = \{0\}$, we have 
\begin{align*}
e_{m-1}\beta = \vol(\K_m)^{-1} \int_{\K_{m}} \pi(k) \beta dk = 0
\end{align*} 
and $Z_m(s,\beta) (= Z_{m}(s,e_{m-1}\beta))= 0$.
The last assertion follows from Lemma \ref{lem:PSzeta} i) and Corollary \ref{cor:Z}.
\end{proof}
Now for unramified representations,  when $E = E_\s$ is not in the case U-ii), we can show the following theorem.
\begin{thm}\label{thm:unrcase}
Let $\pi \in \Ir(PG)$ be unramified with $\B_\s(\pi) \neq 0$.
Then, 
\begin{align}
n_\pi = 0, \ \ \ep_\pi = +, \ \ L(s,\pi) = L(s,\phi_\pi). \label{eq:unrnep}
\end{align}
The strict minimal space of sign plus is $\B_0^+$ and one-dimensional.
Assume that $\beta \in \B_0^+$ is not identically zero.
Then $Z(s,\beta)$ is equal to 
\begin{align*}
L(s,\pi) \times 
\begin{cases}
(1-(X')^2)& \mbox{if $E = E_\s$ is in the case U}, \\
1- X' & \mbox{if $E$ is in the case R} 
\end{cases}
\end{align*} 
up to scalars, and it holds that 
\begin{align}
P_0(X,\beta) \in \C^\times. \label{eq:P0C}
\end{align} 
\end{thm}
\nid
A proof for (\ref{eq:unrnep}) including the case U-ii) is as follows.
Any unramified representation $\pi$ is a constituent of a Borel parabolic induction, and given by the $\theta$-lift from $\GL_2(F) \times \GL_2(F)$ (c.f. \cite{G-T}).
From the Table A.14 of \cite{R-S}, and the table in p. 538 of \cite{R-S2}, $\pi$ is a representation of type I, IIb, or Vd since $\B_\s(\pi) \neq \{0 \}$.
Type I is generic, and it was showed in \cite{Sc-T} that $L(s,\pi) = L(s,\phi_\pi)$, and in \cite{PS-S} that $\ep(s,\pi,\psi) = \ep(s,\phi_\pi,\psi)$ for generic representations.
In particular, when $\pi$ is unramified generic, $\pi$ is $\theta(\tau^1 \bt \tau^2)$ for some unramified $\tau^1,\tau^2 \in \Ir(\PGL_2(F))$, and (\ref{eq:unrnep}) is now obvious.
Type IIb, Vd are $SK(\tau)$ for some $\tau \in \Ir'(\PGL_2(F))$ (c.f. p. 511 of \cite{R-S2}).
In particular, $\tau$ is unramified when $\pi$ is so, and (\ref{eq:unrnep}) for unramified $SK(\tau)$ follows from Theorem \ref{thm:LSK}.

For the other statements of the theorem, we use the well-known fact that an unramified $\pi$ has a unique $G(\mathfrak{o}) (=K_0)$-invariant vector up to scalars.
When $E$ is not in the case U-ii), since $\imath \in K_0$, we obtain $\B_0 = \B_0^+$, and the other statements from Lemma \ref{lem:PSzeta} iii), \ref{lem:m>n}, and Theorem \ref{thm:paravn}.

The proof in the case U-ii) will be given in the next section.
\subsection{Hecke theory}\label{sec:Hecke}
Let $\s$ be a Hankel matrix in (\ref{eq:defsigma}) such that $E_\s$ does not split.
Let $\pi \in \Ir(PG)$ be unitary.
Assume that $\B_\s(\pi) \neq \{ 0\}$.
Let $\ep$ be a sign such that $\B_\s(\pi)^\ep \neq \{ 0\}$.
Let $K$ be the paramodular group defining the strict minimal space of sign $\ep$, and $m = m_\pi^\ep$ be the strict minimal level.
We will compute $m$-th and $(m+1)$-th canonical Piatetski-Shapiro zetas of strict newforms.
By Lemma \ref{lem:PSzeta} and \ref{lem:flatmin}, it suffices to compute $Z(s,\beta)$, and $Z(s,\beta^*)$ additionally when $K = K_{2m+1}, K_{2m+1}^\sharp$ with $E$ in the case R.
For this, we need the following Hecke operators $T_K^\pm$ acting on $\B_\s(\pi)^{K,\ep}$ defined by 
\begin{align*}
\vol(K)^{-1} \int_{G} \vp_K^\pm(g)\pi(g) d g
\end{align*} 
where $\vp_K^\pm \in \Ss(G)$ denotes the characteristic function of $K \hv^\pm K$.
From the unitarity assumption of $\pi$ and the triviality of $w_\pi$, it follows the coincidence $T_K^+ = T_K^-$ and the self-adjointness of $T_K$, when $K$ is complete.
But when $K$ is not complete, both of them do not hold.
For this reason we define
\begin{align*}
T = \begin{cases}
T_K^+ & \mbox{if $K$ is complete}, \\
T_K^+ + T_K^-& \mbox{otherwise}, 
\end{cases}
\end{align*}
so that $T$ is self-adjoint.
Define a self-adjoint Hecke operator $T^*$ acting on $\B_\s(\pi)^{K^*,\ep}$ similarly. 
A basis of the strict minimal space consists of eigenvectors $\{\beta\}$.
Observe that 
\begin{align*}
T \beta = \lambda \beta \Longleftrightarrow T^* \beta^* = \lambda \beta^*
\end{align*}
for $\lambda \in \C$.
We call $\lambda$ the eigenvalue of $\beta$ (relevant to $T$).
Put 
\begin{align*} 
c_i = \beta(\hv^i), \ \ c_i ^*= \beta^*(\hv^i).
\end{align*}
Computing the relevant coset spaces $K \hv^\pm K/K$, we obtain 
\begin{align*}
\lambda \beta = \sum_{n \in N_K\la \hv \ra/N_K} \pi(\hv n) \beta + \sum \pi(\hv^{-1}n') \beta,
\end{align*}
where $n'$ run through the coset space 
\begin{align}
\begin{split}
\begin{cases}
\bar{N}_K\la \hv^{-1} \ra/\bar{N}_K& \mbox{when noncomplete $K$ is not $K_3^\flat$ with $E$ in the case R.} \\
\bar{N}_ {K_{2m-1}}/ \bar{N}_K& \mbox{when $K = K_{2m (\ge 2)}$}.
\end{cases}
\end{split}\label{eq:ess1}
\end{align} 
Therefore,
\begin{align*}
\lambda c_i &= q^{3} c_{i+1} + \sum \beta(\hv^{i-1}n'), \ \ \ i \ge 0.
\end{align*}
Our task is to compute the last sum, which is called the {\it heart of} $T \beta$.
If $K = K_{2m+1}, K_{2m+1}^\sharp$ with $E$ in the case R, then we also need to compute $T^* \beta^*$.
It holds that 
\begin{align*} 
\lambda c_i^* &= q^{3} c_{i+1}^* + \sum \beta^*(\hv^{i-1}n'), \ \ \ i \ge -1
\end{align*} 
where $n'$ run through the coset space $\bar{N}_{K^*}\la\hv^{-1}\ra/\bar{N}_{K^*}$, and the last sum is called the heart of $T^* \beta^*$.
If $K$ is not 
\begin{align}
 \begin{cases}
K_0, K_1 & \mbox{if $E$ is in the case U-i)}, \\
K_0, K_1, K_1^\sharp & \mbox{if $E$ is in the case U-ii)}, \\
K_0, K_2, K_3^{\flat},K_3 & \mbox{if $E$ is in the case R,} 
\end{cases}
\label{excK}
 \end{align}
then we can compute the hearts by using $\B_{\s}(\pi)^L = \{0\}$ as below, where $L$ is the paramodular group chosen as follows: 
\begin{table}[H]
\begin{tabular}{lcccc}\toprule
$K$  & $K_{2m}$ & $K_{2m+1}^\flat$ & $K_{2m+1}$ & $K_{2m+1}^\sharp$  \\ \midrule
$L$ & $K_{2m-1}^\sharp$ & $K_{2m-1}^\sharp$ & $K_{2m+1}^\flat$ & $K_{2m+1}$  \\ 
 \bottomrule
\end{tabular}
\end{table}
\noindent
Here observe that $N_L \subset N_K$.

Suppose that $K$ is neither $K_{2m+1}^\sharp$ nor (\ref{excK}).
It holds that $A_L = A_K$ and 
\begin{align*}
L/(L\cap K) \simeq \bar{N}_L/\bar{N}_K
\end{align*}
by (\ref{eq:decK}).
The latter coset space contains (\ref{eq:ess1}).
Therefore, the heart  is zero, and it follows that $c_{i+1} = \lambda q^3 c_i $.
Now we consider the linear operator 
\begin{align}
U_K  = \int_{N_{K}} \pi(n\hv) dn \label{eq:UKop}
\end{align}
for $\B_\s(\pi)$.
It is easy to see that $U_K \beta \in \B_{\s}(\pi)^L$.
This implies that
\begin{align*} 
Z(s,\beta) = c_0.
\end{align*}
If $K$ is not $K_{2m+1}$ with $E$ in the case R, then this value is not zero by Theorem \ref{thm:paravn}, and hence $\lambda = 0$.
Otherwise, by considering $\int_{N_K} \pi(n) \beta^* dn$ and $\bar{N}_L \subset \bar{N}_{K^*}$, we conclude that
\begin{align*} 
Z(s,\beta^*) = q^{s-3/2}c_{-1}^*
\end{align*}
(c.f. (\ref{eq:ZbZ})).
This value is not zero by the same Theorem. 
Now it is possible to describe the zetas except for the case $K = K_{2m+1}^{\sharp}$.

Now consider the situation where $K = K_{2m+1}^{\sharp}$, and $m \ge 1$ when $E$ is in the case U-ii).
For such a $K$, we need  the following lemmas and the compact subgroup
\begin{align*} 
\Gamma_- = \begin{bmatrix}
1 + \varpi R& \\
\varpi R_a & 1+ \varpi R
\end{bmatrix} 
\end{align*} 
where $a$ is $0$ (resp. $1$) if $E$ is in the case U (resp. R).
\begin{lem}\label{lem:K2m+1sharp}
Assume that a set $\{r_1,\ldots,r_l \}$ of $R_m^\times$ is a complete system of representatives for $R_{m}^{\times}/R_{m+1}^{\times}$.
If $E$ is in the case U (resp. R) and $m \ge 0$ (resp. $\ge 1$), then any set 
\begin{align*}
\{\begin{bmatrix}
r_1 & * \\
* & *
\end{bmatrix}, \ldots,  \begin{bmatrix}
r_l & * \\
* & *
\end{bmatrix} \} \subset K_{2m+1}
\end{align*}
is that for $K_{2m+1}/K_{2m+1}^{\sharp}$.
\end{lem}
\begin{proof}
By (\ref{eq:decK}).
\end{proof}
\begin{lem}\label{lem:pcp}
Let $\beta$ be in $\B_\s(\pi)^{\Gamma_-}$. 
Suppose that $\B _1$ (resp. $\B_3$) is zero if $E$ is in the case U (resp. R).
If $k$ runs through the set of Lemma \ref{lem:unitorder} ii) with $m = 0$ (resp. $m = 1$), then it holds that 
\begin{align*}  
\sum_{k} \beta(\hv^j a_{k}) = 0, \  \ j \ge 0.
\end{align*}
\end{lem}
\begin{proof}
Let $K$ denote $K_1$(resp. $K_3$.).
By (\ref{eq:decK}) and Lemma \ref{lem:K2m+1sharp}, 
\begin{align*}
\beta' := \int_{N_{K}}\int_{\mathfrak{O}^\times}\int_{\mathfrak{o}^\times} \sum_{k}\pi(\hat{u}a_{t k^\dag}n)\beta du dt dn
\end{align*} 
is $K$-invariant.
Since $\B_\s(\pi)^K = \{0 \}$, $\beta'$ is identically zero.
In particular, $\beta'(\hv^j)$ is a nonzero constant multiple of $\sum_{k}\beta(\hv^ja_{k})$.
Hence the assertion.
\end{proof}
\begin{lem}\label{lem:B1sharp}
Suppose that $\beta \in \B_\s(\pi)$ ia a strict newform in $\B_{2m+1}^\sharp$.
Assume that $m \ge 0$ if $E$ is in the case U-i), and $m \ge 1$ otherwise.
Then it holds that
\begin{align*}
\sum_{x \in \F} \beta(\hv^{j} \bar{n}_{x \dv \varpi^{\f m}\h}) = \beta(\hv^{j}), \ \ j \in \Z.
\end{align*} 
\end{lem}
\begin{proof}
Since the proofs are similar, we only give the proof for the case U-i).
Let 
\begin{align}
t = \begin{bmatrix}
y&ze  \\
z &y
\end{bmatrix} \in \mathfrak{O}. \label{eq:tO}
\end{align} 
Let $x \in \mathfrak{o}$.
Then,
\begin{align*}
\beta(\hv^{j} \bn_{\varpi^{2m}x \h}) = \beta(\hv^{j} n_{\varpi^{-m} t}\bar{n}_{\varpi^{2m} x \h}) = \beta(\hv^{j}
\begin{bmatrix}
1 + x \varpi^{m} t \h & \varpi^{-m} t  \\
x \varpi^{2m} \h & 1
\end{bmatrix}).
\end{align*} 
Here 
\begin{align*}
1 + x\varpi^{m} t \h = \begin{bmatrix}
 1& x \varpi^{m} y  \\
& 1 + x\varpi^{m} z
\end{bmatrix} \in R_m.
\end{align*} 
(This is a unit of $R_m$ if $m \ge 1$.)
If $1 + x\varpi^m t \h \in R_m^\times$, then $v := - (1+x \varpi^m t\h)^{-1} t$ lies in $R_m$, and
\begin{align*}
\beta(\hv^{j}
\begin{bmatrix}
1 + x \varpi^m t \h & \varpi^{-m}t  \\
x \varpi^{2m} \h & 1
\end{bmatrix}) &= \beta(\hv^{j}
\begin{bmatrix}
1 + x\varpi^m t \h & t  \\
x \varpi^{2m} \h & 1
\end{bmatrix} n_{\varpi^{-m}v}) \\
& = 
\beta(\hv^{j}
\begin{bmatrix}
1 + x \varpi^m t \h &  \\
x \varpi^{2m} \h & 1- x\varpi^m \h v
\end{bmatrix}) \\
&= \beta(\hv^{j}
\bn_{x\varpi^m(1 + x\varpi^mz)^{-1} \h}a_{1+ x \varpi^mt\h}).
\end{align*} 
For any element $r$ in Lemma \ref{lem:unitorder} ii) with $u = 1$, and a fixed $x \in \mathfrak{o}^\times$, there exists a $t \in \mathfrak{O}$ such that $1 + x \varpi^m t \h$ coincides with $r$.
From Lemma \ref{lem:K2m+1sharp} and the assumption $\B_{2m+1} = \{0\}$, it follows that $
\beta(\hv^{j} \bar{n}_{x\varpi^m\h}) = 0$ if $x \in \mathfrak{o}^\times$.
Thus the assertion.
\end{proof}
 \noindent
 We will compute the zetas according to the case of $E$.\\
 {\bf Case U).}
We may assume that 
\begin{align*}
\beta \in \B_{2m+1,\kappa}^\sharp
\end{align*} 
where $\kappa \in \{ \pm \}$ since $T$ commutes with $\pi(w_{m}')$.
The coset space (\ref{eq:ess1}) is isomorphic to $(\dv \varpi^{m}R_m)^H/ (\dv \varpi^{m+1} R_m)^H$, and to $\E \varpi^m \op \F \varpi^{2m}\h$ by Lemma \ref{lem:R^0}.
We divide the heart of $T\beta$ according to the partition: 
\begin{align*}
\E \varpi^m \op \F \varpi^{2m}\h = \F \varpi^{2m}\h \sqcup \E^\times \varpi^m \sqcup (\E^\times \varpi^m \op \F^\times \varpi^{2m}\h).
\end{align*} 
We will compute the heart for $i \ge 1$.
The first part of the heart equals $c_{i-1}$ by lemma \ref{lem:B1sharp}.
By the identity (\ref{eq:ufid}), we can transform
\begin{align}
\begin{split}
\beta(\hv^{i-1} \bn_r) &= \beta(\hv^{i-1} n_{r^{-1}} a_{r}^{-1}w  n_{r^{-1}})  \\
&= \psi(l_\s(\varpi^{i-1} r^{-1})) \beta(\hv^{i-1}a_{r}^{-1} w) \\
&= \beta(\hv^{i-1}a_{s}^{-1} w_m) \\
&= \kappa \beta(\hv^{i}a_s^{-1})
\end{split} \label{eq:trb}
\end{align} 
if $r = \dv\varpi^m s \in \dv\varpi^m (R_m^\times)^H$.
Therefore the second part equals $\kappa(q^2-q) c_{i}$.
Suppose that $m \ge 1$.
For a fixed $x \in \mathfrak{o}^\times$, 
\begin{align*}
\sum_{t \in \E} \beta(\hv^{i-1} \bn_{\dv\varpi^m t + x \dv^2 \varpi^{2m} \h})) = \kappa \sum_{t \in \E} \beta(\hv^{i} a_{1+ \dv t \varpi^m\h})  .
\end{align*} 
by (\ref{eq:trb}).
This is zero by Lemma \ref{lem:unitorder} ii), and Lemma \ref{lem:K2m+1sharp}.
Hence by lemma \ref{lem:B1sharp} the third part equals $-\kappa (q-1) c_i$.
Suppose that $m = 0$, and $E$ is in the case U-i).
By (\ref{eq:ufid})
\begin{align*} 
\sum_{t \in \E^\times, x \in \F^\times} \beta(\hv^{i-1} \bn_{x\h} \bn_{t}) &=
\sum_{t \in \E^\times, x \in \F^\times} \beta(\hv^{i-1} \bn_{ x\h} n_{- t^{-1}} a_{t}^{-1}w n_{-(t)^{-1}}) \\
& = \kappa \sum_{t \in \E^\times, x \in \F^\times} \beta(\hv^{i}
\begin{bmatrix}
1  & (\varpi)^{-1}t \\
\varpi x \h & 1+ x \h t
\end{bmatrix}).
\end{align*} 
Write $t$ as in (\ref{eq:tO}). 
Each term in the last sum is equal to 
\begin{align*}
\begin{cases}
\beta(\hv^{i} \begin{bmatrix}
1 & - y & &  \\
& x^{-1} & & \\
& & x& xy \\
& & & 1  \\
\end{bmatrix}s_\p) & \mbox{if $1 + z x = 0$}, \\
 \beta(\hv^i a_{v^\dag}) & \mbox{otherwise.} 
\end{cases}
\end{align*} 
where $v = 1 + x \h t$.
By Lemma \ref{lem:unitorder} ii), and Lemma \ref{lem:K2m+1sharp}, for a fixed $x \in \F^\times$,
\begin{align*}
\sum_{t \in \E^\times, z \neq -x^{-1}} \beta(\hv^{i-1} \bar{n}_{x\h + t}) &= \kappa (- \beta(\hv^i) + \sum_{t \in \E, z \neq -x^{-1}} \pi(a_{(1+ x \h t)^\dag}) \beta(\hv^i) ) \\
&=  -\kappa \beta(\hv^i).
\end{align*}
Since $\pi(\hv s_\p) \beta$ is invariant under $\Gamma \subset K_1^\sharp \la \hv s_\p \ra$, we have 
\begin{align} 
\sum_{x \in \F^\times, t \in \E^\times, z = -x^{-1}} \beta(\hv^{i-1} \bar{n}_{ x\h + t}) = 0 \label{eq:thirdHt0}
\end{align}
by Lemma \ref{lem:pcp}.
Now we have showed that the third part equals $- \kappa (q-1)c_i$ also in the case.
We obtain a recursion formula: 
\begin{align}
\lambda_\kappa c_i = q^{3} c_{i+1} +c_{i-1}, \ \  i \ge 1 \label{eq:recfK2m+1sharp}
\end{align} 
with $\lambda_\kappa = \lambda- \kappa(q-1)^2$.
It follows that 
\begin{align}
Z(s,\beta) = \frac{P(X)}{f_{\lambda_\kappa}(X)}, \ P(X) = c_0 + q c_1 X, \ f_{\lambda_\kappa}(X) = 1- \frac{\lambda_\kappa X}{q^2} + \frac{X^2}{q}. \label{eq:fl}
\end{align} \\
{\bf Case R).}
The coset space (\ref{eq:ess1}) is identified with $(\varrho^{m-1}R_m)^H/ (\varrho^{m+1} R_m)^H$, which is isomorphic to $(\mathfrak{O}/\p\mathfrak{O}) \varrho^{m-1} \op \F \varpi^{m}\h$ by Lemma \ref{lem:R^0}.
We divide the heart of $T\beta$ according to the partition:
\begin{align*}
(\mathfrak{O}/\p\mathfrak{O}) \varrho^{m-1} \op \F \varpi^{m}\h = \F \varpi^{m}\h \sqcup (\F^\times \varrho^{m} \op \F \varpi^m \h) \sqcup ((\mathfrak{O}/\p \mathfrak{O})^\times \varrho^{m-1} \op \F \varpi^m \h).
\end{align*}
We will compute the heart for $i \ge 1$.
The first part of the heart equals $c_{i-1}$ by lemma \ref{lem:B1sharp}.
When $m \ge 2$, the second part is zero by Lemma \ref{lem:K2m+1sharp} and (\ref{eq:trb}).
When $m =1$,  
\begin{align*}
\sum_{t \in \F^\times, x \in \F} \beta(\hv^{i-1} \bar{n}_{x \varpi \h + \varrho t}) &=  \sum_{t \in \F^\times, x \in \F} \beta(\hv^{i-1} \bar{n}_{x \varpi \h}n_{(\varrho t)^{-1}} w_1 n_{(\varrho t)^{-1}} ) \\
&= (q-1) c_{i-1}^* + \sum_{t \in \F^\times, x \in \F^\times} \beta^*(\hv^{i-1} \bn_{x\varpi\h} n_{\varrho^{-1}t}) \\
& = (q-1) c_{i-1}^* + \sum_{t \in \F^\times, x \in \F^\times}  \beta^*(\hv^{i-1}
\begin{bmatrix}
1 & \varrho^{-1}t  \\
x \varpi \h & 1 + x \h \varrho t
\end{bmatrix}).
\end{align*} 
Write 
\begin{align*}
\varrho t = \begin{bmatrix}
y&ze  \\
z &y
\end{bmatrix}, 
\end{align*} 
then  
\begin{align*}
\ 1 + x \h \varrho t = \begin{bmatrix}
 1 + xz & x y  \\
& 1 
\end{bmatrix}.
\end{align*} 
Therefore, the term in the last sum is equal to 
\begin{align*}
& \beta^*(\hv^{i-1} \begin{bmatrix}
1& & \varpi^{-1} y & \varpi^{-1} z e \\
& 1 & -(\varpi x)^{-1}& \varpi^{-1} y \\
& x \varpi & & x y  \\
& & & 1\\
\end{bmatrix})\\
& = \beta^*(\hv^{i-1} 
\begin{bmatrix}
1& & & \varpi^{-1} z e \\
&1 & \varpi^{-1} x^{-1} & \\
& &1 &  \\
& & &1  \\
\end{bmatrix}
\begin{bmatrix}
1& & \varpi^{-1} y &  \\
&  & -(\varpi x)^{-1}&  \\
& x \varpi & & x y  \\
& & & 1\\
\end{bmatrix})\\
&= 
\psi(2 \varpi^{i-2} z e)\beta^*(\hv^{i-1} \begin{bmatrix}
1 & -y & &  \\
& x^{-1} & & \\
& & x& xy \\
& & & 1  \\
\end{bmatrix}s_\p) \\
&= \beta^*(\hv^{i-1} \begin{bmatrix}
1 & -y & &  \\
& x^{-1} & & \\
& & x& xy \\
& & & 1  \\
\end{bmatrix}s_\p)
\end{align*} 
if $z = -x^{-1}$, and 
\begin{align*}
\beta^*(\hv^{i-1}n_{t(u\varrho)^{-1}}a_{u^\dag}\bn_{(u)^{-1}x \varpi \h}) = \beta^*(\hv^{i-1}a_{u}), 
\end{align*} 
otherwise, where $u = 1+ x \h\varrho t \in R_1^\times$.
Since $\Gamma \subset K_3^{\sharp} \la \hv s_\p \ra$, by Lemma \ref{lem:pcp}, it holds that 
\begin{align*}
\sum_{x \in \F^\times, t \in \F^\times, z = -x^{-1}} \beta^*(\hv^{i-1} \bn_{x\varpi\h} n_{\varrho^{-1}t})= 0, \ \ i \ge 1.
\end{align*} 
Since $\B_{2m+1}$ is the zero space, we have $\sum_{t \in \F, z \neq -x^{-1}} \pi(a_{1+ x \h\varrho t}) \beta^* = 0$ for a fixed $x \in \F^\times$, and hence
\begin{align*}
\sum_{x \in \F^\times, t\in \F^\times, z \neq -x^{-1}}\beta^*(\hv^{i-1} \bn_{x\varpi\h} n_{\varrho^{-1}u}) = -(q-1)c^*_{i-1}, \ \ i \ge 1.
\end{align*} 
Thus the second part of the heart is zero also in the case $m = 1$.
Each term in the third part of the heart equals $c_i^*$ by (\ref{eq:trb}) when $m \ge 2$.
When $m = 1$, we transform it by (\ref{eq:ufid}) 
\begin{align*}
& \sum_{t \in \mathfrak{O}^\times/\p \mathfrak{O}, x \in \F} \beta(\hv^{i-1} \bar{n}_{x \varpi \h}n_{t^{-1}} w n_{t^{-1}} ) \\
&= (q^2-q) \beta(\hv^{i-1}w) + \sum_{t \in \mathfrak{O}^\times/\p \mathfrak{O}, x \in \F^\times} \beta(\hv^{i-1} \bar{n}_{x\varpi\h} n_{t}w n_t).
\end{align*} 
Set 
\begin{align*} 
u = 1+ x\varpi \h t \in R_2^\times,
\end{align*}
and $v = t u^{-1}, x' = u^{-1}x \varpi \h$.
When $i \ge 1$, 
\begin{align*}
\beta(\hv^{i-1} \bar{n}_{x\varpi\h} n_{t}w n_t) &= \beta(\hv^{i-1}
\begin{bmatrix}
1 & t \\
x \varpi \h & 1 + x\varpi \h t
\end{bmatrix}w) \\
& = \beta(\hv^{i-1} n_v a_{u^\dag}wn_{x'}) \\
&= \beta(\hv^{i-1} w) = c_i^*.
\end{align*}
Therefore the third part equals $q^2(q-1) c^*_i$ in any case.
We obtain a recursion formula:
\begin{align*}
\lambda c_i = q^{3} c_{i+1} + c_{i-1} + q^2(q-1) c_{i}^*, \ \ i \ge 1.
\end{align*} 
For $\beta^*$, we consider the coset space $\bar{N}_{K^*}\la\hv^{-1}\ra/ \bar{N}_{K^*}$.
This is identified with $(\varrho^{m-2}R_{m}^\times)^H/(\varrho^{m} R_{m})^H$ which is isomorphic to $(\mathfrak{O}/\p\mathfrak{O})\varrho^{m-2} \op \F \varpi^{m-1}\h$.
Divide the heart of $T^* \beta^*$ according to the partition: 
\begin{align*}
(\F \varrho^{m-1} \op \F \varpi^{m-1} \h) \sqcup ((\mathfrak{O}^\times/\p\mathfrak{O}) \varrho^{m-2} \op \F \varpi^{m-1} \h).
\end{align*} 
Similar to $\beta$, we calculate the first part is $ c^*_{i-1} + (q-1) c_i$ and the second part is zero if $i \ge 0$.
We obtain a recursion formula 
\begin{align*}
\lambda c_i^* = q^{3} c_{i+1}^* + c^*_{i-1} + (q-1) c_i, i \ge 0.
\end{align*}
Combining these formulas,
\begin{align*}
\begin{bmatrix}
f_\lambda(X) & a X \\
q^{-2}a X & f_\lambda(X)
\end{bmatrix}\begin{bmatrix}
Z(s,\beta)  \\
Z(s,\beta^*)
\end{bmatrix} = \begin{bmatrix}
(c_0 + q^{-1}a c^*_{-1})+ (a c^*_0 -\lambda q^{-2}c_0 + q c_1)X \\
(c^*_0 - q^{-3}\lambda c^*_{-1}) + q^{-1} c^*_{-1}X^{-1}
\end{bmatrix}
\end{align*} 
where $a = q-1$ and $f_\lambda(X)$ is the polynomial defined at (\ref{eq:fl}).
Therefore, $Z(s,\beta)$ and $Z(s,\beta^*)$ are rational functions in forms of 
\begin{align*}
\frac{Q(X)}{\Delta_\lambda(X)}, \ \ \frac{R(X)}{X\Delta_\lambda(X)}
\end{align*} 
respectively.
Here
\begin{align*}
\deg Q(X), \deg R(X) \le 3, \ \ \Delta_\lambda(X) = f_\lambda(X)^2 -q^{-2}a^2 X^2
\end{align*} 
with $f_\lambda$ defined in (\ref{eq:fl}).
This completes the computation for the case of $K = K_{2m+1}^\sharp$.

Now assume that 
\begin{quote}
(\ddag) $\pi$ has no nontrivial vector invariant under the subgroup
\begin{align}
\begin{bmatrix}
1 & &   \\
& SL_2(\mathfrak{o})&  \\
& & 1 
\end{bmatrix}. \label{eq:SL2o}
\end{align} 
\end{quote}
For example, all nongeneric supercuspidal representations satisfy this condition.
To see this, we introduce the subgroup $K_N(n)$ consisting of elements  
\begin{align*}
k \in \begin{bmatrix}
\mathfrak{o} & \p^N&\p^N & \p^{-n}  \\
\p^{n+N}& 1 + \p^N & \p^N & \p^N\\
\p^{n+N}& \p^N &1+  \p^N &  \p^N\\
\p^n &\p^{n+N} &\p^{n+N} & \mathfrak{o}  \\
\end{bmatrix}
\end{align*} 
with $\mu(g) \equiv k_{11}k_{44} - k_{14} k_{41} \equiv  1 \pmod{\p^{2N}}$ for a fixed integer $N$ and an arbitrary $n \ge N$.
This is a subgroup of the original paramodular group of level $\p^n$ in \cite{R-S}.
Any open compact subgroup containing (\ref{eq:SL2o}) can contain $K_N(n)$ for some $N$ and $n$.
By the proof of Lemma 3.1.2. of loc. cit, we have the following lemma.
\begin{lem}
For different integers $n,n' \ge N$, the subgroups $K_N(n)$ and $K_{N}(n')$ generate a subgroup containing $\Sp_4(F)$.
\end{lem}
\nid
Immediately follows the linear independence of $K_N(n_j)$-fixed vectors in $\pi \in \Ir(G)$ for different $n_j$'s, the analogue of Theorem 3.1.3 of loc. cit, if $\pi$ has no nontrivial $\Sp_4(F)$-invariant vector.
Therefore by the proof of Proposition 3.4.2. of loc. cit, we have: 
\begin{prop}
If $\pi \in \Ir(G)$ is nongeneric supercuspidal, then $\pi$ has no $K_N(n)$-invariant vector for arbitrary $N$ and $n (\ge N)$.
In particular, $\pi$ satisfies the condition (\ddag).
\end{prop}
\nid
The groups $K_2, K_3^{\flat}$ obviously contain a group conjugate to the subgroup (\ref{eq:SL2o}), if $E$ is in the case R-i).
In the case R-ii), they also contain a group conjugate by the element $u_\alpha$ (c.f. (\ref{eq:ualpha})) of the subgroup.
Hence, they do not define the strict minimal subspace in the case R.
If $K = K_1$ with $E$ in the case U-i), under the assumption (\ddag) we can compute $Z(s,\beta)$ as follows.
We may assume that $\beta \in \B_{1,\kappa}$.
By (\ddag), the heart of $T \beta$ is equal to 
\begin{align}
\begin{split}
\sum_{r \in R/\p R} \beta(\hv^{i-1}\bn_{r}) 
&= \sum_{x,y,z \in \F} \beta(\hv^{i-1}\begin{bmatrix}
1& & &  \\
&1 & & \\
x&y &1 &  \\
z&x & &1  \\
\end{bmatrix})\\
& = \sum_{x\in \F} \beta(\hv^{i-1}\bn_{x}J) \\
&= q \beta(\hv^{i-1}J)\\
&= q \kappa c_i
\end{split}
\label{eq:HTass}
\end{align} 
Therefore, we obtain a recursion formula:
\begin{align*}
q^3 c_{i+1} = (\lambda + q \kappa )c_i, \ \  i \ge 0. 
\end{align*} 
Further we can find that $\pi$ is not supercuspidal as follows.
\begin{prop}\label{prop:supdag}
Let $E$ be in the case U-i).
Assume (\ddag) and that $\pi$ is supercuspidal.
Then $\B_1 = \{ 0\}$.
\end{prop}
\begin{proof}
Assuming that there exists nontrivial $\beta' \in \B_{1,\kappa}$ for $\kappa \in \{ \pm \}$, we will derive a contradiction.
Let $M = K_1$.
Define Hecke operators $T_M^\pm$ similar to $T_K^\pm$.
Since they commute with $\pi(w_0')$, it holds that, if $T_M^+ \beta' =\mu \beta'$ for some $\mu \in \C$, then $T_M^- \beta' = \mu \beta'$.
We can assume $\beta' \in \B_{1,\kappa}$ is an eigenvector, and it holds that $\mu \beta'(\hv^i) = q^3 \beta'(\hv^{i+1})$ and that $\mu \beta'(\hv^i) = q \beta'(\hv^{i-1} J) = \kappa q \beta'(\hv^i)$ by the computation (\ref{eq:HTass}).
Now from the supercuspidality assumption, it follows that $\mu = 0$ and $\beta'(\hv^i) = 0, i \ge 0$.
But this conflicts to Theorem \ref{thm:paravn}.
\end{proof}
Summarizing, for unitary $\pi$, we have obtained the table of zetas of a strict newform $\beta \in \B_\s(\pi)$.
\begin{table}[H]
\caption{Zetas of strict newforms (Case U)}
\label{tab:ZU}
\begin{tabular}{lccc}\toprule
st. min. sp. & $Z(s,\beta)$ & $Z_m(s,\beta)$ & $Z_{m+1}(s,\beta)$ \\ \midrule
$\B_{2m(\ge 2)} $ & $1$ & $1/(1- X'^2)$  & $(1+ X^2)/(1- X'^2)$ \\ 
$\B_{2m+1 (\ge 3)}^\flat$ & $1$ &  0& $1$ \\ 
$\B_{2m+1 (\ge 3),\kappa}$ & 1 & $1/(1-\kappa X'^{-1})$ & $1/(1 -\kappa X')$ \\ 
$\B_{1,\kappa}$ (U-i) with (\ddag)) & $1/(1- (\lambda - \kappa q)q^{-2} X)$ & $Z/(1-\kappa X'^{-1})$ & $Z/(1 -\kappa X')$ \\ 
$\B_{2m+1,\kappa}^\sharp$ ($m \ge 1$, if U-ii)) & $P(X)/f_{\lambda_\kappa}(X)$ & $Z/(1-\kappa X'^{-1})$ & $Z/(1 -\kappa X')$ \\ 
\bottomrule
\end{tabular}
\end{table}
Here we denote $Z = Z(s,\beta), Z^* = Z(s,\beta^*)$, and normalize all zetas suitably.
\begin{table}[H]
\caption{Zetas of strict newforms (Case R)}
\label{tab:ZR}
\begin{tabular}{lccc}\toprule
st. min. sp. & $Z(s,\beta)$ & $Z_m(s,\beta)$ & $Z_{m+1}(s,\beta)$ \\ \midrule
$\B_{2}$ (\ddag) & --- & --- & --- \\ 
$\B_{3}^{\flat}$ (\ddag) & --- & --- & --- \\ 
$\B_{2m+1 (\ge 5)}^{\flat}$ & $1$ & 0 & $1$ \\
$\B_{2m+1(\ge 5)}$ & 1 & $ (1 + c^*_{-1}X^{-1})/(1- X')$ & $(1 + q^{-1}c^*_{-1})/(1- X')$ \\ 
$\B_{2m+1}^{\sharp}$ & $Q(X)/\Delta_\lambda(X)$ &$(Z + q Z^*)/(1 - X')$ &$(Z + X Z^*)/(1 - X')$  \\ 
\bottomrule
\end{tabular}
\end{table}
The following is an immediate consequence of (\ref{eq:FNEQP}), and the above tables.
\begin{prop}\label{prop:negflat}
Let $\pi$ be unitary with $L(s,\pi)^{-1} \in \C[X]$ having no sign.
Then $\B_{2m+1}^{\flat,\ep}$ is not the strict minimal space of sign $\ep$ for $m \ge 1$ (resp. $\ge 2$) if $E$ is in the case U (resp. R).
\end{prop}
Finally, we complete the proof of Theorem \ref{thm:unrcase} in the case U-ii) using a Hecke operator.
We do not assume the unitarity of $\pi$.
Assume that $K = K_1^\flat$.
Then, there exists a nontrivial eigenvector $\beta \in \B_1^{\flat,\ep}$ for $T_K^+$.
It is easy to see that $\lambda_+ \beta(\hv^{i}) = q^3 \beta(\hv^{i+1})$ where $\lambda_+$ indicates the eigenvalue for $T_K^+$.
It follows from Lemma \ref{lem:flatmin} and Corollary \ref{cor:Z} that
\begin{align}
P_1(X,\beta) = \frac{\beta(1) L(s,\pi)^{-1}}{(1+q^2)(1- \lambda_+ q^{-2}X)} \neq 0. \label{eq:U2flat}
\end{align} 
By Lemma \ref{lem:sdzp}, its diameter is $2 - n_\pi$. 
Since the degree of $L(s,\pi)^{-1} \in \C[X]$ is four when $\pi$ is unramified, both $\B_1^{\flat,\pm}$ are not the strict minimal spaces of unramified representations.
Hence, for the theorem, it suffices to show that there exists a nontrivial vector in $\B_1^{\flat,+}$ by Lemma \ref{lem:m>n}.
We will construct such a form from the unique $G(\mathfrak{o})$-invariant form $\beta_0 \in \B_\s(\pi)^{G(\mathfrak{o})}$.
\begin{lem}\label{lem:unrconst}
$\beta_0(1) \neq 0$.
\end{lem}
\begin{proof}
We can construct such a $\beta_0$ by the integral 
\begin{align*}
I_f(g) := \vol(E^\times/F^\times)^{-1}\int_{E^\times/F^\times} f(tg) d^\times t
\end{align*}
where $f$ is a $G(\mathfrak{o})$-invariant function on $G$. 
By an elementary computation, 
\begin{align*}
I_f(1) = f(1).
\end{align*} 
If $\pi$ is generic, then let $f$ be the Whittaker function $W_0$ given in 7.1. of \cite{R-S}.
If $\pi$ is in a SK-packet, then let $f$ be the function $\delta$ used in sect. \ref{sec:NFsp}.
\end{proof}
\noindent
For $\beta_0$ given in this lemma, observe the integral: 
\begin{align*}
\beta'& := \vol(K_1)^{-1} \int_{K_1} \pi(k)\beta_0 dk \\
& =  \vol(\mathfrak{O}^\times)^{-1} \vol(\mathfrak{O}/2)^{-1}\int_{2^{-1}\mathfrak{O}}\int_{\mathfrak{O}^\times} \pi(n_u a_t) \beta_0 dt du.
\end{align*}  
It is easy to see that $\beta'(1) = \beta_0(1) \neq 0$.
For $x \in \mathfrak{o}, u \in 2^{-1} \mathfrak{O}$ and $t \in \mathfrak{O}^\times$, we calculate
\begin{align*}
\bn_{4 x \h}\langle n_u a_t\rangle 
&=  \begin{bmatrix}
1 + 4 x u \h \langle t^c \rangle & - 4x u \h \langle t^c \rangle u  \\
4x \h \langle t^c \rangle & 1 - 4x \h \langle t^c \rangle u
\end{bmatrix} \\
& = n_{4x u \h \langle t^c \rangle u d^{-1}} \cdot a_{d'} \cdot \bn_{d^{-1}4x \h \langle t^c \rangle} \\& \in N_{K_1} \cdot A_{K_1} \cdot \bar{N}_{G(\mathfrak{o})}.
\end{align*}
Here $d = 1 - 4x \h \langle t^c \rangle u \in R^\times$.
Therefore, $\beta'$ is the desired nonzero vector in $\B_1^{\flat,+}$.
\section{Local newform (nonsplit case)}\label{sec:NFN}
Throughout this section, let $\tau \in \Ir'(\PGL_2(F))$.
Let $\pi$ be in the SK-packet of $\tau$.
Let $\s$ be a Hankel matrix in the form of (\ref{eq:defsigma}), and $E = E_\s$ be in the nonsplit case.
For $\beta \in \B_\s(\pi)$, let 
\begin{align*}
g_\beta(X) = \frac{Z(s,\beta)}{L^{reg}(s,\pi)} \in \C[X,X^{-1}]. 
\end{align*} 
We will drop the superscript $\ep$ from the paramodular space of sign $\ep$, when the sign is clear from the context.
For nontrivial formal Laurent series $f(X)$ and $g(X)$ in $\C((X))$, we will write 
\begin{eqnarray*}
f(X) \app g(X)
\end{eqnarray*}
if $f/g \in \C^\times$.
We will use the result 
\begin{align}
\frac{L^{reg}(s,\pi)}{L(s,\tau) }= \begin{cases}
(1-X)^{-1} & \mbox{if $\pi = SK(\tau)$,} \\
1 & \mbox{if $\pi = SK(\tau^{JL})$}
\end{cases} \label{eq:LregSK}
\end{align} 
(c.f. Table 5 of \cite{Sc-T}).
\subsection{$SK(\tau)$}\label{subsec:SKtau}
If $n_\tau = 1$, then $\tau$ is $St$ or $\chi St$, where $\chi$ is the nontrivial unramified quadratic character of $F^\times$, and therefore, the Waldspurger model $\Ts_{\1}(\tau)$ and the Bessel model $\B_\s(\pi)$ are zero by Corollary 4.7.2 of \cite{R-S2}, Proposition 1.7 of \cite{T}.
For this reason, we may assume that  
\begin{align}
\begin{split}
& n_\tau \ge 2, \ L(s,\tau) = 1, \\ 
& L^{reg}(s,\pi) = \frac{1}{1-X}, \  
Z(s,\beta) = \frac{g_\beta(X)}{1-X}.
\end{split} \label{eq:gbSKU}
\end{align} 
when $SK(\tau)$ is ramified.
For $\pi = SK(\tau)$, we will prove:
\begin{thm}\label{thm:SKrepn}
Let $\s$ be a Hankel matrix in the form of (\ref{eq:defsigma}), and $E = E_\s$ be in the nonsplit case.
Then, the followings are true.
\begin{enumerate}[i)]
\item Assume that $\B_\s(\pi) \neq \{ 0\}$.
In the following cases, the strict minimal subspace of sign $\ep_\tau$ is
\begin{align*}
\begin{cases}
\B_0^+ & \mbox{if $n_\tau = 0$}, \\
\B_{{n_\tau}-1,+}^{\sharp,\ep_\tau} & \mbox{if $n_\tau \ge 2$ is even, and $E$ is in the case U, }  \\
\B_{2n_\tau-1}^{\sharp,\ep_\tau} & \mbox{if $n_\tau \ge 2$, and $E$ is in the case R.} 
\end{cases}
\end{align*} 
In these cases, the strict minimal space is one-dimensional and spanned by $\beta$ with the properties    
\begin{align}
Z_{\frac{n_\tau}{\f}}(s,\beta) &\app L(s,\pi) (= L(s,\phi_\pi)), \label{eq:ZLSK} \\
\frac{Z(s,\beta)}{L^{reg}(s,\pi)} &\app 
\begin{cases}
1+ X'& \mbox{if $n_\tau = 0$, and $E$ is in the case U,} \\
1 & \mbox{otherwise.}
\end{cases}
\end{align}
\item There is no paramodular form of sign $-\ep_\tau$ in $\B_\s(\pi)$.
\item If $n_\tau$ is odd and $E_\s$ is in the case U, then $\B_\s(\pi) = \{0 \}$.
\end{enumerate}
\end{thm}
\nid
To prove this theorem, we use the operator $U_K$ defined at (\ref{eq:UKop}). 
When $K= K_{2m+1}^{\sharp}$, $U_{K}$ is an endomorphism of $\B_{2m+1}^{\sharp,\ep}$, and if $\beta \in \B_{2m+1}^{\sharp,\ep}$ has $Z(s,\beta) = c_0 + c_{1}X + \cdots$(c.f. (\ref{eq:ZbZ})), then 
\begin{align*} 
Z(s,U_K^j \beta) \approx c_{j} + c_{j+1} X +  c_{j+2} X^2 + \cdots.
\end{align*} 
By using such an operator, we can find there exists a paramodular form of sign $\ep_\tau$ in $SK(\tau)$ as follows.
\begin{lem}\label{lem:inflem}
Assume that $\tau$ is ramified, and $\B_\s(\pi) \neq \{0\}$.
Then there exists a paramodular form $\beta$ of sign $\ep_\pi (= \ep_\tau)$ such that $Z(s,\beta) = L^{reg}(s,\pi)$.
\end{lem}
\begin{proof}
By (\ref{eq:gbSKU}) and definition of $L^{reg}(s,\pi)$, there exists a $\beta \in \B_\s(\pi)$ such that
\begin{align*}
\beta(\hv^j) = q^{-j}, \ \ j >>0.
\end{align*} 
We may assume that $\beta$ has a sign, say $\ep$.
By the proof of Proposition \ref{prop:extpf}, and the operator $U_{K_{2m+1}^\sharp}$ for a sufficiently large $m$, we may assume that $\beta$ lies in $\B_{2m+1}^{\sharp,\ep}$ and $Z(s,\beta) = (1-X)^{-1}$.
From (\ref{eq:PSzeta}), it follows that there exists a polynomial $P(X)$ in $X$ such that 
\begin{align*}
P_m(X,\beta) = 1+ q^\f X^{-1} P(X), \ P_{m+1}(X,\beta) = 1+ X^{\f-1} P(X).
\end{align*} 
By (\ref{eq:FNEQP}), 
\begin{align*}
1+ q^\f XP(X^{-1}) &= \ep\ep_\pi X^{n_\pi-\f m}(1+ q^\f X^{-1}P(X)), \\
1+ X^{1-\f} P(X^{-1})&= \ep\ep_\pi X^{n_\pi-\f (m+1)}(1 + X^{\f-1} P(X)).
\end{align*} 
Substituting $X =1$, we find that $\ep = \ep_\pi$.
\end{proof}
\noindent
For $SK(\tau)$, we have a slight better estimation than (\ref{eq:m>npre}):
\begin{lem}\label{lem:m>nSK}
\begin{align}
m_\pi^\ep \ge \frac{n_\pi}{\f} - 1. \label{eq:m>n}
\end{align}
\end{lem}
\begin{proof}
When $E$ is in the case U, this is same as (\ref{eq:m>npre}).
Suppose that $E$ is in the case R.
It suffices to derive a contradiction assuming that $P_{m_\pi^\ep}(X,\beta) \app X^{-1}$ and $P_{m_\pi^\ep+1}(X,\beta) =0$ for a strict newform $\beta$ (only in this case it holds that $m_\pi^\ep = n_\pi-2$.).
In this case, $Z(s,\beta) \app L(s,\pi)$ by Lemma \ref{lem:PSzeta} ii).
This is a contradiction, since we have seen that $L(s,\pi) \neq L^{reg}(s,\pi)$ at \ref{sec:packet}.
\end{proof}
\noindent
We will prove the theorem.
Note that the strict minimal space of sign $\ep$ is not 
\begin{align*}
\B_{2m+1}^{\flat,\ep} \ \mbox{for $m \ge $}  \ 
\begin{cases}
1& \mbox{if $E$ is in the case of U}, \\
2& \mbox{if $E$ is in the case of R}
\end{cases}
\end{align*}
by Proposition \ref{prop:negflat}.

\paragraph{\bf Proof of i)}
When $n_\tau = 0$, this is a case of Theorem \ref{thm:unrcase}.
We will prove for ramified $\tau$.
First, assume that $\pi = SK(\tau)$ is unitary.
We can apply the Hecke theory.
Let $V$ and $m$ denote the strict minimal space and level of sign $\ep_\pi = \ep_\tau$, respectively. 
Let $\beta$ be a nontrivial form of $V$, and abbreviate the zeta polynomial $P_n(X,\beta)$ to $P_n$.
We will prove i) according to the case of $E$. \\
{\bf Case U)}
Consider the situation where $n_\tau \ge 3$.
By (\ref{eq:m>n}), $m \ge 1$.
Assume that $V= \B_{2m}$, then by the table $P_m \app (1-X)/(1 + X')$, which is not a polynomial in $X^\pm$, a contradiction.
Assume that $V = \B_{2m+1}$.
We may assume $\beta \in \B_{2m+1,\kappa}$.
By Table \ref{tab:ZU},
\begin{align}
P_{m+1} \app \frac{(1-X)(1-X')}{1- \kappa X'}. \label{eq:SKU2m+1}
\end{align}
Therefore $\kappa$ is plus, and $P_{m+1} \app 1- X$, having sign minus conflicting to Lemma \ref{lem:sdzp}.
Hence, $V = \B_{2m+1}^\sharp$.
We may assume $\beta \in \B_{2m+1, \kappa}^{\sharp}$ with $\kappa \in \{ \pm \}$.
By Theorem \ref{thm:paravn}, $Z(s,\beta) \neq 0$.
By Hecke theory, $Z(s,\beta)$ is in the form of $P(X)/f(X)$ where $P, f$ are polynomials in $X$ with $\deg P \le 1, \deg f = 2$.
Taking (\ref{eq:gbSKU}) into account, we conclude that $g_\beta(X)$ is a nonzero constant, and $\deg P = 1$.
Thus we have 
\begin{align}
Z(s,\beta) \app L^{reg}(s,\pi) \label{eq:ZL^reg}
\end{align} 
and $P_{m+1} \app (1 - X')/(1 -\kappa X')$ by the table.
Hence
\begin{align}
\beta \in \B_{2m+1,+}^\sharp,\ \  P_{m+1} \app 1. \label{eq:Pm+11}
\end{align}
By Lemma \ref{lem:sdzp} again, 
\begin{align} 
m = \frac{n_\pi}{2} -1. \label{eq:nmrel}
\end{align} 

Consider the situation where $n_\tau = 2$, and $E$ is in the case U-ii).
If we assume that $m \ge 1$, then by the above argument, $V = \B_{2m+1}^\sharp$ and $m = 2/2-1 = 0$, a contradiction.
Hence $m = 0$, and (\ref{eq:nmrel}) also holds.
Assume that $V = \B_{1}$.
We may assume $\beta \in \B_{1,\kappa}$.
Then 
\begin{align}
P_1 \app \frac{(1-X)(1- X')Z(s,\beta)}{1- \kappa X'}, \label{eq:SKNFU}
\end{align} 
which is not a constant conflicting to Lemma \ref{lem:sdzp}. 
Assume that $V=  \B_0$. 
Then $P_0$ is a nontrivial element of $\C[X]$ by Theorem \ref{thm:paravn}.
But by the same Lemma, $dia P_0$ is $-2$, a contradiction.
Hence $V = V_{1}^\sharp$.
We may assume $\beta \in \B_{1,\kappa}^\sharp$.
It holds (\ref{eq:SKNFU}) also in this situation.
Since $P_1$ is a nonzero constant by the same lemma, we conclude (\ref{eq:ZL^reg}) and (\ref{eq:Pm+11}) also in this situation.

Consider the situation where $n_\tau = 2$, and $E$ is in the case U-i).
Similar to the above situation, we can conclude that $m = 0$.
We find $V = V_{1}^\sharp$ by the following lemma.
\begin{lem}\label{lem:Sglind}
If  $n_\tau \ge 2$, then $I(\tau)^{\Gamma_0(\p)} = \{ 0\}$.
\end{lem}
\begin{proof}
We find $\Sigma := \{1, s, w_1, s w_1\}$ is a complete system of representatives for $P\bs G/ \Gamma_0(\p)$,  by Lemma 5.1.1. of \cite{R-S}, and the Bruhat decomposition: 
\begin{align*}
P \bs G / \Gamma_0(\p) \simeq P(\mathfrak{o}) \bs G(\mathfrak{o})/ \Gamma_0(\p) \simeq P(\F)\bs G(\F) /P(\F). 
\end{align*} 
Let $\xi \in I(\tau)^{\Gamma}$.
By definition of $I(\tau)$, if $\xi(r) \neq 0$ for an element $r \in \Sigma$, then for all $h \in \GL_2(F)$ and $u \in F^\times$ such that
\begin{align*}
\begin{bmatrix}
u h  & * \\
& h^\dag
\end{bmatrix} \in \Gamma_0(\p) \la r \ra,
\end{align*} 
we must have $\tau(h)|u| = 1$.
At least, $\tau$ has a nontrivial vector invariant under the Hecke subgroup (of $\GL_2(F)$) of level $\p$.
This conflicts to the newform theory for $\GL_2(F)$.
Hence $\xi$ is identically zero.
\end{proof}
\nid
Now (\ref{eq:ZL^reg}), (\ref{eq:Pm+11}) follow similar to the above situation.

The one-dimensionality of $V$ follows from (\ref{eq:ZL^reg}) and Theorem \ref{thm:paravn}, and the proof for the case U is completed. \\
{\bf Case R)}
Consider the situation where $n_\tau \ge 3$.
It follows from (\ref{lem:m>n}) that $m \ge 2$.
Assume that $V = \B_{2m+1}$.
By the table, $P_{m+1}$ is a constant multiple of $(1-X)$.
Since its sign is plus by Lemma \ref{lem:sdzp}, $P_{m+1} =0$.
By Lemma \ref{lem:PSzeta} ii) and the table, $P_{m} \app 1-X^{-1}$, having sign minus, a contradiction.
Hence 
\begin{align}
 V= \B_{2m+1}^{\sharp}. \label{eq:VB2m+1sharpR}
 \end{align}
By Hecke theory, $Z(s,\beta)$ is in the form of $Q(X)/\Delta(X)$ with $\deg Q \le 3, \deg \Delta = 4$.
Taking (\ref{eq:gbSKU}) into account, we conclude that $g_\beta(X)$ is a nonzero constant.
Thus $Z(s,\beta) \app (1-X)^{-1}$.
Similarly, we conclude that $Z(s,\beta^*) \approx X^{-1}(1-X)^{-1}$.  
Viewing 
\begin{align*}
 Z_m(s,\beta) &\app \frac{Z(s,\beta)+ q Z(s,\beta^*)}{1-X'} \\
 Z_{m+1}(s,\beta) &\app \frac{Z(s,\beta)+ X Z(s,\beta^*)}{1-X'}, 
 \end{align*}
 and Lemma \ref{lem:sdzp}, we deduce that $P_{m} \app 1+ X^{-1}$, and that there exists a nonzero constant $c$ such that
\begin{align}
Z(s,\beta) = \frac{c}{1-X}, \ Z(s,\beta^*) = \frac{c}{qX(1-X)}. \label{eq:ZZ*RSK}
\end{align} 
It follows that  
\begin{align}
P_{m+1} \app 1 \label{eq:Pm+11RSK}
\end{align}
and that $\dim V = 1$ from Theorem \ref{thm:paravn}.
It follows from the same lemma that 
\begin{align}
m = n_\pi - 1. \label{eq:mnSKR}
\end{align} 
Consider the situation where $n_\tau = 2$.
From the above argument, we can deduce $m \le 1$ and $V \neq \B_3$.
Indeed, even if one assumes that $V= \B_{2m+1}^{\sharp}$ with $m \ge 2$, then it follows that $P_m \app 1+ X^{-1}$, and that $-1 = m - 2$ by Lemma \ref{lem:sdzp}, a contradiction.
Hence (\ref{eq:mnSKR}) holds in this situation also.
Assume that $V = \B_2$.
By (\ref{eq:gbSKU}) and Lemma \ref{lem:PSzeta} iii), $P_1 = g_\beta(X)$.
But its diameter is nonnegative, conflicting to Lemma \ref{lem:sdzp}.
Assume that $V = \B_3^{\flat}$.
By Lemma \ref{lem:flatmin} and Lemma \ref{lem:sdzp}, $P_2(X) \app 1$.
From Lemma \ref{lem:PSzeta} i) it follows that $g_\beta(X) \app (1-X')^{-1}$, a contradiction.
Hence (\ref{eq:VB2m+1sharpR}) holds in this situation also.
Assume that $P_2 = 0$.
Then $P_1 \neq 0$ by Corollary \ref{cor:Z}.
By Lemma \ref{lem:sdzp}, \ref{lem:PSzeta} ii), $P_1 \app 1+X^{-1}$, and $g_\beta(X) \app (1+X)/((1-X)(1-X'))$, a contradiction.
Hence $P_2 \neq 0$.
From Lemma \ref{lem:sdzp}, it follows (\ref{eq:Pm+11RSK}) in this situation also.
Similarly, one can prove $P_1(X) \app 1 + X^{-1}$.
In particular, from (\ref{eq:ZbZ}) it follows that 
\begin{align*}
Z(s,\beta^*) \in X^{-1}\C[[X]] \setminus \C[[X]].
\end{align*} 
This implies $\dim V = 1$ by Theorem \ref{thm:paravn}.
This completes the proof of i) for unitary ramified $\pi$.

Next, we apply Robert and Schmidt argument for nonunitary $\pi = SK(\tau)$.
Such a $\pi$ is given by a principal series $\tau = \chi \times \chi^{-1}$ with exponent $e(\chi) \neq 0$, and $\pi$ is the Siegel induction $S(\chi): = \chi \1_{GL(2)} \rtimes \chi^{-1}$ (c.f. 5.5. of \cite{R-S}).
Here $e(\chi)$ is defined by $|\chi(x)| = |x|^{e(\chi)}, x \in F^\times$.
By definition, if a $\C$-valued function $f \in S(\chi)$ is invariant under a compact subgroup $K' \subset G$, then $f$ is determined by its values at $r \in P \bs G /K'$, and we must have $f(rk) = f(r)$ for all $r$ and $k \in K'$.
In particular, it holds that 
\begin{align}
\chi(u^{-1}\det(h)) = 1 \label{eq:condind}
\end{align} 
for any element 
\begin{align}
h \rt u  := \begin{bmatrix}
h&  \\
& u h^{\dag} 
\end{bmatrix}
\end{align} 
in $P \cap K' \la r \ra$.
However, since $K' \la r\ra$ is compact, it holds that $u^{-1}\det(h) \in \mathfrak{o}^\times$ for any $h \rtimes u \in P \cap K' \la r \ra$, and the condition (\ref{eq:condind}) is same as for the unitary representation $S(\chi_1)$, where $\chi_1 = \chi/|\chi|$.
Now let $K'$ be the paramodular group $K$ defining $V$.
We have showed that $K$-invariant function $f \in S(\chi_1) (\simeq \pi)$ with $\pi(\imath)f = \ep_\pi f$ is unique up to scalars, and so is that in $S(\chi)$.
Therefore, $V$ is one-dimensional.
In particular, the action of $T_K$ on $V$ is same as for the unitary case, and so is the proof for this case.\\

\paragraph{\bf Proof of ii)}
Assuming that there exists the strict minimal space $V'$ of sign $-\ep_\tau$, we will derive contradictions.
By the last argument of i), it suffices to treat the unitary case.
Let $m'$ denote the principal level of $V'$.
Let $\beta$ be nontrivial form of $V'$, and abbreviate $P_n(X,\beta)$ to $P_n$.
We will prove ii) according to the case of $E$. \\
{\bf Case U)}
Consider the situation where $n_\tau > 2$.
Similar to i), we can conclude that $V' \neq \B_{2m'}$, and that $m' \ge 1$.
Assume that $V' = \B_{2m'+1}$, then by the table (\ref{eq:SKU2m+1}) also holds, and hence $P_{m' +1} \app (1-X)$.
But, by Lemma \ref{lem:sdzp}, $2m' - 2m_\pi^{\ep_\pi} = 1$, a contradiction.
Assume that $V' = \B_{2m'+1}^\sharp$.
Then similar to i), we can conclude that $P_{m'+1} \app 1$, conflicting to the same lemma.

Consider the situation where $n_\tau = 2$. 
From the above argument we deduce that $m' = 0$.
But, by the argument of i), $P_1 \neq 0$.
But this conflicts to Lemma \ref{lem:sdzp}.

Consider the situation where $n_\tau = 0$.
Assume that $m' \ge 1$.
Assume that $V=  \B_{2m'}$.
By the table, 
\begin{align*}
P_{m'} \app \frac{1-X}{L(s,\tau)(1-X')}.
\end{align*} 
By Lemma \ref{lem:sdzp}, $P_{m'}$ should be a polynomial in $X$ of diameter $2m'$ of sign minus.
But it is impossible, since $L(s,\tau)^{-1}$ is in the form of $(1-a X)(1- a^{-1}X')$.
Assume that $V = \B_{2m'+1}$, and $\beta \in \B_{2m'+1,\kappa}$.
By the table, 
\begin{align*}
P_{m'+1} \app \frac{(1-X)(1-X')}{L(s,\tau)(1- \kappa X')}.
\end{align*} 
So, ${\rm dia} P_{m'+1} = 3$, conflicting to Lemma \ref{lem:sdzp}.
Assume that $V'= \B_{2m'+1}^{\sharp}$ and $\beta \in \B_{2m'+1,\kappa}^\sharp$.
By Hecke theory, $P_{m'+1}$ is a polynomial in $X$ in the form of 
\begin{align*}
\frac{P(X)(1-X)(1-X')}{L(s,\tau)f(X) (1 -\kappa X')}
\end{align*} 
where $P,f \in \C[X]$ wiith $\deg P \le 1, \deg f = 2$.
Hence $\deg P_{m'+1} \le 2$.
But this conflicts to Lemma \ref{lem:sdzp}.
Thus $m' \neq 1$.
Assume that $m' = 0$.
However, since $\imath$ is contained in the paramodular groups of principal level $0$ except for $K_1^\sharp$, we may assume that $V' = \B_{1}^\sharp$, and $\beta \in \B_{1,\kappa}^\sharp$ for some $\kappa$.
By Lemma \ref{lem:sdzp}, it must be hold that $P_1 \app (1-X^2)$.
Therefore  
\begin{align*}
g_\beta(X) \app \frac{(1+ X)(1- \kappa X')(1-X)}{1-X'}.
\end{align*} 
Therefore $\kappa$ is plus, and $g_\beta(X) \app (1+X)(1-X)$.
Now it is possible to derive a contradiction using the operator $U_{K_1^\sharp}$.
This completes the proof for the case U.\\
{\bf Case R)}
Since $\imath$ is contained in $K_3$, we may assume that 
\begin{align*}
m' \ge 2, \ \mbox{or} \ V' = \B_3^{\sharp}.
\end{align*} 

Consider the situation where $n_\tau \ge 2$.
Assume that $V' =\B_{2m+1}^\sharp$.
By the similar argument to i), we deduce that $P_{m'} \app 1 -X^{-1}$ and there exists a nonzero constant $c$ such that 
\begin{align*}
Z(s,\beta) = \frac{c}{1-X}, \  Z(s,\beta^*) = \frac{-c}{qX(1-X)}.
\end{align*}  
Then, $P_{m'+1} \app 1-X'$ by the table.
But, this has no sign, a contradiction.
Assume that $V' = \B_{2m'+1}$.
Assume that $P_{m'+1} \neq 0$.
By the table, $P_{m'+1} \app 1-X$.
By Lemma \ref{lem:sdzp}, $m' = n_\pi$, and $P_{m'} = 0$, conflicting to the table.
Assume that $P_{m'+1} = 0$, then $c_{-1}^* = -q$ by the table, and $P_{m'} \app (1-X'^{-1})(1-X)$, having no sign, a contradiction.

Consider the situation where $n_\tau = 0$.
Assume that $V' = \B_{2m'+1}$.
Then $P_{m'+1}$ is a constant multiple of $(1-X)/L(s,\tau)$.
Since $(1-X)/L(s,\tau)$ has no sign, $P_{m'+1}= 0$.
By Lemma \ref{lem:PSzeta} ii) and Corollary \ref{cor:Z}, $P_{m'} \app X^{-1}(1-X)(1-X')/L(s,\tau)$, having no sign, a contradiction. 
Assume that $V' = \B_{2m'+1}^{\sharp}$.
By Hecke theory, $P_{m'}$ is in the form of
\begin{align*}
\frac{(1-X)(XQ(X) + q^{-1}R(X))}{L(s,\tau)X \Delta(X)}
\end{align*} 
where $\deg Q,\deg R \le 3$, and $\Delta(X) \in \C[X] \setminus X \C[X]$ with $\deg \Delta =4$.
Therefore, $XP_{m'}$ is a polynomial in $X$ of degree $\le 3$.
Similarly $P_{m'+1}$ is a polynomial in $X$ of degree $\le 2$.
From Lemma \ref{lem:sdzp}, it follows that $m' = 1$ and that there exist some $a,b,c \in \C$ such that
\begin{align*} 
P_{1}(X)= a(X^{-1}-X^2) + b(1-X), P_{2}(X) = c (1- X^2).
\end{align*}
Assume that $c = 0$.
By Lemma \ref{lem:PSzeta} ii), 
\begin{align*}
g_\beta(X) \app \frac{a(1-X^3) + b (X -X^2)}{1-X'}.
\end{align*} 
Hence $a(1-q^3)+ b(q-q^2) = 0$, and 
\begin{align} 
Z(s,\beta) \app (1-q X)L(s,\tau). \label{eq:SKiiRZ}
\end{align}
Now, consider $U_{K'} \beta$ for the compact subgroup 
\begin{align*}
K' := \begin{bmatrix}
R_2 & R  \\
\varpi R & R_2
\end{bmatrix}.
\end{align*} 
Obviously $\pi(\imath)U_{K'} \beta = - U_{K'} \beta$.
Taking (\ref{eq:SKiiRZ}) into account, we find $Z(s,U_{K'} \beta)$ is nontrivial. 
But, this conflicts to the following lemma.
\begin{lem}\label{lem:unrKminus}
If $\tau$ is unramified, then $\pi = SK(\tau)$ has no $K'$-invariant vector of sign minus.
\end{lem}
\begin{proof}
The Hecke subgroup $\Gamma_0(\p)$ also decomposed as in (\ref{eq:decK}).
Therefore, by (\ref{eq:P2T}), the set $\Sigma$ in the proof of Lemma \ref{lem:Sglind} is also a complete system of representatives for $P \bs G /K'$.
For all $r \in \Sigma$, it holds that $r \la \imath \ra = \pm r $.
Let $\chi$ be an unramified character such that $\tau \simeq \chi \times \chi^{-1}$. 
Let $\xi \in \pi = S(\chi)$ be a vector.
If $\xi$ is of sign minus and invariant under $K'$, then since $\chi(-1) = 1$, it holds that 
\begin{align*}
\xi(r) = \chi(-1) \xi(r) = \xi(\imath r) = \xi(r\la \imath \ra \imath) = \xi(r \imath) = -\xi(r)
\end{align*} 
for all $r \in \Sigma$, and therefore $\xi$ is identically zero.
This completes the proof.
\end{proof}
\nid
Hence $c \neq 0$, and $Z(s,e_1 \beta) \neq 0$.
It follows from Lemma \ref{lem:sdzp} that $P_2(X, e_1 \beta) \app (1-X^2)$, and 
\begin{align*} 
Z(s,e_1 \beta) \app (1+X)(1-X')L(s,\tau).
\end{align*}
Obviously, this is not a polynomial in $X$, and therefore $Z(s,U_{K'} e_1\beta) \neq 0$.
But, $U_{K'} e_1\beta$ is $K'$-invariant.
This conflicts to Lemma \ref{lem:unrKminus} again.
This completes the proof of ii).

iii) follows from (\ref{eq:nmrel}) and ii). 
\subsection{$SK(\tau^{JL})$}\label{subsec:SKtauJL}
Let $\tau$ be a discrete series.
Then $\tau$ is a supercuspidal representation, or $St$, or its twist $\chi_L St$ by the quadratic character $\chi_L$ associated to a quadratic extension $L$ of $F$. 
Let $D$ denote the division quaternion algebra defined over $F$, and $\tau^{JL} \in \Ir(D)$ the Jacquet-Langlands transfer.
Let $\pi = SK(\tau^{JL})$.
Combining the results of \cite{G-T},\cite{G-T2} and \cite{R-S2}, we obtain the following table.
\begin{table}[H]
\caption{classification of $SK(\tau^{JL})$}
\label{tab:SKJL}
\begin{tabular}{llcccc}\toprule
$\tau$ & $\tau^{JL}$  & cond of $E$ & type of $\pi$ & $L(s,\phi_\pi)^{-1}$ & $N_\pi$ \\ \midrule
$St$ & $\1_{D^\times}$ & any & VIb  & $(1-X')^2$ & $2$  \\ 
$\chi_L St$(U) & $\chi_L \circ N_{D/F}$ & $E= L$ &  Va* & $1-X'^{2}$ & $2$ \\ 
$\chi_L St$(R) & $\chi_L \circ N_{D/F}$  & $E= L$ &  Va* & $1-X'$ & $2n_{\chi_L}+ 1$ \\ 
s.c & $\tau^{JL}$ & $\Ts_E(\tau) = 0$ & XIb* & $1-X'$ & $n_\tau+1$ \\ 
\bottomrule
\end{tabular}
\end{table}
\nid
Here we use the notation for types of $\pi$ in loc. cit, and $N_\pi$ indicates the integer $\log_{|X|} |\ep(s, \phi_\pi,\psi)|$ for an additive character $\psi$ such that $\psi(\mathfrak{o}) = \{1 \} \neq \psi(\p^{-1})$.
All these types are unitary nongeneric representation.
Non-supercuspidal representation among them is only VIb, which is denoted by $\tau(T,|*|^{-1/2})$ in \cite{R-S}, and is the unique nongeneric constituent of $I(\1_{GL(2)})$ (c.f. (2.11) of loc. cit.). 
In this subsection, when $F$ is even residual, we assume that 
\begin{align}
\gamma(s, SK(\tau^{JL}),\psi) = \gamma(s, \tau^{JL},\psi)\gamma(s,\1_{D},\psi). \label{eq:gammaprod}
\end{align} 
By \cite{Sc-T}, $L^{reg}(s,\pi)$ is equal to $1$ when $\pi$ is type Va* or XIb*, and to $(1-X')^{-1}$ when $\pi$ is type VIb.
By (\ref{eq:divLL})
\begin{align}
\frac{L^{reg}(s,\pi)}{L(s,\pi)} = 1, \mbox{or} \ 1 \pm X', \mbox{or} \ 1-X'^2. \label{eq:LLreg}
\end{align} 
Now from (\ref{eq:gammaprod}) and the fact that $L(s,\tau^{JL})$ is $(1 \pm X')^{-1}$ or $1$, it follows that 
\begin{align}
\begin{split}
L(s,SK(\tau^{JL})) &= L(s,\tau)L(s,St), \\
\ep(s,SK(\tau^{JL}),\psi) &= \ep(s, \tau^{JL},\psi)\ep(s,\1_{D},\psi) 
\end{split}
\label{eq:gLep}
\end{align} 
for an arbitrary nonarchimedean local field $F$.
In particular, 
\begin{align*}
L(s,SK(\tau^{JL})) \neq L^{reg}(s,SK(\tau^{JL})).
\end{align*} 
Now the estimation (\ref{eq:m>n}) also holds by the proof of Lemma \ref{lem:m>nSK}.
If $\tau$ is not $St$, and $F$ is odd residual, (\ref{eq:gammaprod}) was proved by Danishman \cite{D}.
We give a proof (\ref{eq:gammaprod}) for the case where $\tau$ is $St$ in this subsection (c.f. Proposition \ref{prop:LepKlind}), and that for the case where $F$ is even residual in sect. \ref{sec:real}.
Now we will prove:
\begin{thm}\label{thm:SKJL}
With notations and the assumption as above, the followings are true.
\begin{enumerate}[i)]
\item If $\B_\s(\pi) \neq \{ 0\}$, then the strict minimal subspace of sign $-\ep_\tau$ is 
\begin{align*}
\begin{cases}
\B_{1,+}^{+} & \mbox{if $\tau = St$, and $E$ is in the case U},  \\
\B_{3,+}^{+} & \mbox{if $\tau = St$, and $E$ is in the case R},  \\
\B_{2}^{-\ep_\tau} & \mbox{if $\tau = \chi_E St$, and $E$ is in the case U, }  \\
\B_{2n_\tau -1}^{-\ep_\tau} & \mbox{if $\tau = \chi_E St$, and $E$ is in the case R, }  \\
\B_{{n_\tau},+}^{-\ep_\tau} & \mbox{if $\tau$ is supercuspidal with $2 \nmid n_\tau$, and $E$ is in the case U, }  \\
\B_{2n_\tau-1}^{-\ep_\tau} & \mbox{if $\tau$ is supercuspidal, and $E$ is in the case R.} 
\end{cases}
\end{align*} 
and one-dimensional space spanned by $\beta$ with the properties:
\begin{align}
Z_{\frac{(n_\tau+1)}{\f}}(s,\beta) \app L(s,\pi),  \label{eq:ZLSKJL}
\end{align} 
and 
\begin{align*}
Z(s,\beta) \app L^{reg}(s,\pi) =  
\begin{cases}
(1-X')^{-1} & \mbox{if $\tau = St$,} \\
1 & \mbox{otherwise.}
\end{cases}
\end{align*}
\item There is no paramodular form of sign $\ep_\tau$ in $\B_\s(\pi)$.
\item If $n_\tau$ is even and $E_\s$ is in the case U, then $\B_\s(\pi) = \{0 \}$.
\end{enumerate}
\end{thm}
The representations Va*, XIb* are supercuspidal nongenric and satisfy the condition (\ddag) in the previous subsection.
Since XIb*, and Va* in the case R have analytic conductor $\ge 3$ by Table \ref{tab:SKJL} and (\ref{eq:gLep}), our proofs for them are similar to that of Theorem \ref{thm:SKrepn}, and omitted.
We will give proofs of i), ii) for Va* in the case U, and VIb.
Since $n_\tau = 1$, those of iii) are needless. \\
\paragraph{\bf Proof for Va*)}
It is easy to derive from Table \ref{tab:ZU} that $m := m_\pi^{\ep_\pi}$ (resp. $m' := m_\pi^{-\ep_\pi}$) is less than $2$ if it is finite, i.e., there exists a paramodular form of sign $\ep_\pi$ (resp. $-\ep_\pi$).
Assume that $m = 0$.
Then the strict minimal space $V$ of sign $\ep_\pi$ is $\B_{1}$, $\B_1^\sharp$.
If we assume that $V$ is $\B_{1}$ or $\B_1^\sharp$, then we may assume that a strict newform $\beta$ lies in $\B_{1,\kappa}$ or $\B_{1,\kappa}^\sharp$, and it follows from Lemma \ref{lem:PSzeta}, \ref{lem:sdzp} that 
\begin{align*}
Z(s,\beta)/(1- \kappa X') \app (1- X'^2)^{-1}.
\end{align*} 
This conflicts to the supercuspidality of $\pi$. 
Hence $m = 1$.
Now it is easy to derive $V = \B_2$ from the same table, and the one-dimensionality of $V$ follows from Theorem \ref{thm:paravn}.
 This proves i).
For ii), similarly, we conclude $m' \neq 0$.
We see that no paramodular space of principal level $1$ does not the strict minimal space of sign $-\ep_\pi$ using the same table.\\

\paragraph{\bf Proof for VIb)}
The $\theta$-lift from the trivial representation of $D^\times \times D^\times$ is just $\pi$ (the big theta $\Th$ is irreducible and coincides with the small theta $\theta$).
We will really construct a paramodular form by the $\theta$-lift.
Let $\xi$ be a nonzero constant function on $H = D^\times \times D^\times/\{(z, z^{-1}) \mid z \in F^\times \}$.
Let $\OO$ denote the maximal order of $D$, and $f$ be the characteristic function of $\OO \op \OO$.
By using the formulas (\ref{eq:Weil}), one can see that the Bessel function $\xi_f$ defined at (\ref{eq:xivp}) is invariant under the Hecke subgroup $\Gamma_0(\p)$ and satisfies 
\begin{align}
\pi(w_0')\xi_f = \xi_f,\ \pi(\imath)\xi_f = \xi_f. \label{eq:B0++pre}
\end{align} 
Further, we have
\begin{align}
Z(s,\xi_f) \app \frac{1}{1-X'} \label{eq:ZsxifSKJL}
\end{align} 
since $f$ is invariant under the isometry subgroup $H^1 \subset H$, and 
\begin{align*}
\xi_f(\hv^j) & = q^{-2j}\int_{E^1 \bs H^1} f(z_0 h r^j) \xi(h r^j) dh \\
&= q^{-2j}\int_{E^1 \bs H^1} f(z_0 r^j) dh \\
&= q^{-2j}\vol(E^1 \bs H^1)
\end{align*}  
by definition, where $r \in D^\times$ indicates an element of reduced norm $\varpi$.
We will prove i), ii) according to the case of $E$. \\
{\bf Case U).}
In the case U-ii), by (\ref{eq:gLep}), it holds that
\begin{align}
n_\pi = 2, \ \ \ep_\pi =1, \ \ L(s,\pi) = \frac{1}{(1-X')^{2}}. \label{eq:SKJLnep}
\end{align} 
This holds also in the case U-i).
Indeed, for the above $\xi_f \in \B_{1,+}^+$, it holds that 
\begin{align*}
Z_1(s,\xi_f) \app (1-X')^{-2}
\end{align*} 
by (\ref{eq:B0++pre}), (\ref{eq:ZsxifSKJL}) and Lemma \ref{lem:PSzeta} iv).
By (\ref{eq:LLreg}), $L(s,\pi)^{-1}$ is equal to $(1-X'), (1-X')^{2}$, or $(1-X')^{2}(1+X')$.
Considering that $P_1(X,\xi_f)$ is a nonzero polynomial in $X$ of a sign, we can conclude (\ref{eq:SKJLnep}) from Lemma \ref{lem:sdzp}.

We will claim
\begin{align*}
\B_{1,-}^{\pm} = \B_{1,-}^{\sharp, \pm} = \{ 0\}.
\end{align*}
For $\beta \in \B_{1,-}^{\sharp,\pm}$, assume that $P_1(X,\beta) \neq 0$.
Then $P_1(X,\beta) \app 1$ by (\ref{eq:SKJLnep}) and Lemma \ref{lem:sdzp}.
But, by Lemma \ref{lem:PSzeta}, 
\begin{align*} 
P_1(X,\beta) \app g_\beta(X) \frac{1-X'}{1+ X'}.
\end{align*}
This is a contradiction.
Hence the claim.

Now we have showed $\B_1^+$ is the strict minimal space of sign plus and spanned by $\xi_f$ in the case U-i).
When $E$ is in the case U-ii), similar to Lemma \ref{lem:unrconst}, we can construct a desired strict newform of sign plus by the integral: 
\begin{align*}
\vol(K_1)^{-1} \int_{K_1} \pi(k) \xi_f d k = \vol(A_{K_1})^{-1} \vol(N_{K_1})^{-1} \int_{N_{K_1}} \pi(an) \xi_f d a d n.
\end{align*} 
This completes the proof for i).

For ii), by the argument for i), we find $m_\pi^{-} >1$.
Using Table \ref{tab:ZU}, one can show there is no paramodular form of sign minus.\\
{\bf Case R).}
If $E$ is in the case R-ii), (\ref{eq:SKJLnep}) holds by (\ref{eq:gLep}).
We can show this holds also in the case R-i) as follows.
By table A. 10 of \cite{R-S}, $\B_2 = \B_3^{\flat} = \{ 0\}$.
Put 
\begin{align*}
\beta = \int_{\mathfrak{P}^{-1}} \pi(n_x) \xi_f dx,
\end{align*} 
which belongs to $\B_3^{+}$, and not zero.
Thus $\B_3^{+}$ is the strict minimal space of sign plus.
We will compute $Z_2(s,\beta)$.
Let $u \in F^\times$.
On the one hand, $\beta(\hat{u}) = \xi_{f'}(\hat{u})$ where $f'(x) = \vol(\mathfrak{P}^{-1}) \Ch(x; \OO \op \varrho \OO)$.
It is easy to see that 
\begin{align*}
Z(s,\beta)= \frac{\vol(\mathfrak{P}^{-1})}{1-X'}.
\end{align*} 
On the other hand, $\beta^*(\hat{u}) = \xi_{f'}(\hat{u}w_1)$.
We calculate 
\begin{align*}
w_\psi(w_1,1)f'(x) = \vol(\mathfrak{P}^{-1})\Ch(x;\varrho^{-1}\OO \op \OO).
\end{align*} 
Now it is easy to see that
\begin{align*}
Z(s,\beta^*)= \frac{\vol(\mathfrak{P}^{-1})}{X(1-X')}.
\end{align*} 
It follows from (\ref{eq:PSzeta}) that $Z_2(s,\beta) \app (1-X')^{-2}$.
It is easy to deduce (\ref{eq:SKJLnep}) from (\ref{eq:divLL}) and (\ref{eq:LregSK}).

Now by Lemma \ref{lem:sdzp}, $P_1(X,\beta')$ and $P_2(X,\beta')$ for arbitrary $\beta' \in \B_3$ are constant multiples of $1+X^{-1}$ and $1$ respectively.
Taking (\ref{eq:PSzeta}) into account, we conclude that $Z(s,\beta')$ and $Z(s,\beta'^*)$ are constant multiples of $(1-X')^{-1}$ and $(X(1-X'))^{-1}$ respectively.
The one-dimensionality of $\B_3$ follows from Theorem \ref{thm:paravn} immediately.
The proof for ii) is similar to Va* and omitted.

We have also proved (without assuming (\ref{eq:gammaprod})):
\begin{prop}\label{prop:LepKlind}
For the representation $\tau(T,|*|^{-1/2})$, the $L$- and $\ep$-factors coincide with those of the Langlands-parameter unless $E$ is in the case R-ii).
\end{prop}
\subsection{Conclusion}\label{subsec:concl}
Let $\pi$ be in a SK-packet.
We have showed $m_\pi^{-\ep_\pi} = \infty$.
It holds for a paramodular form $\beta'$ of principal level $n$ that 
\begin{align*}
Z_{n+1}(s, e_n \beta') \app Z_{n+1}(s, \beta')
\end{align*}
by (\ref{eqn:e_m}) and (\ref{eq:PSzeta}).
Hence 
\begin{align*}
M_\pi = m_\pi^{\ep_\pi}
\end{align*}
by (\ref{eq:mmM}), and $e_m \beta$ for the strict newform $\beta$ is a newform of $\pi$of sign $\ep_\pi$ with the property 
\begin{align*}
Z_{m+1}(s,e_m\beta) \app L(s,\pi) =L(s,\phi_\pi)
\end{align*} 
where $m = M_\pi$.
Here we assume (\ref{eq:gammaprod}) if $F$ is even residual.
By Corollary \ref{cor:Z} and Lemma \ref{lem:m>n}, $\B_{2M_\pi}^{\ep_\pi}$ is one-dimensional.
Combining the theorems in the previous subsections, we have 
\begin{thm}\label{thm:SKNF}
Let $\pi$ be in the SK-packet of $\tau$.
Let $\s$ be a Hankel matrix in the form of (\ref{eq:defsigma}).
Then, the minimal level is 
\begin{align*}
M_\pi = \f^{-1} \times  
\begin{cases}
n_\tau & \mbox{if $\pi$ is $SK(\tau)$,} \\
(n_\tau+1) & \mbox{if $\pi$ is $SK(\tau^{JL})$.}
\end{cases}
\end{align*}
and the minimal space is the one-dimensional spanned by $\beta$ of sign 
\begin{align*}
\ep_\pi = 
\begin{cases}
\ep_\tau & \mbox{if $\pi = SK(\tau)$}, \\
-\ep_\tau & \mbox{if $\pi = SK(\tau^{JL})$} 
\end{cases}
\end{align*} 
such that 
\begin{align*} 
Z_{\frac{M_\pi}{\f}}(s,\beta) = L(s,\pi).
\end{align*}
\end{thm}
\nid
We have also showed that there is no paramoudular form of sign $-\ep_\pi$, and that $\B_{\s}(\pi) = \{ 0 \}$ when $E_\s$ is in the case U, and $n_\tau$ is odd (resp. even) if $\pi$ is $SK(\tau)$ (resp. $SK(\tau^{JL})$).
The latter result is consistent to the fact due to Waldspurger \cite{W}, Tunnel \cite{T}:
\begin{align*}
\dim \Hom_{E_\s^\times}(\tau, \1)+ \dim \Hom_{E_\s^\times}(\tau^{JL}, \1) = 1.
\end{align*} 
\section{Local oldforms}\label{sec:old}
Let $\tau \in \Ir'(\PGL_2(F))$, and $\pi$ be in the SK-packet of $\tau$.
The following local oldform theory is due to Roberts and Schmidt \cite{R-S}.
\begin{thm}\label{thm:oldsp}
Let $\pi = SK(\tau)$.
The dimension of paramodular vectors in $\pi$ of level $n_\tau +k$ is $[k/2] + 1$.
\end{thm}
\nid
Of course, this theorem holds also for split Bessel models of $SK(\tau)$.
The oldvectors are constructed from the newvector by two kinds of level-raising operators denoted by $\eta, \theta'$ (see 5.5 of loc. cit.).
In particular, $\eta$ and $\theta'$ are injective.

We treat the nonsplit case for $\pi$.
Let $\s$ be a Hankel matrix in the form of (\ref{eq:defsigma}) such that $\B_\s(\pi) \neq \{0\}$.
Let $\beta \in \B_\s(\pi)$ be a complete paramodular form of level $m$.
Let $e_m$ and $\eta$ be the level raising operators defined at (\ref{eq:emidm}) and (\ref{eq:defeta}) respectively.
Then
\begin{align*}
Z(s,e_m\beta) = q^\f(1 + X^\f) Z(s,\beta),\ Z(s, \eta \beta) = (qX)^\f Z(s,\beta).
\end{align*} 
The latter is obvious, and the former follows from (\ref{eqn:e_m}).
It follows from Lemma \ref{lem:PSzeta} iii) that 
\begin{align}
\frac{P_{m+1}(s,e_m \beta)}{ P_m(s,\beta)} = q^\f(1 + X^\f),\ \frac{P_{m+2}(s,\eta \beta)}{P_m(s,\beta)} = (qX)^\f . \label{eq:Polds}
\end{align} 
Therefore, for nonnegative integers $a,b,k$ such that $2a + b = k$, the complete paramodular forms 
\begin{align}
\eta^a e^b \beta^{new} \in \B_{2(M_\pi + k)} = \B_{2(M_\pi + k)}^{\ep_\pi} \label{eq:olds}
\end{align}
are linearly independent, where $\beta^{new}$ indicates a newform of $\pi$, and $e^b$ means the identity mapping if $b = 0$, and $e_{M_\pi+ b-1} \circ \cdots \circ e_{M_\pi}$ otherwise.
\begin{thm}\label{thm:oldns}
Let $\pi$ be in a SK-packet.
Then the set of the paramodular forms (\ref{eq:olds}) is a basis of $\B_{2(M_\pi + k)}$.
In particular, $\dim \B_{2(M_\pi + k)} = \dim \B_{2(M_\pi + k)}^{\ep_\pi}  = [k/2]+ 1$.  
\end{thm}
\begin{proof}
In the case of $k= 0$, this is Theorem \ref{thm:SKNF}.
For $k >0$, consider the subspace $\C[X]_k^+ \subset \C[X]$ consisting of polynomials of sign plus with diameter $k$, which has a basis: $
(1+X)^k, X(1+X)^{k-2}, \ldots, X^{[k/2]}(1+X)^{k- 2[k/2]}$.
By Lemma \ref{lem:sdzp}, if a nontrivial $\beta$ lies in $\B_{2(M_\pi + k)}$, then $P_{2(M_\pi + k)}(s,\beta) \in \C[X^\f]_k^+$.
Now the assertion follows from Theorem \ref{thm:paravn} and (\ref{eq:Polds}).
\end{proof}
Now it is obvious that 
\begin{align}
e_m: \B_{2m} \hookrightarrow \B_{2m+2}, \ \ \eta:  \B_{2m} \hookrightarrow \B_{2m+4}. \label{eq:injeeta}
\end{align}  
If we denote also by $K_{2m}$ and $\B_{2m}$ the paramodular group of level $m$ in the sense of \cite{R-S} and the subspace of $\B_\s(SK(\tau))^{K_{2m}}$, respectively, where $\s = 1_2$, then the idempotent $e_{m} := e_{K_{2m+2}}$ of the Hecke algebra of $K_{2m+2}$ is just the level raising operator $\theta'$ (c.f. sect. 3 of loc. cit), and thus (\ref{eq:injeeta}) holds also in this case.
\section{Local functional equation (real case)}\label{sec:real}
Let $\A = \A_\Q$ be the adele of $\Q$.
Let $\Phi$ be an automorphic cuspform on $\GSp_4(\A)$.
The Fourier coefficient of $\Phi$ relevant to $\s \in H_2(\Q)$ and $\psi$ is defined by 
\begin{align*}
\Phi_\s(g) = \int_{H_2 \bs H_2(\A)} \psi_{\s}(-x)\Phi(n_x g) dx, \ g \in \GSp_4(\A).
\end{align*} 
For a continuous character $\Lambda$ of $T(\Q) \bs T(\A)$, the global Bessel period $\Phi_\s^\Lambda$ relevant to $\Lambda^\psi_\s$ is defined by 
\begin{align*}
\Phi_{\s}^\Lambda(g) = \int_{T(\Q) \bs T(\A)} \Lambda(t)^{-1}\Phi_\s(t g) d^\times t.
\end{align*} 

Let $\tau = \ot_v\tau_v$ be an irreducible cuspidal automorphic representation of $\PGL_2(\A)$.
Let $\pi = \ot_v \pi_v$ be an irreducible cuspidal automorphic representation of $\PGSp_4(\A)$ in its SK-packet.
Let $\Phi \in \pi$. 
By Theorem \ref{thm:1dimSK}, all $\Lambda_v$ for $v < \infty$ are trivial, and so is $\Lambda_\infty$, and we restrict ourselves to special Bessel models of $\pi_\infty$.
Here, we mean by Bessel models of $\pi_\infty$ an irreducible $\PGSp_4(\Rb)$-module equivalent to $\pi_\infty$ generated by a Bessel function.
We will denote by $\B_E(\pi_\infty)$ the special Bessel models relevant to $E^\times$, where $E$ is isomorphic to $\C$ or $\Rb \op \Rb$.
Note that any automorphic form $\Phi$ in a member $\pi$ of the SK-packet has a simple Fourier-Bessel expansion:
\begin{align}
\Phi(g) = \sum_\s \Phi_{\s}^{\1}(g) \label{eq:FBSK}
\end{align} 
where $\s$ runs all regular Hankel matrices (singular one does not appear since $\pi$ is nongeneric and cuspidal.).
Assume that $\tau_\infty$ is the holomorphic discrete series of minimal weight $2 \kappa (\ge 2)$, and $\pi_\infty = SK(\tau^{JL})$.
Then $\pi_\infty$ is the (limit of) holomorphic discrete series of minimal weight $(1+ \kappa,1+ \kappa)$.
By Koecher's principle, it has no split Bessel model.
By the work of Pitale and Schmidt \cite{P-Sc}, it has a special nonsplit Bessel model.
In this case, we may assume $E = \C$, and the relevant group $\G$ is $\{g \in \GL_2(\C) \mid \det(g) \in \Rb^\times \}$.
Let $K$ be the standard compact maximal subgroup of $G= \GSp_4(\Rb)$.
Then $K_\G := K \cap \G$ is a maximal compact subgroup of $\G$ isomorphic to $\mathbb{H}^\times \rt \Z/2\Z$, where $\mathbb{H}$ indicates Hamilton's quaternion algebra.
Now let $l$ be a positive integer, and $\rho_l$ be an $l$-dimensional irreducible representation of $\mathbb{H}^\times/\Rb^\times$.
Let $V_l$ and $\langle, \rangle$ denote the representation space and natural pairing of $\rho_l$ respectively.
Let $\vp =(\vp_i)_{i =1, \ldots, l} \in \Ss(\C^2) \ot V_l$.
Let $\beta =(\beta^i)_{i =1, \ldots, l} \in \B_\C(\pi_\infty) \ot V_l$ which transforms according to $\rho_l$.
Consider the zeta integral
\begin{align*} 
Z(s,\langle \beta, \vp\rangle) := \vol(K_\G)^{-1} \int_{N_\G \bs \G } \langle \beta(g), \vp(z_0 g) \rangle |\det(g)|^{s+1/2} dg.
\end{align*}
We may assume that $\vp$ also transforms according to $\rho_l$, since
\begin{align*}
Z(s,\langle \beta, \vp\rangle) = \vol(K_\G)^{-1} \int_{N_\G \bs \G } \int_{K_\G}\langle \beta(gk), \vp(z_0 gk) \rangle |\det(g k)|^{s+1/2} dk dg.
\end{align*} 
Then we compute, by using the Iwasawa decomposition of $\G$,
\begin{align*}
 \vol(K_\G)^{-1} Z(s,\beta,\vp) &= \int_{N_\G \bs \G/K_\G} \langle \beta(g), \vp(z_0 g) \rangle |\det(g)|^{s+1/2} dg \\
& = \int_{\Rb_{>0}} \int_{\C^\times} \langle \beta(\hat{u}a_t ), \vp([0,t^c])\rangle |t|^{2s +1}|u|^{s-3/2} d^\times t d^\times u \\
&= \int_{\Rb_{>0}} \int_{\C^\times} \langle \beta(\hat{u}), \vp([0,t^c])\rangle |t|^{2s +1}|u|^{s-3/2} d^\times t d^\times u \\
&= \sum_{i =1}^l Z(s,\beta^i) \int_{\C^\times} \ol{\vp_i([0,t^c])} |t|^{2s +1} d^\times t.
\end{align*} 
We find that each integral in the last sum is $\Gamma(s+ 1/2)$ times a holomorphic function if $l=1$( i.e., $\rho_l$ is trivial), and zero otherwise.
To compute $Z(s,\beta^i)$, we use Yoshida's construction of $\xi_f \in \B_\C(\pi_\infty)$ (\cite{Y}).
Here $\xi$ is a matrix coefficient of $\rho_{\kappa-1}$, and $f$ is a Schwartz function on $\mathbb{H} \op \mathbb{H}$ in the form of 
\begin{align}
P(x_1^*x_2)\exp(-a|x_1|^2 -b Tr(x_1^*x_2) -c |x_2|^2), \ x_i \in \mathbb{H} \label{eq:Yf}
\end{align}
where $P(x)$ is a homogeneous polynomial of degree $\kappa-1$ on the trace zero part of $\mathbb{H}$, and $a,b,c$ are some real numbers. 
Since $\pi_\infty$ is irreducible, and $G = \G P \cup \G s P$ by Lemma 5.4.2. of \cite{Sc-T}, we may assume that each $\beta^i$ is given by some $\xi_{f'}$ with $f'$ being a translation of the above $f$ by an element of $P \cup sP$.
However, by the formula (\ref{eq:Weil}) and the argument in p. 200-202 of \cite{Y}, such an $f'$ is still in the form of (\ref{eq:Yf}).
It is easy to see that each $Z(s,\beta^i)$ is $\Gamma(s+\kappa-1/2)$ times a holomorphic function.
In particular, it is possible to construct $\beta_0$ such that
\begin{align*}
\beta_0(\hat{u}) = 
\begin{cases}
u^{1+\kappa} \exp(-2 \pi u) & \mbox{if $u >0$}, \\
0 & \mbox{if $u < 0$}. 
\end{cases}
\end{align*} 
This is just the special Bessel function described in  Theorem 3.4. of \cite{P-Sc}.
If we set $\vp_0(z) := 2\exp(- 2\pi (|z_1|^2 +|z_2|^2))$, then $Z(s,\beta_0,\vp_0)$ equals 
\begin{align*}
(2\pi)^{-2s-\kappa}\Gamma(s+1/2)\Gamma(s+\kappa-1/2),
\end{align*} 
the $L$-function $L(s,\phi_{\pi_\infty})$.
We can prove: 
\begin{thm}\label{thm:fereal}
Let $\pi_\infty$ be the (limit of) holomorphic discrete series representation of $\PGSp_4(\Rb)$ with minimal weight $(\kappa+1,\kappa+1)$.
Let $\beta \in \B_\C(\pi_\infty)$ and $\vp \in \Ss(\C^2)$.
Then the ratio $Z(s,\beta,\vp)/L(s,\phi_{\pi_\infty})$ is absolutely convergent when $\Re(s) >>0$, and extends to an entire function of $s$.
Further, we have the local functional equation:
\begin{align*}
\frac{Z(1-s,\beta^\imath,\vp^\sharp)}{L(1-s,\phi_{\pi_\infty})} = (-1)^{\kappa+1}|a|^{4s-2}\frac{Z(s,\beta,\vp)}{L(s,\phi_{\pi_\infty})}
\end{align*} 
where $\psi_\infty(x) = \exp(2 \pi \sqrt{-1} a x)$.
If $\beta =\beta_0$ and $\vp = \vp_0$ as above, then the ratio in the RHS is just one.
\end{thm}
\begin{proof}
Since both $\beta$ and $\vp$ are $K_\G$-finite, they are matrix coefficients of some finite-dimensional representations of $K_\G$, and the analicity of the ratio follows from the above argument.
For the functional equation, we consider a SK-packet of $\tau$, where $\tau = \ot_v \tau_v$ is an irreducible cuspidal automorphic representation of $\PGL_2(\A)$ such that 
\begin{itemize}
\item $\tau_\infty$ is holomorphic discrete series of minimal weight $2\kappa (\ge 2)$,
\item $\tau_2$ is a principal series,
\item $\tau_p$ is a discrete series for some odd $p$.
\end{itemize}
Such a representation exists by the lemma below.
Let $\ep$ be the root number of $\tau$.
Applying to $\tau$ the main lifting theorem of \cite{Sc}, the global representation $\pi = \ot_v \pi_v$ in the SK-packet settled as follows is an irreducible cuspidal automorphic representation. 
\begin{itemize}
\item $\pi_\infty$ is the (limit of) holomorphic discrete series of minimal weight $\kappa+1$, 
\item $\pi_p$ is $SK(\tau_p)$ if $\ep = -1$, and $SK(\tau_p^{JL})$ if $\ep = 1$, 
\item $\pi_v$ is $SK(\tau_v)$ for nonarchimedean $v \neq p$.
\end{itemize}
Let $\Phi \in \pi$, which has a (nontrivial) special Bessel period $\Phi_\s = \Phi_\s^\1$ for some $\s$ by (\ref{eq:FBSK}).
By Proposition 7 of \cite{Q}, $\Phi_\s$ is given by the $\theta$-lift from the mataplectic group ${\rm Mp}_2(\A)$ of rank $1$: 
\begin{align*}
\Phi_\s(h) = \int_{N_2(\A) \bs SL_2(\A)} W(g)w(g, h) f(t_0) dg, \ h \in \PGSp_4(\A).
\end{align*}
Here $W$ is a Whittaker period of a Shimura-Waldspurger transfer of $\tau$, $t_0$ is a point of a five-dimensional space $U(\Q)$, $w = \ot_v w_v$ is a Weil representation of ${\rm Mp}_2(\A) \times \PGSp_4(\A)$ realized in the space $\Ss(U(\A))$, and $f = \ot_v f_v \in \Ss(U(\A))$.
Since $W = \prod_v W_v$, and $f$ is a linear combination of decomposable Schwartz functions, we may assume that $\Phi_\s = \prod_v \Phi_{\s,v}$. 
Since $\pi_\infty$ is holomorphic and irreducible, the algebra $E_\s$ is a CM-field, and we may assume $\beta = \Phi_{\s,\infty}$.
Since $\pi_2 = SK(\tau_2)$, we may assume $\Phi_{\s, v}$ is the local newform at all nonarchimedean $v$.
Let $\vp_v$ at $v < \infty$ be the Schwartz functions corresponding to the local newforms, and set $\vp = \prod_v \vp_v$.
By Theorem 5.1. of \cite{PS2} and the cuspidality of $\Phi$, 
\begin{align*}
Z(s,\Phi_\s,\vp) = Z(1-s,\Phi_\s^\imath,\vp^\sharp).
\end{align*} 
Now our local functional equation follows from the factorization $Z(s,\Phi_\s,\vp) = \prod_v Z(s,\Phi_{\s,v},\vp_v)$, and the global functional equation (\ref{eq:GlFE}). 
\end{proof}
\begin{lem}
Let $\tau_\infty$ be the holomorphic discrete series representation of $\PGL_2(\Rb)$ with minimal weight $2\kappa (\ge 2)$.
Let $p_1$ be a prime and $S$ be a finite set of primes $\neq p_1$.
Then, there exists an irreducible cuspidal automorphic representation $\tau = \ot_v \tau_v$ of $\PGL_2(\A_\Q)$ such that 
\begin{itemize}
\item $\tau_{p_1}$ is an unramified representation,
\item $\tau_p$ is a discrete series for all $p \in S$.
\end{itemize}
\end{lem}
\begin{proof}
There exists a definite quaternion algebra $D$ defined over $\Q$ such that $D(\Q_{p_1}) \simeq M_2(\Q_{p_1})$ and $D(\Q_p) \not\simeq M_2(\Q_p)$ for all $p \in S$. 
There exists a sufficiently small order $\OO \subset D(\Q)$ such that $\OO_{p_1} \simeq M_2(\Z_{p_1}), \OO_v^\times \supset \Z_v^\times, v < \infty$, and 
\begin{align*}
1< |D^1(\Q) \bs D^1(\A)/\OO^1(\A)|.
\end{align*}
Here $D^1$ indicates the group consisting elements of reduced norm $1$, and $\OO^1 = \OO \cap D^1$.
Let $\widehat{\OO}$ denote the finite part of the adelization of $\OO$, i.e., 
\begin{align*}
\widehat{\OO} = \OO \ot \prod_p \Z_p.
\end{align*}
Since automorphic forms on $PD(\A)^\times$ are square integrable, one can find by the last property that there exists an $\widehat{\OO}^\times$-invariant automorphic form $\xi$ such that 
\begin{itemize}
\item $\xi$ is a matrix coefficient of $\rho_{2\kappa-2}$.
\item $\xi$ is a Hecke eigenform  for almost all $v$ such that 
\begin{align}
\int_{D^1(\Q) \bs D^1(\A)} \xi(h) dh = 0. \label{eq:intxi0}
\end{align} 
(this integral always vanishes when $\kappa >1$).
\end{itemize}
Hence the $D(\A)^\times$-module generated by $\xi$ is an irreducible automorphic representation $\tau' = \ot_v \tau_v'$ of $PD(\A)^\times$ such that $\tau_{p_1}'$ unramified and $\tau_\infty = \rho_{\kappa-1}$.
One can show that the Jacquet-Langlands lift (a $\theta$-lift) of $\tau'$ is cuspidal by (\ref{eq:intxi0}), and is the desired representation.
\end{proof}
Now, we can prove (\ref{eq:gammaprod}).
By Lemma 5.7. of \cite{Sc}, there is a totally real number field $\Fb$ such that $F$ is isomorphic to a completion $\Fb_v$ for a dyadic nonarchimedean place $v$.
If $\tau$ is an irreducible cuspidal automorphic representation of $\PGL_2(\A_\Fb)$ and a cuspidal member $\pi = \Pi(\tau \bt \pi_S)$ in the SK-packet of $\tau$ has all archimedean components $\pi_v$ holomorphic discrete series, then it holds that 
\begin{align*}
\prod_{j = 1}^{[\Fb:\Q]} \gamma(s,\phi_{\pi_{\infty_j}}, \psi_{\infty_j}) \prod_{w< \infty} \gamma(s,\pi_w, \psi_w) =1
\end{align*}
by (\ref{eq:GlFE}).
Therefore, by the main lifting theorem \cite{Sc} and Theorem \ref{thm:fereal}, for the proof of (\ref{eq:gammaprod}), it suffices to show that an arbitrary discrete $\tau_v \in \Ir(\PGL_2(\Fb_v))$ is embeddable to an irreducible cuspidal automorphic representation $\tau'$ of $\PGL_2(\A_\Fb)$ such that 
\begin{itemize}
\item all archimedean components of $\tau$ are holomorphic discrete series, 
\item $\tau_w'$ is principal series for all other dyadic nonarchimedean place $w$,
\item $\tau_w'$ is discrete for a nondyadic nonarchimedean place $w$.
\end{itemize}
Such a $\tau$ can be also given by the Jacquet-Langlands lift since we can take a convenient totally definite quaternion algebra $D$ defined over $\Fb$ and automorphic representation of $PD(\A_\Fb)^\times$ as follows. 
\begin{lem} 
Let $\Fb$ be a totally real number field, and $D$ be a totally definite quaternion algebra defined over $\Fb$.
Let $\kappa_1, \ldots ,\kappa_{[\Fb:\Q]}$ be nonnegative integers.
Let $v_1$ be a nonarchimedean place at which $D$ does not split.
Let $v_2$ be another nonarchimedean place.
Let $\tau_{v_1} \in \Ir(PD_{v_1}^\times)$.
Then there exists an irreducible automorphic representation $\tau' = \ot_v \tau_v'$ of $PD(\A)^\times$ such that 
\begin{itemize}
\item $\tau_{\infty_j}'$ is equivalent to $\rho_{2\kappa_j}$ for each archimedean place $\infty_j$,
\item $\tau_{v_1}'$ is equivalent to $\tau_{v_1}$,
\item $\tau_{v}'$ are unramified for all nonarchimedean places $v \neq v_1, v_2$ at which $D_v$ splits. 
\end{itemize}
\end{lem}
\begin{proof}
Since $\tau_{v_1}$ is smooth and finite dimensional, there is an order $\OO$ such that $\tau_{v_1}$ is invariant under $\OO_{v_1}^\times$, and $\OO_v$ are maximal for $v \neq v_1$.
A matrix coefficient $f$ of $\tau_{v_1}$ is determined by its values at finitely many points $g_1, \ldots, g_l \in \Fb_{v_1}^\times \bs D_{v_1}^\times/\OO_{v_1}^\times$.
Further we can take $\OO_{v_2}$ sufficiently small so that $\tau_{v_1}$ and $\rho_\infty := \rho_{2\kappa_1} \bt \cdots \bt \rho_{2\kappa_{[\Fb:\Q]}}$ are invariant under $D(\Fb)^\times \cap (\cap_{i=1}^l \widehat{\OO}^\times \langle g_i \rangle)$, and 
\begin{align*}
\A^\times D(\Fb)^\times g_{i} \OO(\A)^\times \cap \A^\times D(\Fb)^\times g_{j} \OO(\A)^\times = \emptyset, \ \ \mbox{for $i \neq j$}.
\end{align*}
For a matrix coefficient $\xi_\infty$ of $\rho_\infty$, we can extend $f$ to an automorphic form $\xi$ on $D(\A)^\times$ invariant under $\A^\times \widehat{\OO}^\times$ by setting 
\begin{align*}
\xi(g) = 
\begin{cases}
0 & \mbox{if $g \not\in \sqcup_{i = 1}^l \A^\times D(\Fb)^\times g_{i} \OO(\A)^\times$,}\\
f(g_i) \xi_\infty(k_\infty) & \mbox{if $g = z d g_i k$ with $z \in \A^\times, d \in D(\Fb)^\times, k \in \OO(\A)^\times$.} 
\end{cases}
\end{align*}
For $g \in D(\A)^\times$, the function $\xi(gh)$ of $h$ in $D_{v_1}^\times$ (resp. $(\mathbb{H}^\times)^{[\Fb:\Q]}$) is a matrix coefficient of $\tau_{v_1}$ (resp. $\rho_\infty$).
This implies the assertion.
\end{proof}
\section{Siegel modular forms}\label{sec:Siegel}
In this section, to describe our result in classical terms, we change the definition of the group $\GSp_4$.
We replace the defining matrix $J$ in (\ref{eq:defJ}) with
\begin{align*}
\begin{bmatrix}
 & -1_2\\
1_2 &
\end{bmatrix},
\end{align*} 
which is the conjugate of $J$ by the element 
\begin{align*}
\begin{bmatrix}
1 & & & \\
 &1 & & \\
 & & &1 \\
 & & 1 &
\end{bmatrix}.
\end{align*} 
Accordingly, we denote by $K_{E_\s}(2m) = K_E(2m)$ the conjugate of the nonsplit paramodular groups $K_{2m}$ (relevant to $\s$ in the form of (\ref{eq:defsigma})) by the same element.
We also denote by $K_{E_\s}(m) = K_E(m)$ the conjugate of 
\begin{align}
K(m)\langle \begin{bmatrix}
1&-1 & &  \\
1& 1& & \\
& & 1&1  \\
& & -1&1  \\
\end{bmatrix} \rangle, \label{eq:defparasp'}
\end{align} 
where $K(m)$ indicates the original paramodular group of level $\p^m$, and 
\begin{align*}
\s = \begin{bmatrix}
& -1 \\
1 &
\end{bmatrix}.
\end{align*} 
Observe that the subgroups (\ref{eq:defparasp'}) provide a newform theory for Bessel vectors relevant to this Henkel matrix (c.f. sect. \ref{sec:NFsp}).
Let $E$ be a $2$-dimensional semisimple algebra over $\Q$, and $\mathfrak{e}_p$ indicate the ramification index of $E_p/\Q_p$.
We denote by $\K_E(n)$ the $\Q$-rational points in  
\begin{align*}
\prod_p K_{E_p}(\mathfrak{e}_p ord_p(n)),
\end{align*} 
where $n$ is a positive integer such that $ord_p(n)$ is even if $E_p$ is an unramified quadratic field extension of $\Q_p$. Since the class number of $\Q$ is one, by virtues of the strong approximation theorem, Siegel modular forms relevant to an arithmetic subgroup $\K$ are interpreted as automorphic forms on $\GSp_4(\A)$ invariant under $\prod_p \K_p$, if $\K_p$ contains all elements
\begin{align*}
\begin{bmatrix}
u v& & & \\
 &uv & & \\
& & v& \\
& & &v 
\end{bmatrix}, \ \ 
u,v  \in \Z_p^\times
\end{align*} 
for all primes $p$ (c.f. sect. 3 of \cite{Y}).
The arithmetic subgroup $\K_E(n)$ satisfies this condition.

Let $\tau = \ot_v \tau_v$ be an irreducible cuspidal automorphic representation of $\PGL_2(\A)$ with $\tau_\infty$ holomorphic discrete series $2\kappa \ge 2$.
Let $S_\tau$ be the set of all primes $p$ at which $\tau_p$ is discrete.
If $S$ is a subset (possibly empty) of $S_\tau$ such that $(-1)^{|S|} = -\ep(1/2,\tau)$, then we denote $\Pi_S = \Pi(\tau \bt \pi_{S \cup \{\infty\}})$, the cuspidal member of the SK-packet of $\tau$.
For such an $S$, let $E$ be an imaginary quadratic field such that $\Pi_{S,p}$ has the special Bessel model relevant to $E_p$ at all primes $p$.
Equivalently, 
\begin{align*}
\ep(1/2,\tau_p)\ep(1/2,\tau_p \ot \chi_{E,p}) = 
\begin{cases}
\chi_{E,p}(-1) & \mbox{if $p \in S$}, \\
-\chi_{E,p}(-1) & \mbox{if $p \in S_\tau \setminus S$} 
\end{cases} 
\end{align*} 
(c.f. Theorem 2. of \cite{W2} and Corollary 4.7.1. of \cite{R-S2}).
Here $\chi_{E,p}$ indicates the quadratic character of $\Q_p^\times$ associated to the extension $E_p/\Q_p$.
We say $E$ {\it matches to $S$} in this case. 
Further, by Theorem 2. of \cite{Q}, there exists an automorphic form $\Phi$ in the packet of $\tau$ with nontrivial global special Bessel period relevant to $E$ if and only if $L(1/2,\tau \ot \chi_E) \neq 0$.
Observe that there are possibly infinitely many such $E$, but there are only finitely many subgroups $\K_E(n)$ for fixed $S$ and $n$, and that if $E$ and $E'$ match different $S$ and $S'$ respectively, then $\K_E(n)$ is not isomorphic to $\K_{E'}(n)$. 
Combining these results with the local newform theory, we have a concise version of Theorem 4.3.16. of \cite{A} for SK-packets:
\begin{thm}\label{thm:maincls}
With notation and assumptions as above, let $E$ be an imaginary quadratic field matching to $S (\subset S_\tau)$.
Let $N_\tau \in \mathbb{N}$ denote the level of $\tau$.
In the space of Siegel modular forms of weight $\kappa +1$ with respect to 
\begin{align*} 
\K_E(N_\tau \prod_{p \in S} p), 
\end{align*}
the cuspidal member $\Pi_S$ of the SK-packet of $\tau$ has an unique Siegel cuspform $\Phi_S^{E}(Z)$ up to scalars with the Fourier expansion 
\begin{align*}
\sum_{\s' \in Sym_2^+(\Q)} \mathscr{F}_S^{E}(\s')\exp(2 \pi \sqrt{-1} Tr(\s' Z))
\end{align*} 
for elements $Z$ in the Siegel upper half space of degree $2$.
Here $Sym_2^+(\Q)$ indicates the $2 \times 2$ symmetric matrices with positive determinants.
If $L(1/2,\tau \ot \chi_E) \neq 0$ for $E = \Q(\sqrt{-d})$ with $d$ squarefree, then for $\s' = diag(d,1)$, we have a formal identity
\begin{align*}
\sum_{n=1}^\infty \frac{\mathscr{F}_S^{E}(n \s')}{n^{s+\kappa-\frac{1}{2}}} &=\mathscr{F}_S^{E}(\s')\prod_p L(s+\frac{1}{2},\chi_{E,p}) L(s,\tau_p) \cdot \prod_{p \not\in S} (1-p^{-s+\frac{1}{2}})^{-1}.
\end{align*}
\end{thm}
In the SK-packet of $\tau$, if $\ep(1/2,\tau) = -1$ (and hence $L(1/2,\tau) = 0$), then there is a Siegel cuspform $\Phi_{\emptyset}^E$.
In this case, if $E$ is chosen so that $E_p$ splits at all $p$ where $\tau_p$ is ramified, then the arithmetic subgroup is just the global paramodular group of level $N_\tau$, and $\Phi_{\emptyset}^E$ is possibly given by Gritsenko's lift \cite{G} (which need a conjugation).
If $\ep(1/2,\tau) = 1$ and $S_\tau = \emptyset$, then there is a no Siegel cuspform in the packet.
If $\ep(1/2,\tau) = 1$ and $S_\tau \neq \emptyset$, then there is a Siegel cuspform $\Phi_{\{p \}}^E$ for each prime $p \in S_\tau$, but no global paramodular Siegel cuspform in the packet.
In this case, if $E$ is chosen so that $E_{p'}$ splits at all $p' \neq p$ where $\tau_{p'}$ is ramified, then the $p'$-completions of the arithmetic subgroup are isomorphic to the local paramodular groups.
Since $SK(\tau_v^{JL})$ has no local paramodular vector (c.f. \cite{R-S}), there is no other Siegel cuspform in this packet with respect to the group $\K_E(N_\tau)$ (resp. $\K_E(N_\tau p))$ or those of `lower' levels, for $E$ chosen as above.
For this reason, we may call $\Phi_{\emptyset}^E$ (resp. $\Phi_{\{p \}}^E$) a `Siegel newform of the packet' in a sense.

For Hilbert-Siegel modular forms over a totally real field $\Fb$, we have a similar statement, but need a suitable conjugation of the arithmetic subgroup as for Hilbert modular forms, since any additive character $\psi$ on $\Fb \bs \A_\Fb$ does not takes $1$ on a local ring of integers in general.

\end{document}